\documentclass[11pt, reqno]{amsart}
\usepackage{amsmath,amsthm,amssymb,mathrsfs}
\usepackage[shortalphabetic,abbrev,non-sorted-cites]{amsrefs}
\usepackage{color}
\usepackage{latexsym}
\usepackage{datetime}
\usepackage{hyperref}

\theoremstyle{plain}
  \newtheorem{thm}{Theorem}[section]
  \newtheorem{lem}[thm]{Lemma}
  \newtheorem{prop}[thm]{Proposition}
  \newtheorem{cor}[thm]{Corollary}
  \newtheorem*{thm*}{Theorem}  
  \newtheorem*{thmm*}{Main Theorem} 
\theoremstyle{definition}
  \newtheorem{defn}[thm]{Definition}
  \newtheorem{rmk}[thm]{Remark}
  \newtheorem*{ack*}{Acknowledgement}
\theoremstyle{plain}

\numberwithin{equation}{section}

\newcommand\HD[1]{{\rm H}^{#1}}

\newcommand\pl{\partial}
\newcommand\dbr{\bar{\partial}}
\newcommand\dbrs{\bar{\partial}^*}

\newcommand\oh{\frac{1}{2}}
\newcommand\ot{\frac{1}{3}}
\newcommand\dd{{\mathrm d}}

\newcommand\ex{\tilde}
\newcommand\prp{\natural}
\newcommand\bo{{\bf 1}}

\newcommand\CV{\mathcal{V}}

\newcommand\CC{\mathcal{C}}
\newcommand\CS{\mathcal{S}}
\newcommand\CD{\mathcal{D}}
\newcommand\CO{\mathcal{O}}

\newcommand\CI{\mathcal{I}}

\newcommand\BC{\mathbb{C}}
\newcommand\BR{\mathbb{R}}
\newcommand\BN{\mathbb{N}}
\newcommand\BZ{\mathbb{Z}}
\newcommand\BS{\mathbb{S}}
\newcommand\BI{\mathbb{I}}

\newcommand\rb{{\bf r}}

\newcommand{\canspinex}{\underline{\BC}\oplus\ex{K}^{-1}}

\newcommand\aE{A_E}
\newcommand\aeE{A_{\ex{E}}}
\newcommand{\aer}{\ex{A}_r}
\newcommand{\ahr}{\hat{A}_r}

\newcommand\can{{o}}
\newcommand\appr{(0)}
\newcommand\err{({\varepsilon})}

\newcommand\fc{\mathfrak{c}}

\newcommand\fe{\mathfrak{e}}

\newcommand\fb{\mathfrak{b}}

\newcommand\fj{\mathbf{j}}
\newcommand\fk{\mathbf{k}}

\newcommand\fD{\mathfrak{D}}
\newcommand\fR{\mathfrak{R}}
\newcommand\fE{\mathfrak{E}}
\newcommand\fr{\mathfrak{r}}
\newcommand\fs{\mathfrak{s}}

\newcommand\ck{\check}

\DeclareMathOperator{\sfa}{\mathsf f}
\DeclareMathOperator{\cl}{cl}
\DeclareMathOperator{\pr}{pr}

\DeclareMathOperator{\spn}{span}
\DeclareMathOperator{\spec}{spec}

\DeclareMathOperator{\End}{End}

\DeclareMathOperator{\supp}{supp}

\begin{document}

\title[Dirac Spectral Flow on Contact $3$-Manifolds II]{Dirac Spectral Flow on Contact Three Manifolds II:\\
Thurston--Winkelnkemper Contact Forms}

\author[C.-J. Tsai]{Chung-Jun Tsai}
\address{Department of Mathematics \\ and National Center for Theoretical Sciences (Mathematics Division, Taipei Office)\\
National Taiwan University\\ Taipei 10617\\ Taiwan}
\email{cjtsai@ntu.edu.tw}

\date{\usdate{\today}}

\maketitle

\begin{abstract} Given an open book decomposition $(\Sigma,\tau)$ of a three manifold $Y$, Thurston and Winkelnkemper \cite{ref_TW} construct a specific contact form $a$ on $Y$.  Given a spin-c Dirac operator $\CD$ on $Y$, the contact form naturally associates a one parameter family of Dirac operators $\CD_r = \CD - \frac{ir}{2}\cl(a)$ for $r\geq0$.  When $r>\!>1$, we prove that the spectrum of $\CD_r = \CD_0 - \frac{ir}{2}\cl(a)$ within $[-\oh r^\oh,\oh r^\oh]$ are almost uniformly distributed.  With the result in Part I \cite{ref_Ts2}, it implies that the subleading order term of the spectral flow from $\CD_0$ to $\CD_r$ is of order $r (\log r)^{\frac{9}{2}}$.  Besides the interests of the spectral flow, the method of this paper provide a tool to analyze the Dirac operator on an open book decomposition.
\end{abstract}


\section{Introduction}
Suppose that $(Y,a)$ is a contact three manifold, and $\CD$ is a spin-c Dirac operator on $Y$.  It naturally associates a one parameter family of Dirac operators $\CD_r = \CD - \frac{ir}{2}\cl(a)$ for $r\geq0$.  The spectral flow from $\CD_0$ to $\CD_r$ appears in the study of contact geometry \cite{ref_Taubes_SW_Weinstein}.  In Part I \cite{ref_Ts2}, we analyze the eigensections of $\CD_r$ along the Reeb vector field and on the contact hyperplane, then relate the spectral flow to certain spectral asymmetry property of small eigenvalues of $\CD_r$.

This article focuses on the case when $Y$ is given by an open book decomposition and $a$ is the Thurston--Winkelnkemper contact form \cite{ref_TW}.  We introduce another Dirac operator $\ex{\CD}_r$ over a fibered three manifold.  The Dirac operator $\ex{\CD}_r$ captures the spectral properties of $\CD_r$, and the spectrum of $\ex{\CD}_r$ is easier to handle.  It can be used to study the spectral asymmetry and the spectral flow of $\CD_r$.  Besides the interests of the spectral flow, our construction provides a tool to analyze the Dirac equation on an open book decomposition.

\subsection{Dirac spectral flow}
Let $(Y,\dd s^2)$ be an oriented Riemannian three manifold.  A \emph{spin-c structure} on $Y$ consists of a rank $2$ Hermitian vector bundle $\BS$ and a bundle map $\cl:TY\to\End(\BS)$ such that:
\begin{itemize}
\item $\cl(v)^2 = - |v|^2$ for any $v\in T_yY$ ~;
\item if $|v|=1$, then $\cl(v)$ is unitary;
\item if $\{e_1,e_2,e_3\}$ is an oriented orthonormal frame for $T_yY$, then $\cl(e_1)\cl(e_2)\cl(e_3)$ is the identity endomorphism of $\BS|_y$.
\end{itemize}
The bundle $\BS$ is called the \emph{spinor bundle}, and the map $\cl$ is called the \emph{Clifford action}.

A \emph{spin-c connection} on $\BS$ is a Hermitian connection $\nabla:\CC^\infty(Y;\BS)\to\CC^\infty(Y;T^*Y\otimes\BS)$ which is compatible with the Clifford action in the following sense:  for any tangent vector field $v$ and any section $\psi$ of $\BS$,
\begin{align*}
\nabla(\cl(v)\psi) &= \cl(\nabla^{\rm LC}v)\psi + \cl(v)\nabla\psi
\end{align*}
where $\nabla^{\rm LC}$ is the Levi-Civita connection.  Given a spin-c connection, the associated \emph{spin-c Dirac operator} $\CD$ is defined to be the composition:
\begin{align*}
\CC^\infty(Y;\BS) \stackrel{\nabla}{\longrightarrow} \CC^\infty(Y;T^*Y\otimes\BS) \stackrel{\text{metric dual}}{\longrightarrow} \CC^\infty(Y;TY\otimes\BS) \stackrel{\cl}{\longrightarrow} \CC^\infty(Y;\BS) ~.
\end{align*}

\smallskip

Atiyah, Patodi and Singer \cite{ref_APS1, ref_APS2, ref_APS3} pioneered the study of the Dirac spectral flow.  What follows is the basic idea.  Suppose that $\CD$ is a spin-c Dirac operator.  Let $\{A_s\}_{s\in[0,1]}$ be a one parameter family of real valued $1$-forms.  Consider the one parameter family of Dirac operators $\{\CD_{A_s} = \CD - i\cl(A_s)\}_{s\in[0,1]}$.  In other words, they are Dirac operators associated to the spin-c connections $\nabla - iA_s\BI$, where $\BI$ is the identity endomorphism.  For simplicity, assume that $\CD_{A_0}$ and $\CD_{A_1}$ have trivial kernel.  The Dirac spectral flow is the count of the total number of zero eigenvalues of $\{\CD_{A_s}\}_{s\in[0,1]}$ with sign.  More precisely, the eigenvalues near zero move in a continuously differentiable manner if $\{A_s\}_{s\in[0,1]}$ is suitably generic.  The Dirac spectral flow is equal to the number of eigenvalues which cross zero with positive slope minus the number which cross zero with negative slope.  The resulting count turns out to be path independent and so depends only on the ordered pair $(A_0,A_1)$.

In particular, if there is a real valued $1$-form $a$, we can consider the spectral flow from $A_0 = 0$ to $A_1 = \frac{r}{2}a$ for $r>\!>1$.  This spectral flow can be regarded as a function of $r$, and will be denoted by $\sfa_a(\CD,r)$.  In \cite[\S5]{ref_Taubes_SW_Weinstein} and \cite{ref_Taubes_sf}, Taubes studied this spectral flow function $\sfa_a(\CD,r)$.
\begin{thm*} {\rm(}\cite[Proposition 5.5]{ref_Taubes_SW_Weinstein}{\rm)}
Suppose that $Y$ is a compact, oriented three manifold with a Riemannian metric $\dd s^2$.  Suppose that $\CD$ is a spin-c Dirac operator on $Y$.
Then, there exist a universal constant $\delta\in(0,\oh)$ and a constant $c_1$ is determined by $\dd s^2$ and $\CD$ such that
\begin{align*}
\big|\sfa_a(\CD,r)-\frac{r^2}{32\pi^2}\int_Y a\wedge\dd a\big| \leq c_1r^{\frac{3}{2}+\delta}
\end{align*}
for any $1$-form $a$ with $||a||_{\mathcal{C}^3}\leq 1$ and any $r\geq c_1$.
\end{thm*}
This theorem specifies the leading order term of the spectral flow function, and gives a bound on the subleading order term.

\subsection{Contact three manifold}
A $1$-form $a$ on an oriented three manifold is called a \emph{contact form} if $a\wedge\dd a>0$.  A contact form determines a vector field $v$ by $\dd a(v,\cdot)=0$ and $a(v)=1$.  This vector field is called the \emph{Reeb vector field}.  A contact form also defines a two plane distribution by $\ker(a)\subset TY$, which is called the \emph{contact hyperplane} or the \emph{contact structure}.  By the Frobenius theorem, $a\wedge\dd a>0$ implies that the contact hyperplane is everywhere non-integrable.

Given a contact form, it is a convenient normalization to take an adapted metric to consider spin-c structures and Dirac operators.  A Riemannian metric $\dd{s}^2$ is said to be \emph{adapted} to $a$ if $|a| = 1$ and $\dd a = 2*a$, where $*$ is the Hodge star operator.  Chern and Hamilton \cite{ref_CH} proved that such a metric always exists.

Suppose that $\CD$ is a spin-c Dirac operator on a contact three manifold $(Y,a)$.  It turns out that the zero eigensections of the Dirac operator $\CD_r = \CD - \frac{ir}{2}\cl(a)$ is closely related to the geometry of the contact form:  
\begin{itemize}
\item their derivative along the Reeb vector field is close to the multiplication by $ir/2$;
\item on the contact hyperplane, they almost solve certain Cauchy--Riemann equation.
\end{itemize}
The precise statements can be found in \cite[\S3]{ref_Ts2}.  The main goal is to understand more how $\CD_r$ is related to the geometry of the contact form.  The following question is the first step in this direction:  when $a$ is a contact form, is the subleading order term of $\sfa_a(\CD,r)$ of order $r$?

In this paper, we confirm the answer for certain types of contact forms, with a slightly larger order.

\begin{thmm*}
Suppose that $a$ is a Thurston--Winkelnkemper contact form \cite{ref_TW} on an open book decomposition.  Let $\CD$ be a spin-c Dirac operator.  Then, there exists a constant $c_2$ determined by the contact form $a$, the adapted metric $\dd{s}^2$ and the Dirac operator $\CD$ such that
\begin{align*}
\big|\sfa_a(\CD,r)-\frac{r^2}{32\pi^2}\int_Y a\wedge\dd a\big| \leq c_2r(\log r)^{\frac{9}{2}} ~.
\end{align*}
for any $r\geq c_2$.
\end{thmm*}

The Thurston--Winkelnkemper contact form will be explained momentarily.  The celebrated Giroux correspondence \cite{ref_Giroux} implies that each isotopy class of contact structures admits such a contact form.  In other words, the theorem asserts that the subleading order term of the spectral flow function is of order $r(\log r)^{\frac{9}{2}}$ for certain types of contact forms in each isotopy class of contact structures.

\subsection{Open book decomposition}
We now review the necessary background on the open book decomposition and the Thurston--Winkelnkemper contact form.  The reader can find a complete discussion on the open book decomposition in \cite[ch.9]{ref_OS} and \cite{ref_Etnyre}.  The notations set up here will be used throughout the \emph{rest of this paper}.

\subsubsection{Open book}\label{subsec_open_book}
An (\emph{abstract}) \emph{open book} consists of $(\Sigma,\tau)$ where
\begin{itemize}
\item $\Sigma$ is a Riemann surface with non-empty boundary, and $\pl\Sigma$ is a finite union of circles;
\item $\tau:\Sigma\to\Sigma$ is a diffeomorphism such that $\tau$ is the identity map on a collar neighborhood of $\pl\Sigma$.  The map $\tau$ is called the \emph{monodromy}.
\end{itemize}
An open book $(\Sigma,\tau)$ gives a three manifold $Y$ as follows:
\begin{align}\label{eqn_open_book_01}
Y &= \big(\Sigma\times_\tau S^1\big)\cup_\phi\big(\coprod_{|\pl\Sigma|}S^1\times B\big)
\end{align}
where $|\pl\Sigma|$ is the number of boundary components and $B$ is a two dimensional disk.  The component $\Sigma\times_\tau S^1$ is the mapping torus of $\tau$,
\begin{align*}
\Sigma\times_\tau S^1 &= \frac{\Sigma\times[0,2\pi]}{(\tau(x),0)\sim(x,2\pi)} ~.
\end{align*}
Since $\tau$ is the identity map on $\pl\Sigma$, the boundary of $\Sigma\times_\tau S^1$ is $(\pl\Sigma)\times S^1$.  The gluing map $\phi$ identifies the boundary of $\Sigma\times_\tau S^1$ with the boundary of $\coprod_{|\pl\Sigma|}S^1\times B$.  It is determined uniquely (up to isotopy) by the following properties:  for each boundary component of $\Sigma$,
\begin{itemize}
\item $\phi$ takes $(\pl\Sigma)\times\{y\}$ to the longitude of $\pl(S^1\times B)$ where $y\in S^1$;
\item $\phi$ takes $\{x\}\times S^1$ to the meridian of $\pl(S^1\times B)$ where $x\in\pl\Sigma$.
\end{itemize}
Note that there is an $S^1$ family of $\Sigma$ in $Y$, which are called the \emph{pages}.  The cores $\coprod_{|\pl\Sigma|}S^1\times\{0\}$ of the attaching handles are called the \emph{bindings}.  The term `\emph{near the bindings}' refers to the attaching handles $\coprod_{|\pl\Sigma|}S^1\times B$.  It is a particular tubular neighborhood of the bindings.

It is useful to describe the gluing map $\phi$ in terms of the local coordinate.  Let $\{\rho e^{it} ~|~ 1\leq\rho<1+50\delta, e^{it}\in S^1\}$ be a coordinate on a collar neighborhood of $\pl\Sigma$.  By taking $\delta$ to be sufficiently small, we may assume the monodromy $\tau$ to be the identity map on this chart.  The mapping torus $\Sigma\times_\tau S^1$ carries a canonical map to $S^1$.  Denote this map by $e^{i\theta}$.  It follows that $\{(\rho e^{it},e^{i\theta}) ~|~ 1\leq\rho<1+50\delta\}$ parametrize a collar neighborhood of $(\pl\Sigma)\times S^1$.  Let $\{(e^{it},\rho e^{i\theta}) ~|~ \rho< 1+50\delta\}$ be the coordinate on $S^1\times B$.  The gluing map $\phi$ is defined by identifying the corresponding coordinates.

\subsubsection{Contact form}\label{sec_TW_01}
Given an open book $(\Sigma, \tau)$, Thurston and Winkelnkemper \cite{ref_TW} construct a contact form $a$ on $Y$.  To start, choose a $1$-form $\zeta$ on $\Sigma$ such that
\begin{itemize}
\item $\dd\zeta$ defines an area form on $\Sigma$;
\item $2\zeta = (2-\rho)\dd t$ on the collar neighborhood $\{1\leq\rho<1+50\delta\}$ of $\pl\Sigma$.
\end{itemize}
There always exists such a $1$-form $\zeta$; see \cite[p.141]{ref_OS}.

Let $\chi(\theta)$ be a smooth, non-negative function of $\theta\in[0,2\pi]$ such that $\chi(\theta)=1$ near $0$ and $\chi(\theta)=0$ near $2\pi$.  For any $\theta\in[0,2\pi]$,
\begin{align*}  \omega_\theta = \chi(\theta)\dd\zeta + (1-\chi(\theta))\tau^*\dd\zeta  \end{align*}
is an area form on $\Sigma$.  Let $V$ be a positive scalar such that
\begin{align}\label{eqn_conformal_factor_01}
\frac{9}{10}\leq \big( 1 + \frac{2\chi'(\theta)}{V}\frac{\zeta\wedge\tau^*\zeta}{\omega_\theta} \big) \leq \frac{10}{9}
\end{align}
on $\Sigma$ and for any $\theta\in[0,2\pi]$.  With $\chi(\theta)$ and $V$ chosen, the contact form on the mapping torus $\Sigma\times_\tau S^1$ is defined by
\begin{align}\label{eqn_ob_contact_01}
a &= V\dd\theta + 2\chi(\theta)\zeta + 2(1-\chi(\theta))\tau^*\zeta ~.
\end{align}
A direct computation shows that
\begin{align*}
\dd a &= 2\big( \omega_\theta + \chi'(\theta)\dd\theta\wedge(\zeta-\tau^*\zeta) \big) ~, \\
\oh a\wedge\dd a &= \big( V + {2\chi'(\theta)}\frac{\zeta\wedge\tau^*\zeta}{\omega_\theta} \big)\dd\theta\wedge\omega_\theta ~.
\end{align*}
It follows that the Reeb vector field is
\begin{align}\label{eqn_Reeb_01}
\big( V + {2\chi'(\theta)}\frac{\zeta\wedge\tau^*\zeta}{\omega_\theta} \big)^{-1}\big( \frac{\pl}{\pl\theta} - \chi'(\theta)\omega_\theta^{-1}(\zeta - \tau^*\zeta) \big) ~.
\end{align}
The above expression is written on $\Sigma\times[0,2\pi]$, and $\frac{\pl}{\pl\theta}$ is the coordinate vector field.  The vector field $\omega_\theta^{-1}(\zeta - \tau^*\zeta)$ is defined by $\omega_\theta\big( \omega_\theta^{-1}(\zeta - \tau^*\zeta), \,\cdot\,\big) = (\zeta - \tau^*\zeta)(\,\cdot\,)$.

To extend the contact form to the attaching handles $\coprod_{|\pl\Sigma|}S^1\times B$, choose two smooth functions $f(\rho)$ and $g(\rho)$ of $\rho\in[0,1+50\delta)$ such that
\begin{itemize}
\item $f(\rho) = V$ and $g(\rho) = 2-\rho$ when $\rho\in[1-50\delta,1+50\delta)$;
\item $f(\rho) = \rho^2$ and $g(\rho) = 2-\rho^2$ when $\rho\in[0,50\delta]$;
\item for any $\rho\in(0,1+50\delta)$, $f'(\rho)\geq0$ and $g'(\rho)<0$.
\end{itemize}
It is not hard to see the existence of $f$ and $g$.  The contact form near the bindings is defined by
\begin{align}\label{eqn_ob_contact_02}
a &= f(\rho)\dd\theta + g(\rho)\dd t ~.
\end{align}
When $\rho<50\delta$, it is equal to $x\dd y - y\dd x + (2-x^2-y^2)\dd t$ in terms of the rectangular coordinate $x+iy = \rho e^{it}$.  Therefore, $a$ is a smooth $1$-form on $\coprod_{|\pl\Sigma|}S^1\times B$.

\subsubsection{Adapted metric}
The Main Theorem requires a specific adapted metric.  The metric is set to be
\begin{align}\label{eqn_metric_01}
a^2 + (\dd\rho)^2 + \frac{1}{4}(f'(\rho)\dd\theta + g'(\rho)\dd t)^2
\end{align}
near the bindings.  Such an adapted metric always exists.

Here is a parenthetical remark.  The attaching handles $\coprod_{|\pl\Sigma|}S^1\times B$ admits a $S^1\times S^1$-action:
\begin{align*} \begin{array}{ccccc}
(S^1\times S^1) &\times &(S^1\times B) &\longrightarrow &(S^1\times B) \\
\big((e^{it'},e^{i\theta'}) &, & (e^{it},\rho e^{i\theta})\big) &\mapsto &(e^{i(t+t')},\rho e^{i(\theta+\theta')})
\end{array} ~. \end{align*}
Near the boundary of the mapping torus $\Sigma\times_\tau S^1$,  the first $S^1$ factor rotates the boundary of $\Sigma$, and the second $S^1$ factor flips the pages.  The method of this paper should apply to any adapted metric which is invariant under this $S^1\times S^1$ action.  However, adapted metrics are not the main interests of this paper.  We only considered the metric (\ref{eqn_metric_01}), and did not try to work out general (locally $S^1\times S^1$ invariant) metrics.

\subsection{Contents of this paper}
Spin-c structures can be described more geometrically on a contact three manifold.  \S\ref{sec_Dirac_contact} is a review of the construction.  We also recall the results from Part I \cite{ref_Ts2} that will be used in this paper.

\S\ref{sec_sh_model} contains the key geometric construction.  Suppose that $Y$ is given by an open book $(\Sigma,\tau)$, and $a$ is the Thurston--Winkelnkemper contact form.  Let $D$ be a spin-c Dirac operator on $Y$.  Denote $D-\frac{ir}{2}\cl(a)$ by $D_r$.  In \S\ref{sec_sh_model} we construct another compactification $\ex{Y}$ of $\Sigma\times_\tau S^1$, and a one parameter family of Dirac operators $\{\ex{D}_r\}_{r\geq0}$ on $\ex{Y}$.  They have the following significance.
\begin{itemize}
\item In contrast to the open book (\ref{eqn_open_book_01}), $\ex{Y}$ admits a canonical map to $S^1$.
\item $\ex{D}_r$ on $\Sigma\times_\tau S^1\subset\ex{Y}$ is exactly the same as $D_r$ on $\Sigma\times_\tau S^1\subset Y$.
\item The Bochner--Weitzenb\"ock formula of $\ex{D}_r$ is very similar to that of $D_r$.
\end{itemize}

The canonical map $\ex{Y}\to S^1$ can be viewed as a gauge transform.  With such a gauge transform, Vafa and Witten \cite{ref_VW} have a brilliant argument to estimate the gap of spectrum of a Dirac operator.  By combining with the Bochner--Weitzenb\"ock formula, we prove that small spectrum of $\ex{D}_r$ are almost uniformly distributed.  This is done in \S\ref{sec_Vafa_Witten}.

In \S\ref{sec_local_model} we introduce two kinds of Dirac operators on $S^2\times S^1$.  They mimic the behavior of $D_r$ on $Y\backslash(\Sigma\times_\tau S^1)$ and $\ex{D}_r$ on $\ex{Y}\backslash(\Sigma\times_\tau S^1)$, respectively.  The eigenvalues and eigensections of these Dirac operators on $S^2\times S^1$ can be solved fairly explicitly.

With $\ex{D}_r$ on $\ex{Y}$ and the Dirac operators on $S^2\times S^1$, one can imagine that their spectrum approximate the spectrum of $D_r$ on $Y$.  \S\ref{sec_gluing} is devoted to compare these spectrum.  The precise statement is Theorem \ref{thm_glue_01}.  It is proved by gluing eigensections.

In \S\ref{sec_main_thm} we calculate the spectral flow from $D_0$ to $D_r$ with the help of the above models, and prove the Main Theorem.

\begin{rmk}
The constants $c_{(\cdot)}$ in this paper are always \emph{independent} of $r$.  In other words, they only depend on the contact form $a$, the metric $\dd s^2$ and the unperturbed spin-c Dirac operator.  The subscript is simply to indicate that these constants might increase/decrease after each step.  The subscript will be returned to $1$ at the beginning of each section.
\end{rmk}

\begin{ack*}
The author is very grateful to Cliff Taubes for bringing \cite{ref_VW} to his attention, and for many helpful discussions.  He would like to thank the Department of Mathematics at Harvard University, where part of this work was carried out.
\end{ack*}

\section{Dirac Operator on Contact Three Manifold}\label{sec_Dirac_contact}
Suppose that $(Y,a)$ is a contact three manifold with an adapted metric $\dd{s}^2$.  As in \cite[\S2.1]{ref_Taubes_SW_Weinstein}, the spin-c structures and spin-c connections can be described more geometrically.  It works for a stable Hamiltonian structure as well, of which a contact form is a special case.  We will also work with stable Hamiltonian structures in this paper.

\subsection{Stable Hamiltonian structure}
Suppose that $Y$ is a compact, oriented three manifold.  A \emph{stable Hamiltonian structure} is a pair $(b,\omega)$ where ${b}$ is a $1$-form and $\omega$ is a $2$-form such that
\begin{align*}\left\{\begin{aligned}
\dd\omega &= 0 ~,  \\
b\wedge\omega &> 0 ~,  \\
\ker(\omega) &\subset\ \ker(\dd b) ~.
\end{aligned}\right.\end{align*}
This notion was identified in \cite[\S2]{ref_BEHWZ} and \cite[\S2]{ref_CM}.  Given a contact form $a$, the pair $(a,\dd a)$ is a stable Hamiltonian structure.  A stable Hamiltonian structure determines a vector field $v$ by $\omega(v,\cdot)=0$ and ${b}(v) = 1$.  We still call $v$ the `\emph{Reeb vector field}'.

Such a structure needs not to come from a contact form.  Here is a standard example.  Suppose that $N$ is a compact surface with a symplectic form $\omega$, and $\tau$ is a symplectomorphism of $(N,\omega)$.  Let $\dd\theta$ be the pull-back of the standard $1$-form on $S^1$ by the projection $N\times_\tau S^1\to S^1$.  Since $\tau$ is a symplectomorphism, the $2$-form $\omega$ descends from $N\times\BR$ to $N\times_\tau S^1$.  Then, $(\dd\theta,\omega)$ is a stable Hamiltonian structure on the mapping torus $N\times_\tau S^1$.  Since $\ker(\dd\theta)$ is everywhere integrable, it is different from a contact form.

\subsubsection{Conformally adapted metric}
In Part I \cite{ref_Ts2}, a slightly more general type of metric is also considered.
Suppose that $({b},\omega)$ is a stable Hamiltonian structure on $Y$.  A Riemannian metric $\dd{s}^2$ is said to be \emph{conformally adapted} to $({b},\omega)$ if
\begin{align*}   |{b}|=\Omega^{-1} \qquad\text{and}\qquad *\omega = 2\Omega^{-1}{b}   \end{align*}
for some function $\Omega\in\CC^\infty(Y;\BR)$ with
\begin{align*}   \frac{9}{10}\leq\Omega\leq\frac{10}{9} ~.   \end{align*}
The particular bounds are just convenient normalizations; any other fixed bounds would do the job.  The operator $*$ is the Hodge star operator of $\dd s^2$.  Note that $\Omega^{-2}\dd s^2$ is an adapted metric\footnote{Namely, $|{b}|=1$ and $*\omega = 2{b}$ with respect to $\Omega^{-2}\dd s^2$.}.  The function $\Omega$ is called the \emph{conformal factor}.  Such a metric always exists.  The argument of Chern and Hamilton \cite{ref_CH} applies to a stable Hamiltonian structure as well.  It is equivalent to an almost complex structure $J$ on $\ker({b})$ such that $\Omega^2\omega(\,\cdot\, , J(\,\cdot\,))$ defines a metric on $\ker({b})$.

\subsubsection{Spin-c structure}\label{subsec_can_spinc}
With the metric fixed, $({b},\omega)$ determines a \emph{canonical spin-c structure}.  The spinor bundle is given by $\underline{\BC}\oplus K^{-1}$, where $\underline{\BC}$ is the trivial bundle and $K^{-1}$ is isomorphic as an ${\rm SO}(2)$ bundle to $\ker({b})$ with the orientation defined by $\omega$.  More precisely, for any $u\in\ker({b})$, let $J(u)$ be the metric dual of $\Omega^2\omega(u,\,\cdot\,)$.  The local sections of $K^{-1}$ consists of $u-iJ(u)$ with $u\in\ker({b})$.

The Clifford action is defined as follows.  The Clifford action of the Reeb vector field $v$ acts as $i\Omega$ on $\underline{\BC}$ and as $-i\Omega$ on $K^{-1}$.  In other words, $\underline{\BC}\oplus K^{-1}$ is the eigenbundle splitting of $\cl(v)$.  Let $\bo$ be the depicted unit-normed section of $\underline{\BC}$.  For any unit vector $u\in\ker({b})$, the Clifford action of $u$ is defined by
\begin{align*} \begin{array}{cccc}
\cl(u): &\bo &\mapsto &\frac{1}{\sqrt{2}}(u-iJ(u)) \\
&\frac{1}{\sqrt{2}}(u-iJ(u)) &\mapsto &-\bo
\end{array} ~, \end{align*}
and the Clifford action of $J(u)$ is defined by
\begin{align*} \begin{array}{cccc}
\cl(J(u)): &\bo &\mapsto &\frac{i}{\sqrt{2}}(u-iJ(u)) \\
&\frac{1}{\sqrt{2}}(u-iJ(u)) &\mapsto &i\bo
\end{array} ~. \end{align*}
It is straightforward to check that it does define a spin-c structure.

The set of spin-c structures has a free transitive action of $\HD{2}(Y;\BZ)$, see \cite[Appendix D]{ref_LM}.  The action is given by tensoring with a complex line bundle $E$.  It follows that a spinor bundle can be written as $E\oplus E K^{-1}$ for some Hermitian line bundle $E$.  The Clifford action is induced from that on the canonical spinor bundle, and the splitting $E\oplus EK^{-1}$ is the eigenbundle splitting of $\cl(v)$.

\subsubsection{Spin-c connection}\label{subsec_can_spinc_conn}
The canonical spinor bundle $\underline{\BC}\oplus K^{-1}$ carries a \emph{canonical spin-c connection}, which we denote by $\nabla_o$.  It is the unique spin-c connection whose associated Dirac operator $D_o$ annihilates $\Omega^{-1}\bo$, the depicted section $\bo$ rescaled by $\Omega^{-1}$.  The proof for its existence and uniqueness can be found in \cite[Lemma 10.1]{ref_Hutchings}.

The canonical connection can be written down explicitly in terms of a local trivialization.  Let $e_1,e_2,e_3$ be an oriented, orthonormal local frame for $TY$, where $e_3 = \Omega^{-1}v$ is the Reeb vector field multiplied by $\Omega^{-1}$.  Using the trivialization $\bo$ and $\frac{1}{\sqrt{2}}(e_1-ie_2)$, the local sections of $\underline{\BC}\oplus K^{-1}$ are identified with $\BC^2$ valued functions.  With respect to this trivialization, the canonical connection is given by
\begin{align}\label{eqn_can_conn_01}
\nabla_o\psi = \dd\psi + \oh\sum_{j\leq k}\theta_j^k\cl(e_j)\cl(e_k)\psi + \frac{i}{2}\big(\theta_1^2 - \Omega^2*\dd(\Omega^{-1}{b})\big)\psi
\end{align}
where $\theta_j^k$ is the coefficient $1$-form of the Levi-Civita connection, i.e.\ $\nabla^{\rm LC}e_j = \sum_k\theta_j^k\otimes e_k$.  We leave it to the reader to check that the expression does define a spin-c connection (see also \cite[\S II.4]{ref_LM}).

It follows from $\ker(\omega)\subset\ker(\dd{b})$ and the structure equation that
\begin{align*}
\theta_1^3(e_3) - \Omega^{-1}e_1(\Omega) &=0  ~, &\text{ and }&
&\theta_2^3(e_3) - \Omega^{-1}e_2(\Omega) &=0 ~.
\end{align*}
It follows from $\dd\omega = 0$ and the structure equation that
\begin{align*}   \theta_1^3(e_1) + \theta_2^3(e_2) +2\Omega^{-1}e_3(\Omega)  &= 0 ~.   \end{align*}
Using these relations, a direct computation shows that Dirac operator associated to (\ref{eqn_can_conn_01}) annihilates $\psi = (\Omega^{-1},0)$.  Hence, (\ref{eqn_can_conn_01}) defines the canonical connection.

With this understood, any spin-c connection on a spinor bundle $E\oplus EK^{-1}$ can be expressed as $\nabla_o\otimes\aE$ where $\aE$ is a unitary connection on $E$.

\subsection{Results of Part I}

We now recall the results of Part I \cite{ref_Ts2}.  Suppose that $(Y,a)$ is a contact three manifold with a conformally adapted metric $\dd s^2$.  Suppose that $E\to Y$ is a Hermitian line bundle with a unitary connection $\aE$.  For any $r\geq 0$, consider the one parameter family of Dirac operators on $E\oplus EK^{-1}$ defined by
\begin{align*}
D_r &= \cl(\nabla_o\otimes\aE - \frac{ir}{2}a) ~.
\end{align*}
Denote by $\sfa_a(r)$ the spectral flow from $D_0$ to $D_r$.  In Part I \cite{ref_Ts2}, we obtained the following estimate on $\sfa_a(r)$.

\begin{thm}\label{thm_part_I}
There exists a constant $c_1$ determined by the contact form $a$, the conformally adapted metric $\dd s^2$ and the connection $\aE$ such that for any $\rb\geq 2c_1$,
\begin{align*}
\big| \sfa_a(\rb) - \frac{\rb^2}{32\pi^2}\int_Y a\wedge\dd a \big| &\leq c_1\rb(\log\rb)^{\frac{9}{2}} + |\dot{\eta}(\rb)| + c_1\big|\int_{1}^\rb\ddot{\eta}(r)\dd r\big| ~.
\end{align*}
The function $\dot{\eta}(\rb)$ is defined by
\begin{align}\label{eqn_eta_00}
\dot{\eta}(\rb) = \rb^{-\oh}(\log\rb)^\oh\Big( \sum_{\psi\in\CV_\rb^+}\int_{\lambda_\psi}^{\frac{1}{3}\rb^\oh}e^{-20(\rb^{-1}\log\rb)u^2}\dd u - \sum_{\psi\in\CV_\rb^-}\int^{\lambda_\psi}_{-\frac{1}{3}\rb^\oh}e^{-20(\rb^{-1}\log\rb)u^2}\dd u\Big)
\end{align}
where $\CV_\rb^+$ consists of orthonormal eigensections of $D_\rb$ whose eigenvalue belongs to $(0,\frac{1}{3}\rb^\oh)$, $\CV_\rb^-$ consists of orthonormal eigensections of $D_\rb$ whose eigenvalue belongs to $(-\frac{1}{3}\rb^\oh,0)$, and $\lambda_\psi$ is the corresponding eigenvalue.  The function $\ddot{\eta}(r)$ is defined by
\begin{align}\label{eqn_eta_01}
\ddot{\eta}(r) = (r^{-\frac{3}{2}}\log r)\sum_{\psi\in\CV_r}(\lambda_\psi e^{-20(r^{-1}\log r)\lambda_\psi^2})
\end{align}
where $\CV_r$ consists of orthonormal eigensections of $D_r$ whose eigenvalue belongs to $(-\frac{1}{3}r^\oh,\frac{1}{3}r^\oh)$.
\end{thm}

\begin{proof}
It follows from Proposition 5.7 and Proposition 5.6 of \cite{ref_Ts2}.
\end{proof}

\subsection{Main result}
Theorem \ref{thm_part_I} reduces the spectral flow estimate to $\dot{\eta}(r)$ and $\ddot{\eta}(r)$.  These two functions measure certain spectral asymmetry of $D_r$ within $(-\frac{1}{3}r^{\oh},\frac{1}{3}r^\oh)$.  The main goal of this paper is to show that the spectrum within $(-\frac{1}{3}r^{\oh},\frac{1}{3}r^\oh)$ are almost uniformly distributed when $a$ is the Thurston--Winkelnkemper contact form.  The rest of this paper is devoted to the proof of the following theorem, and its conditions (i) and (ii) will be assumed throughout the \emph{rest of this paper}.

\begin{thm}\label{thm_main_01}
Let $(\Sigma,\tau)$ be an open book.  Denote the three manifold (\ref{eqn_open_book_01}) by $Y$, and denote the Thurston--Winkelnkemper contact form by $a$.  Suppose that
\begin{enumerate}
\item $\dd s^2$ is a conformally adapted metric which is equal to (\ref{eqn_metric_01}) near the bindings, and whose conformal factor $\Omega$ is equal to
\begin{align}\label{eqn_Omega_01}
\big( 1+\frac{2\chi'(\theta)}{V}\frac{\zeta\wedge\tau^*\zeta}{\omega_\theta} \big)^{-1} ~;
\end{align}
this factor appears in (\ref{eqn_conformal_factor_01}) and (\ref{eqn_Reeb_01}), and is equal to $1$ near the bindings;
\item the unitary connection $\aE$ on $E\to Y$ is gauge equivalent to the trivial connection near the bindings; note that any bundle is topologically trivial near the bindings.
\end{enumerate}
Then, there exists a constant $c_2$ determined by $a$, $\dd s^2$ and $\aE$ such that
\begin{align*}
\dot{\eta}(r) &\leq c_2 r(\log r)^\oh   &\text{and}&   &\ddot{\eta}(r) &\leq c_2\log r
\end{align*}
for any $r\geq c_2$.
\end{thm}

By combining this theorem with Theorem \ref{thm_part_I}, there exists a constant $c_3$ such that
\begin{align}\label{sf_est_01}
\big| \sfa_a(r) - \frac{r^2}{32\pi^2}\int_Y a\wedge\dd a \big| &\leq c_3 r(\log r)^{\frac{9}{2}}
\end{align}
for any $r\geq c_3$.

The technical conditions (i) and (ii) of Theorem \ref{thm_main_01} are not crucial for the spectral flow function $\sfa_a(r)$.  The reason goes as follows.
\begin{enumerate}
\item As explained in \cite[\S2.1]{ref_Ts2}, the spectral flow function $\sfa_a(r)$ is invariant under the conformal change of metric.  It follows that the spectral flow estimate (\ref{sf_est_01}) works for any adapted metric that is equal to (\ref{eqn_metric_01}) near the bindings.  Notice that we do not claim that $\dot{\eta}(r)$ and $\ddot{\eta}(r)$ are invariant under the conformal change of metric.
\item According to \cite[Proposition 5.9]{ref_Ts2}, different choices of $\aE$ lead to a $\CO(r)$ difference of the spectral flow function.
\end{enumerate}
With this understood, we conclude the following theorem.

\begin{thm}
Suppose that $a$ is a Thurston--Winkelnkemper contact form \cite{ref_TW}.  Suppose that $\dd{s}^2$ is an adapted metric which is equal to (\ref{eqn_metric_01}) near the bindings.  Let $D$ be a spin-c Dirac operator.  Then, there exists a constant $c_4$ determined by $a$, $\dd{s}^2$ and $D$ such that
\begin{align*}
\big|\sfa_a(r)-\frac{r^2}{32\pi^2}\int_Y a\wedge\dd a\big| \leq c_4 r(\log r)^{\frac{9}{2}} ~.
\end{align*}
for any $r\geq c_4$.
\end{thm}

\begin{rmk}
The conformal factor of Theorem \ref{thm_main_01}(i) shows up naturally.  In the construction of \S\ref{sec_TW_01}, there are two volume forms on $\Sigma\times_\tau S^1$:
\begin{align*}
\oh a\wedge\dd a  \qquad\text{ and }\qquad  \dd\theta\wedge\omega_\theta ~.
\end{align*}
Their ratio defines a function which is equal to $V$ on the boundary of $\Sigma\times_\tau S^1$.  This particular conformal factor plays a key role in the proof of Theorem \ref{thm_Vafa_Witten} below.
\end{rmk}

\section{From Open Book to Mapping Torus}\label{sec_sh_model}
Under the setting of Theorem \ref{thm_main_01}, the main purpose of this section is to construct a model which captures the spectrum of $D_r = \cl(\nabla_o\otimes\aE - \frac{ir}{2}a)$.  For simplicity, assume from now on that $\Sigma$ has \emph{only one} boundary component .  If $\Sigma$ has more than one boundary components, one simply needs to duplicate the construction and the argument.

The model consists of the following objects (which will be introduced momentarily):
\begin{itemize}
\item another compactification $\ex{Y}$ of $\Sigma\times_\tau S^1$, which is a surface bundle over $S^1$;
\item a stable Hamiltonian structure $(\ex{a},2\ex{\omega})$ on $\ex{Y}$, and a conformally adapted metric $\dd s^2$ on $\ex{Y}$;
\item a spinor bundle $(\ex{E}\oplus\ex{E}\ex{K}^{-1})\otimes\ex{L}_r\to\ex{Y}$, and a Dirac operator $\ex{D}_r$ on it.
\end{itemize}
For brevity, the model $(\ex{Y}\to S^1,\ex{a},\ex{\omega},\dd s^2,(\ex{E}\oplus\ex{E}\ex{K}^{-1})\otimes\ex{L}_r,\ex{D}_r)$ will be denoted by $(\ex{Y},\ex{D}_r)$.  The Dirac operator $\ex{D}_r$ has the following salient features:
\begin{enumerate}
\item on $\Sigma\times_\tau S^1$, $\ex{D}_r$ is identically the same as $D_r$;
\item the ``{small}" spectrum of $\ex{D}_r$ is almost uniformly distributed;  the precise statement appears in Theorem \ref{thm_Vafa_Witten}.
\end{enumerate}

\subsection{The mapping torus}
To start, compactify $\Sigma$ by attaching a disk to its boundary.  To be more precise, let $\{\rho e^{it}~|~\rho\geq1, e^{it}\in S^1\}$ be the coordinate on a collar neighborhood of $\pl\Sigma$.  The attaching disk is given by $B = \{\rho e^{it} | 0\leq\rho<1+50\delta\}$.  The compactification is done by identifying the coordinate.  Denote the resulting closed surface by $\ex{\Sigma}$.

Since $\tau$ is the identity on a collar neighborhood of $\pl\Sigma$, it naturally extends to a monodromy $\ex{\tau}$ of $\ex{\Sigma}$ by $\ex{\tau}|_B = \BI_B$.  Let $\ex{Y}$ be the mapping torus $\ex{\Sigma}\times_{\ex{\tau}} S^1$.  Equivalently,
\begin{align*}  \ex{Y} = (\Sigma\times_\tau S^1)\cup\big(\coprod_{|\pl\Sigma|=1}B\times S^1\big)  \end{align*}
where $B$ is attached to $\Sigma$.  It is a surface bundle over $S^1$.  Denote the fibration map $\ex{Y}\to S^1$ by $e^{i\theta}$.

\subsection{The extension of the $1$-form $a$}\label{subsec_ext_oneform}
The $1$-form $a$ can be extended to $\ex{Y}$.  The extension will be denoted by $\ex{a}$.

Near the boundary of $\Sigma\times_\tau S^1$, the $1$-form $a$ (\ref{eqn_ob_contact_01}) is equal to $V\dd\theta + (2-\rho)\dd t$.  Choose a smooth function $\ex{g}(\rho)$ of $\rho\in[0,1+50\delta)$ such that
$$ \ex{g}(\rho) = 0 \text{ when }\rho\leq50\delta \quad\text{ and }\quad \ex{g}(\rho) = 2-\rho \text{ when }\rho\geq1-50\delta ~. $$  On the attaching handle $B\times S^1$, the $1$-form $\ex{a}$ is defined by
\begin{align}\label{eqn_ext_a}
\ex{a} = V\dd\theta + \ex{g}(\rho)\dd t ~.
\end{align}
It is clear that $\ex{a}$ is a smooth $1$-form on $\ex{Y}$.  Notice that $\ex{a}$ is \emph{no longer} a contact form on $\ex{Y}$.

\subsection{The extension of the metric}\label{subsec_ext_metric}
Near the boundary of $\Sigma\times_\tau S^1$, the metric (\ref{eqn_metric_01}) is equal to $a^2 + (\dd\rho)^2 + (\oh\dd t)^2$.  For any scalar $\sigma\in(\frac{9}{10},\frac{11}{10})$, choose a smooth function $\ex{h}_{\sigma}(\rho)$ of $\rho\in[0,1+50\delta)$ such that
\begin{itemize}
\item $\ex{h}_{\sigma}(\rho) = \oh\rho^2$ when $\rho\leq50\delta$;
\item $\ex{h}_{\sigma}(\rho) = \sigma + \oh(\rho-2)$ when $\rho\geq1-50\delta$;
\item $\ex{h}'_{\sigma}(\rho) > 0$ for any $\rho\in(0,1)$, where $\ex{h}'_{\sigma}(\rho)$ means the derivative of $\ex{h}_{\sigma}(\rho)$ in $\rho$.
\end{itemize}
Moreover, the functions $\{\ex{h}_\sigma(\rho)\}_{\frac{9}{10}<\sigma<\frac{11}{10}}$ have uniformly bounded $\CC^k$-norm for any non-negative integer $k$.  Namely, there exist constants $c_k$ such that
\begin{align}\label{rmk_delta_r}
\sup_{\sigma\in(\frac{9}{10},\frac{11}{10})}\sup_{\rho\in[0,1+50\delta)}\big| \frac{\pl^k\ex{h}_\sigma(\rho)}{\pl\rho^k} \big| \leq c_k ~.
\end{align}

The metric on the attaching handle $B\times S^1$ is taken to be
\begin{align} \label{eqn_ext_metric}
\dd s^2 &= \ex{a}^2 + (\dd\rho)^2 + (\ex{h}'_{\sigma}(\rho)\dd t)^2 ~,
\end{align}
and its volume form is
\begin{align} \label{eqn_ext_volume}
\ex{h}'_{\sigma}(\rho)\,\dd\theta\wedge\dd t\wedge\dd\rho ~.
\end{align}
It is clear that the construction gives a smooth extension of the Riemannian metric on $\Sigma\times_\tau S^1$ to $\ex{Y}$.  The precise choice of $\sigma$ will be made later.

\subsection{The stable Hamiltonian structure}\label{subsec_sh_str}
The $2$-form $\oh\dd a$ also admits an extension $\ex{\omega}$ by:
\begin{align*}
\ex{\omega} &= \begin{cases}
\oh\dd a  &\text{on }\Sigma\times_\tau S^1 ~, \\
\ex{h}'_\sigma\dd t\wedge\dd\rho  &\text{on }B\times S^1  ~.
\end{cases} \end{align*}
It is straightforward to check that $(\ex{a},2\ex{\omega})$ forms a stable Hamiltonian structure, and the metric defined in \S\ref{subsec_ext_metric} is conformally adapted to it.

\subsubsection{The canonical spin-c structure}\label{subsec_ext_canspin}
As explained in \S\ref{subsec_can_spinc}, there is a canonical spin-c structure determined by $(\ex{a}$, $2\ex{\omega})$ and the metric $\dd s^2$.  Denote the canonical spinor bundle by $\canspinex\to\ex{Y}$.

It is convenient to fix a trivialization of $\ex{K}^{-1}\to B\times S^1\subset\ex{Y}$.  Consider the unitary vector field
\begin{align}\label{eqn_ext_trivial_K}
\sqrt{2}\pl_B &= \frac{1}{\sqrt{2}}e^{it} \big( \pl_\rho + \frac{i}{\ex{h}'_{\sigma}}(\pl_t-\frac{\ex{g}}{V}\pl_\theta) \big) ~.
\end{align}
The expression is smooth when $\rho>0$.  When $\rho<50\delta$, it is equal to $\frac{1}{\sqrt{2}}(\pl_x - i\pl_y)$ in terms of the rectangular coordinate $x-iy = \rho e^{it}$.

\subsubsection{The extension of $E$}
The Hermitian line bundle $E\to Y$ is assumed to be trivial near the bindings, and the connection $\aE$ is assumed to be the exterior derivative.  It follows that the bundle and the connection can naturally be regarded as being defined over $\ex{Y}$.  Denote the bundle by $\ex{E}\to\ex{Y}$, and the connection by $\aeE$.

\subsection{The degree $r$ bundle}\label{subsec_deg_r}
The purpose of this subsection is to construct a Hermitian line bundle $\ex{L}_r\to\ex{Y}$ with a unitary connection $\aer$ for any $r\geq20$.  The curvature of $\aer$ supports only on the attaching handle $B\times S^1$, and is proportional to $r$.

Since $r > 20$, $\frac{9}{10} < \frac{[r]}{r} < \frac{11}{10}$.  Set the constant $\sigma$ to be
\begin{align}	\sigma = \frac{[r]}{r} ~.		\end{align}
Although $\sigma$ depends on $r$, the function $\ex{h}_\sigma(\rho)$ is \emph{independent of $r$} in the sense of (\ref{rmk_delta_r}).

To construct $\ex{L}_r$, consider the trivial bundle over $\Sigma\times_\tau S^1$ and $B\times S^1$.  Let $\bo_\Sigma$ and $\bo_B$ be the depicted unitary sections, respectively.  On the overlap region $\{1\leq\rho<1+50\delta\}$, identify the bundles by the transition rule
\begin{align}\label{eqn_Lr_trans}    e^{i[r]t}\cdot\bo_\Sigma = \bo_B ~.    \end{align}
The unitary connection $\aer$ is defined as follows:
\begin{itemize}
\item over $\Sigma\times_\tau S^1$, the connection $\aer$ is $\dd$ with respect to $\bo_\Sigma$;
\item over $B\times S^1$, the connection $\aer$ is $\dd + ir(\ex{h}_{\sigma} + \oh\ex{g})\dd t$ with respect to $\bo_B$.
\end{itemize}
When $1\leq\rho<1+50\delta$,
\begin{align*}	 ir(\ex{h}_\sigma + \oh\ex{g})\dd t = i(r\sigma)\dd t = i[r]\dd t	~.	\end{align*}
It follows that $\aer$ obeys the transition rule (\ref{eqn_Lr_trans}), and hence defines a connection on $\ex{L}_r$.

\subsection{The Dirac operator on $(\ex{E}\oplus\ex{E}\ex{K}^{-1})\otimes\ex{L}_r$}\label{subsec_ext_Lr}
The bundle $(\ex{E}\oplus\ex{E}\ex{K}^{-1})\otimes\ex{L}_r\to\ex{Y}$ is also a spinor bundle.  Let $\nabla_\can$ be the canonical connection on $\underline{\BC}\oplus\ex{K}^{-1}$.  The connection ${\nabla}_\can\otimes\aeE\otimes\aer$ is a spin-c connection on $(\underline{\BC}\oplus\ex{K}^{-1})\otimes\ex{E}\otimes\ex{L}_r$.  Perturb the connection by $-\frac{ir}{2}\ex{a}$, and consider the corresponding Dirac operator.  Namely,
\begin{align} \left\{ \begin{aligned}
\ex{\nabla}_r &= {\nabla}_\can\otimes\aeE\otimes\aer - \frac{ir}{2}\ex{a} ~, \\
\ex{D}_r &= \cl\circ({\nabla}_\can\otimes\aeE\otimes\aer - \frac{ir}{2}\ex{a}) ~.
\end{aligned} \right. \end{align}

The Weitzenb\"ock formula for $\ex{D}_r$ reads
\begin{align*}
\ex{D}_r^2 \psi = \ex{\nabla}_r^*\ex{\nabla}_r\psi + \ex{\kappa}(\psi) + \cl(\ex{F}_{r} - \frac{ir}{2}\dd\ex{a})(\psi)
\end{align*}
where $\ex{F}_{r}$ is the curvature of $\aer$.  Here $\ex{\kappa}$ consists of the scalar curvature, the curvature of $\nabla_\can$ and the curvature of $\aeE$; in particular, $\ex{\kappa}$ is an operator independent of $r$.  On $\Sigma\times_\tau S^1$,
\begin{align*}
\cl(\ex{F}_{r} - \frac{ir}{2}\dd\ex{a}) = ir{\Omega}^{-1}\cl(\ex{a})
\end{align*}
where $\Omega$ is the conformal factor (\ref{eqn_Omega_01}).  On $B\times S^1$,
\begin{align*}
\cl(\ex{F}_{r} - \frac{ir}{2}\dd\ex{a}) &= \cl\big( ir(\ex{h}'_{\sigma} + \frac{\ex{g}'}{2})\dd\rho\wedge\dd t - ir\frac{\ex{g}'}{2}\dd\rho\wedge\dd t \big) \\
&= ir\cl(\ex{h}'_{\sigma}\,\dd\rho\wedge\dd t) = ir{\Omega}^{-1}\cl(\ex{a}) ~.
\end{align*}
It follows that the Weitzenb\"ock formula becomes
\begin{align}\label{eqn_ext_Weitzenbock}
\ex{D}_r^2 \psi = \ex{\nabla}_r^*\ex{\nabla}_r\psi + \ex{\kappa}(\psi) + ir{\Omega}^{-1}\cl(\ex{a})(\psi) ~.
\end{align}
The operator $i\cl(\ex{a})$ acts diagonally on $(\ex{E}\oplus\ex{E}\ex{K}^{-1})\otimes\ex{L}_r$.  It acts as $-{\Omega}^{-1}$ on the $\ex{E}\ex{L}_r$ summand, and acts as ${\Omega}^{-1}$ on the $\ex{E}\ex{K}^{-1}\ex{L}_r$ summand.

\begin{rmk}\label{rmk_identification_01}
The spinor bundle $(\ex{E}\oplus\ex{E}\ex{K}^{-1})\otimes\ex{L}_r$ is topologically trivial over the attaching handle $B\times S^1$.  The bundle $\ex{E}$ is trivialized by an $\aeE$-parallel, unit-normed section.  The bundle $\ex{L}_r$ is trivialized by $\bo_{B}$ as in \S\ref{subsec_deg_r}.  The bundle $\ex{K}^{-1}$ is trivialized by $\sqrt{2}\pl_B$ (\ref{eqn_ext_trivial_K}).  They induce a unitary trivialization of $(\ex{E}\oplus\ex{E}\ex{K}^{-1})\otimes\ex{L}_r$ over $B\times S^1\subset\ex{Y}$, and the sections on $B\times S^1$ can be identified with $\BC^2$ valued functions.
\end{rmk}

\section{Eigenvalue Distribution of $\ex{D}_r$}\label{sec_Vafa_Witten}
The main purpose of this section is to show that the ``small eigenvalues" of $\ex{D}_r$ are almost ``uniformly distributed".  The strategy here is learned from Vafa and Witten \cite{ref_VW}.  They applied the Atiyah--Patodi--Singer index theorem \cite{ref_APS1, ref_APS3} to prove that there cannot be large gaps in the Dirac spectrum.

The following proposition gives an integral estimate on the eigensections.

\begin{prop}\label{prop_ext_estimate}
There exists a constant $c_1$ determined by the stable Hamiltonian structure $(\ex{a}, \ex{\omega})$, the metric $\dd s^2$ and the connection $\aeE$ such that the following holds.  For any $r\geq c_1$, suppose that $\psi$ is an eigensection of $\ex{D}_r$ whose eigenvalue $\lambda$ satisfies $|\lambda|^2\leq \frac{3}{4}r$.  Then
\begin{align*}
\int_{\ex{Y}}|\beta|^2 + r^{-1}\int_{\ex{Y}}|\ex{\nabla}_r\beta|^2 &\leq c_1r^{-1}\int_{\ex{Y}}|\alpha|^2
\end{align*}
where $\alpha$ is the $\ex{E}\ex{L}_r$ component of $\psi$, and $\beta$ is the $\ex{E}\ex{K}^{-1}\ex{L}_r$ component of $\psi$.
\end{prop}

\begin{proof}
With the Weitzenb\"ock formula (\ref{eqn_ext_Weitzenbock}), the proof is exactly the same as that for \cite[Proposition 2.2]{ref_Ts2}.
\end{proof}

Here comes the main result about the spectrum distribution.
\begin{thm}\label{thm_Vafa_Witten}
There exist constants $c_2$ and $c_3$ determined by the stable Hamiltonian structure $(\ex{a}, \ex{\omega})$, the metric $\dd s^2$ and the connection $\aeE$ with the following significance.  For any $r\geq c_2$, let $\{\lambda_j\}_{j\in\BZ}$ be the spectrum of $\ex{D}_r$, which are arranged in ascending order.  Then for any $|\lambda_j|\leq\oh r^\oh$,
\begin{align*}
\big| \lambda_{j+\fj} - \lambda_j - \frac{1}{V} \big| \leq c_2r^{-\oh}
\end{align*}
where $\fj = r(\int_{\ex{Y}}\dd\theta\wedge\ex{\omega}) + c_3 > 0$.
\end{thm}

\begin{proof}
Regard the fibration map $\ex{Y}\to S^1$ as a gauge transform.  The Dirac operator $ e^{i\theta}\ex{D}_re^{-i\theta} = \ex{D}_r - i\cl(\dd\theta)$ is gauge equivalent to $\ex{D}_r$, and hence has the same spectrum as $\ex{D}_r$.  Consider the one parameter family of Dirac operators defined by
\begin{align*}
\fD_s = \ex{D}_r - is\cl(\dd\theta)
\end{align*}
for $s\in[0,1]$.  Arrange the eigenvalues $\lambda_j(s)$ of $\fD_s$ in ascending order,
\begin{align*}
-\infty < \cdots \leq \lambda_{j-1}(s) \leq \lambda_j(s) \leq \lambda_{j+1}(s) \leq \cdots <\infty ~,
\end{align*}
and normalize the index so that at $s=0$, $\lambda_1(0)$ is the smallest non-negative eigenvalue.  Since $\fD_0$ is gauge equivalent to $\fD_1$, there exists an integer $\fj$ such that $\lambda_j(1) = \lambda_{j+\fj}(0)$ for any $j\in\BZ$.  According to \cite[section 7]{ref_APS3}, the integer $\fj$ is the spectral flow of the family $\{\fD_s\}_{0\leq s\leq 1}$, and can be computed by the index formula \cite[(4.3)]{ref_APS1}.

\smallskip
(\emph{The spectral flow computation})\;
Since $\fD_0$ is gauge equivalent to $\fD_1$, the boundary contribution of the index formula at $s=0$ cancels with that at $s=1$.  It follows that
\begin{align*}
\fj &= \int_{[0,1]\times\ex{Y}} \big(\frac{1}{8}{\rm c}_1^2(\ex{K}^{-1}\ex{E}^2\ex{L}_r^2) - \frac{1}{24}{\rm p}_1([0,1]\times\ex{Y})\big) ~.
\end{align*}
Here, ${\rm p}_1$ is the first Pontryagin class of the metric.  It is constructed from the Weyl curvature ${\rm p}_1 = \frac{1}{4\pi^2}(|W_+|^2 - |W_-|^2)$, and hence vanishes on $[0,1]\times\ex{Y}$ (see \cite[p.421]{ref_APS2}).  The first Chern class of $\ex{K}^{-1}\ex{E}^2\ex{L}_r^2$ is given by
\begin{align*}
\frac{i}{2\pi}(-F_{\ex{K}} + 2F_{\ex{E}} + 2\ex{F}_{r} -ir\dd\ex{a} - 2i\dd s\wedge\dd\theta)
= \frac{1}{2\pi}(-iF_{\ex{K}} + 2iF_{\ex{E}} + 2r\ex{\omega} + 2\dd s\wedge\dd\theta) ~.
\end{align*}
More precisely, the differential forms are pulled back by the projection map $[0,1]\times\ex{Y}\to\ex{Y}$ except $\dd s$.  It follows that
\begin{align}\label{eqn_ext_fj}
\fj &= \frac{r}{4\pi^2}\int_{\ex{Y}}\dd\theta\wedge\ex{\omega} + \frac{1}{8\pi^2}\int_{\ex{Y}}\dd\theta\wedge(2iF_{\ex{E}} - iF_{\ex{K}}) ~.
\end{align}
It is not hard to see that $\dd\theta\wedge\ex{\omega}>0$.  Thus, $\fj>0$ provided $r$ is sufficiently large.

\smallskip
(\emph{The spectral gap estimate})\;
The second step is to estimates the difference between $\lambda_j(0)$ and $\lambda_j(1)$.  Since $\fD_s$ is $\ex{D}_r$ perturbed by a closed $1$-form, $\fD_s$ obeys a similar Weitzenb\"ock formula as (\ref{eqn_ext_Weitzenbock}).  As a result, Proposition \ref{prop_ext_estimate} also holds for $\fD_s$ for any $s\in[0,1]$.

Let $\psi_j(s)$ be the unit-normed eigensection of $\CD_s$ with eigenvalue $\lambda_j(s)$.  Then
\begin{align}\label{eqn_ext_slope1}
\lambda'_j(s) &= \int_{\ex{Y}}\langle-i\cl(\dd\theta)\psi_j(s),\psi_j(s)\rangle ~,
\end{align}
and thus $|\lambda'_j(s)|\leq c_4$ for $c_4 = \sup_{\ex{Y}}|\dd\theta|$.  In particular, if $|\lambda_j(0)|^2\leq\frac{1}{4}r$, then $|\lambda_j(s)|^2\leq\frac{3}{4}r$ for any $s\in[0,1]$.  It follows that Proposition \ref{prop_ext_estimate} applies to $\psi_j(s)$ for all $s\in[0,1]$.  By (\ref{eqn_Reeb_01}), (\ref{eqn_ext_a}) and (\ref{eqn_ext_metric}),
\begin{align*}
\dd\theta &= V^{-1}\Omega\,\ex{a} + \big(\text{components in the metric dual of }\ker(\ex{a})\big) ~.
\end{align*}
Therefore,
\begin{align}\label{eqn_ext_slope2}
\langle-i\cl(\dd\theta)\psi_j,\psi_j\rangle = \frac{1}{V}(|\alpha_j|^2 - |\beta_j|^2) + \langle \fb(\alpha_j),\beta_j \rangle - \langle\fb^\dagger(\beta_j),\alpha_j\rangle
\end{align}
where $\alpha_j$ and $\beta_j$ are the $\ex{E}\ex{L}_r$ and $\ex{E}\ex{K}^{-1}\ex{L}_r$ components of $\psi_j$, respectively, and $\fb$ and $\fb^\dagger$ are the off-diagonal components of $-i\cl(\dd\theta)$.  The endomorphisms $\fb$ and $\fb^\dagger$ are independent of $r$.

According to (\ref{eqn_ext_slope1}), (\ref{eqn_ext_slope2}) and Proposition \ref{prop_ext_estimate}, there exists a constant $c_5$ such that
\begin{align*}
\text{if }|\lambda_j(0)|\leq\oh r^\oh\text{, then }|\lambda'_j(s) - \frac{1}{V}| \leq c_5r^{-\oh}\text{ for any } s\in[0,1] ~.
\end{align*}
Integrating this inequality from $s=0$ to $s=1$ gives
\begin{align*}
|\lambda_{j+\fj}(0) - \lambda_j(0) - \frac{1}{V}|\leq c_5r^{-\oh} ~.
\end{align*}  The inequality and (\ref{eqn_ext_fj}) complete the proof of theorem.
\end{proof}

The following corollary is a direct consequence of the theorem.

\begin{cor}\label{cor_Vafa_Witten}
There exists a constant $c_6$ determined by the stable Hamiltonian structure $(\ex{a}, \ex{\omega})$, the metric $\dd s^2$ and the connection $\aeE$ with the following significance.  Suppose that $r\geq c_6$, and $\lambda_-,\lambda_+\in[-\oh r^\oh,\oh r^\oh]$ are any two numbers with $\lambda_+-\lambda_-\geq\frac{1}{V}$.  Then, the total number of the eigenvalues (counting multiplicities) of $\ex{D}_r$ within $[\lambda_-,\lambda_+]$ is less than or equal to $c_5r(\lambda_+-\lambda_-)$.
\end{cor}

\section{Two Local Models}\label{sec_local_model}
The model $\ex{Y}$ constructed in \S\ref{sec_sh_model} is useful for analyzing eigensections of $(Y,D_r)$ on the mapping torus $\Sigma\times_\tau S^1$.  The main purpose of this section is to introduce two Dirac operators on $S^2\times S^1$.  They are useful for studying the eigensections of $(Y, D_r)$ near the bindings, and the eigensections of $(\ex{Y},\ex{D}_r)$ on the attaching handle.

\subsection{The local model for the open book}\label{subsec_local1}
The main object of the first model is a contact form on $S^2\times S^1$, which was introduced in \cite[\S4.3]{ref_Ts1}.  The model consists of the datum $(S^2\times S^1, \ck{a}, \dd s^2, \underline{\BC}\oplus\ck{K}^{-1},\ck{D}_r)$.  It will be denoted by $(\ck{S} = S^1\times S^1,\ck{D}_r)$ for brevity.

Let $(\rho,e^{i\theta})\in[0,2]\times S^1$ be the (re-parametrized) spherical\footnote{To be more precise, choose a positive smooth function $\chi(\rho)$ of $\rho\in[0,2]$ such that $\chi(\rho) = 1$ when $\rho\leq\frac{1}{10}$ or $\rho\geq\frac{19}{10}$, and $\int_0^2\chi(\rho)\dd\rho = \frac{\pi}{2}$.  The parametrization of the standard sphere is given by $x = \sin(\int_0^\rho\chi(s)\dd s)\cos\theta$, $y = \sin(\int_0^\rho\chi(s)\dd s)\sin\theta$, $z = \cos(\int_0^\rho\chi(s)\dd s)\cos\theta$.} coordinate for the $S^2$ factor, and let $e^{it}$ be the coordinate for the $S^1$ factor.  The orientation is determined by the $3$-form $\dd\rho\wedge\dd\theta\wedge\dd t$ for $\rho\in(0,2)$.

\subsubsection{The contact form and the adapted metric}
Choose two smooth functions $\ck{f}(\rho)$ and $\ck{g}(\rho)$ of $\rho\in[0,2]$ such that
\begin{itemize}
\item when $0\leq\rho<1+50\delta$, the functions $\ck{f}(\rho)$ and $\ck{g}(\rho)$ coincide with the functions $f(\rho)$ and $g(\rho)$ constructed in \S\ref{sec_TW_01};
\item when $2-50\delta\leq\rho\leq2$, $\ck{f}(\rho) = (2-\rho)^2$ and $\ck{g}(\rho) = -2 + (2-\rho)^2$;
\item for any $\rho\in(0,2)$, the functions $\ck{f}$ and $\ck{f}'\ck{g}-\ck{f}\ck{g}'$ are positive.
\end{itemize}
It is not hard to see that there always exist such $\ck{f}$ and $\ck{g}$.

With these two functions chosen, the $1$-form
\begin{align}
\ck{a} &= \ck{f}(\rho)\dd\theta + \ck{g}(\rho)\dd t
\end{align}
is a contact form on $S^2\times S^1$.  The metric
\begin{align}\label{eqn_metric_ck}
\dd s^2 &= \ck{a}^2 + (\dd\rho)^2 + \frac{1}{4}(\ck{f}'(\rho)\dd\theta + \ck{g}'(\rho)\dd t)^2
\end{align}
is adapted to the contact form $\ck{a}$.

\subsubsection{The Dirac operator}\label{sec_ck_Dirac_component}
As explained in \S\ref{subsec_can_spinc} and \S\ref{subsec_can_spinc_conn}, the contact form $\ck{a}$ and the adapted metric $\dd s^2$ determine a canonical spinor bundle $\underline{\BC}\oplus\ck{K}^{-1}$ and a canonical spin-c Dirac operator $\ck{D}_\can$ on it.  The bundle $\ck{K}^{-1}$ is also a trivial bundle.  It can be globally trivialized by the unit-normed section
\begin{align}\label{eqn_trivial_ck_01}
\frac{e^{-i\theta}}{\sqrt{2}}\big(\pl_\rho - \frac{2i}{\ck{f}'\ck{g}-\ck{f}\ck{g}'}(\ck{g}\pl_\theta - \ck{f}\pl_t)\big) ~.
\end{align}
Together with the depicted section $\bo_{\underline{\BC}}$, the sections of $\underline{\BC}\oplus\ck{K}^{-1}$ are identified with $\BC^2$ valued functions on $S^2\times S^1$.  With respect to this identification, let $\CS_{k,m}$ be the space of smooth sections whose $\underline{\BC}$ component have frequency $k$ in $e^{i\theta}$ and $m$ in $e^{it}$, and whose $\ck{K}^{-1}$ component have frequency $k+1$ in $e^{i\theta}$ and $m$ in $e^{it}$.  Namely,
\begin{align*}
\CS_{k,m} &= \big\{ \psi = (\alpha,\beta)\in\CC^\infty(S^2\times S^1; \underline{\BC}\oplus\ck{K}^{-1}) ~\big|~ \pl_\theta\psi = ik\psi + i(0,\beta) \text{ and } \pl_t\psi = im\psi \big\} ~.
\end{align*}

\begin{rmk}\label{rmk_identification_01}
When $0\leq\rho<1+50\delta$, the contact form and the metric are the same as that near the bindings of the open book, $S^1\times B\subset Y$.  It follows that $\bo_{\underline{\BC}}$ and (\ref{eqn_trivial_ck_01}) also trivialize the canonical spinor bundle $\underline{\BC}\oplus K^{-1}$ over $S^1\times B\subset Y$.  Together with an $\aE$-parallel, unit-normed section of $E$, the sections of $E\oplus EK^{-1}$ over $S^1\times B\subset Y$ are identified with $\BC^2$ valued functions.  This provides an identification of sections of $(E\oplus EK^{-1})|_{S^1\times B}$ with sections of $(\underline{\BC}\oplus\ck{K}^{-1})|_{S^1\times B}$.  Moreover, the Dirac operator $D_r$ is identified with $\ck{D}_r = \ck{D}_\can - \frac{ir}{2}\cl(\ck{a})$.
\end{rmk}

For any $r>0$, consider the Dirac operator $\ck{D}_r = \ck{D}_\can - \frac{ir}{2}\cl(\ck{a})$.  The following notion is useful to describe the eigenvalues of $\ck{D}_r$.

\begin{defn}\label{defn_Dirac_ck_01}
Observe that the function $\ck{g}/\ck{f}$ is monotone decreasing in $\rho\in[0,2]$.  For each positive integer $k$ and integer $m$, there is a unique $\check{\rho}_{k,m}\in(0,2)$ such that $k \ck{g}(\ck{\rho}_{k,m}) = m \ck{f}(\ck{\rho}_{k,m})$.  Let $\ck{\gamma}_{k,m}$ be
\begin{align}\label{eqn_Dirac_ck_01}  \check{\gamma}_{k,m} = \frac{2m\ck{f}'(\ck{\rho}_{k,m})-2k\ck{g}'(\ck{\rho}_{k,m})}{(\ck{f}'\ck{g}-\ck{f}\ck{g}')(\ck{\rho}_{k,m})} = \frac{2k}{\ck{f}(\ck{\rho}_{k,m})} = \frac{2m}{\ck{g}(\ck{\rho}_{k,m})} ~.  \end{align}
The last equality only makes sense at where $g(\check{\rho}_{k,m})\neq0$.  If $k=0$ and $m>0$, set $\ck{\rho}_{k,m}$ to be $0$, and set $\ck{\gamma}_{k,m}$ to be $m$.  If $k=0$ and $m<0$, set $\ck{\rho}_{k,m}$ to be $2$, and set $\ck{\gamma}_{k,m}$ to be $-m$.
\end{defn}

The spectral properties of $\ck{D}_r$ are summarized in the following proposition.

\begin{prop}\label{prop_Dirac_ck_01}
There exists a constant $c_1$ determined by the contact form $\ck{a}$ and the metric $\dd s^2$ with following significance.
\begin{enumerate}
\item $\ck{D}_r(\CS_{k,m})$ belongs to $\CS_{k,m}$ for any $k$ and $m$, and the eigenbasis of $\ck{D}_r$ can be chosen so that each eigensection belongs to some $\CS_{k,m}$.
\item For any $r\geq c_1$, $\ck{D}_r$ has at most one eigenvalue $\lambda$ within $(-(\frac{1}{3}r)^\oh,(\frac{1}{3}r)^\oh)$ on each $\CS_{k,m}$.  If there does exist such an eigenvalue, then $k>-\min\{1,|m|\}$ and
\begin{align*}\quad  \big|\lambda - \frac{r-\ck{\gamma}_{k,m}}{2}\big| \leq c_1 ~.  \end{align*}
Moreover, the corresponding eigensection can be expressed as $\ck{\varphi}_{k,m} = \ck{\varphi}^{\appr}_{k,m} + \ck{\varphi}^{\err}_{k,m}$ with the following properties.
\begin{enumerate}
\item $\ck{\varphi}_{k,m}$, $\ck{\varphi}^{\appr}_{k,m}$ and $\ck{\varphi}^{\err}_{k,m}$ are smooth sections.  If $20\delta<\ck{\rho}_{k,m}<2-20\delta$, the support of $\ck{\varphi}^{\appr}_{k,m}$ is contained in $\{|\rho - \ck{\rho}_{k,m}|\leq2\delta\}$.  If $\ck{\rho}_{k,m}\leq20\delta$, the support of $\ck{\varphi}^{\appr}_{k,m}$ is contained in $\{\rho\leq40\delta\}$.  If $\ck{\rho}_{k,m}\geq2-20\delta$, the support of $\ck{\varphi}^{\appr}_{k,m}$ is contained in $\{\rho\geq2-40\delta\}$.
\item The $L^2$ integrals satisfy
\begin{align*}
\qquad\quad\int_{\ck{S}}|\ck{\varphi}_{k,m}|^2 &= 1 ~,   &\int_{\ck{S}}\langle\ck{\varphi}_{k,m},\ck{\varphi}^{\err}_{k,m}\rangle &= 0 ~,   &\int_{\ck{S}}|\ck{\varphi}^{\err}_{k,m}|^2&\leq c_1r^{-7}
\end{align*}
where the integrals are against the volume form $\oh\ck{a}\wedge\dd\ck{a}$.
\item $\ck{\varphi}^{\appr}_{k,m}$ is an approximate eigensection in the sense that
\begin{align*}
\qquad \int_{\ck{S}}|\ck{D}_r\ck{\varphi}^{\appr}_{k,m} - \lambda\ck{\varphi}^{\appr}_{k,m}|^2 &\leq c_1r^{-6} ~.
\end{align*}
\end{enumerate}
\item On the other hand, for any positive integer $k$ and integer $m$ with $|r-\ck{\gamma}_{k,m}|^2\leq\frac{1}{3}r$,  $\ck{D}_r$ on $\CS_{k,m}$ does admit an eigenvalue $\lambda$ satisfying
\begin{align*}\quad  \big|\lambda - \frac{r-\ck{\gamma}_{k,m}}{2}\big| \leq c_1 ~.  \end{align*}
\end{enumerate}
\end{prop}

This proposition was proved in \cite[\S5]{ref_Ts1}.  We will give the precise reference of each assertion in \S\ref{subsec_ap1}.

\subsection{The local model for the mapping torus}\label{subsec_local2}
The main object of the second model is a stable Hamiltonian structure on $S^2\times S^1$.  The model consists of the datum $(S^2\times S^1,\hat{a},\hat{\omega},\dd s^2, \hat{L}_r\oplus\hat{L}_r\hat{K}^{-1},\hat{D}_r)$.  It will be denoted by $(\hat{S} = S^2\times S^1,\hat{D}_r)$ for brevity.

Let $(\rho,e^{it})\in[0,2]\times S^1$ be the (re-parametrized) spherical coordinate for the $S^2$ factor.  Let $e^{i\theta}$ be the coordinate for the $S^1$ factor of $S^2\times S^1$.  Note that the roles of $\theta$ and $t$ are switched from the convention in \S\ref{subsec_local1}.

\subsubsection{Geometric structures}
Choose a smooth function $\hat{g}(\rho)$ such that
\begin{itemize}
\item when $\rho<1+50\delta$, $\hat{g}(\rho)$ coincides with the function $\ex{g}(\rho)$ defined in \S\ref{subsec_ext_oneform};
\item when $\rho\geq2-50\delta$ , $\hat{g}(\rho)$ is equal to $0$.
\end{itemize}
For any $\sigma\in(\frac{9}{10},\frac{11}{10})$, choose a smooth function $\hat{h}_\sigma(\rho)$ such that
\begin{itemize}
\item when $\rho<1+50\delta$, $\hat{h}_\sigma(\rho)$ coincides with the function $\ex{h}_\sigma(\rho)$ defined in \S\ref{subsec_ext_metric};
\item when $\rho\geq2-50\delta$, $\hat{h}_\sigma(\rho) = 2\sigma - \oh(\rho-2)^2$;
\item its derivative in $\rho$ is positive $\hat{h}_\sigma'(\rho)>0$, at any $\rho\in(0,2)$.
\end{itemize}

With these two functions chosen, the Riemannian metric on $S^2\times S^1$ is taken to be
\begin{align}
\dd s^2 &= \hat{a}^2 + (\dd\rho)^2 + (\hat{h}'_\sigma(\rho)\dd t)^2
\end{align}
where  $$    \hat{a} = V\dd\theta + \hat{g}(\rho)\dd t ~.    $$    The orientation on $S^2$ is determined by $$    \hat{\omega} = \hat{h}'_\sigma(\rho)\dd t\wedge\dd\rho ~,    $$ and the orientation on $S^2\times S^1$ is determined by $\hat{\omega}\wedge\dd\theta = \hat{h}'_\sigma(\rho)\dd\rho\wedge\dd\theta\wedge\dd t$.

The pair $(\hat{a},2\hat{\omega})$ constitutes a stable Hamiltonian structure.  The metric $\dd s^2$ is adapted to it.  According to \S\ref{subsec_can_spinc} and \S\ref{subsec_can_spinc_conn}, they determine a canonical spinor bundle $\underline{\BC}\oplus\hat{K}^{-1}$ with a canonical connection ${\nabla}_\can$.

The symmetric $2$-tensor $(\dd\rho)^2 + (\hat{h}'_\sigma(\rho)\dd t)^2$ defines a metric on $S^2$.  The metric with the symplectic form $\hat{\omega}$ determines a complex structure on $S^2$.  Let $K_{S^2}^{-1}$ be the anti-canonical bundle.  It is not hard to see that $\hat{K}^{-1}$ is isomorphic to the pull-back of $K^{-1}_{S^2}$ by the projection map.

\subsubsection{The degree $2r$ bundle}
Suppose that $r>20$.  Set $\sigma$ to be $\frac{[r]}{r}$ as before.  In order to study the Dirac operator introduced in \S\ref{subsec_ext_Lr}, consider the Hermitian line bundle $\hat{L}_{r}$ over $S^2$ defined as follows.  Take the trivial bundles over $\{0\leq\rho<1+50\delta\}$ and $\{1<\rho\leq2\}$.  Let $\bo_-$ and $\bo_+$ be the depicted unitary sections, respectively.  On the overlap region $\{1<\rho<1+50\delta\}$, identify the bundles by the transition rule
$$    e^{2i[r]t}\cdot\bo_+ = \bo_- ~.    $$

Define a unitary connection $\ahr$ on $\hat{L}_{r}$ as the following:
\begin{itemize}
\item on $\{1\leq\rho<1+50\delta\}$, the connection $\ahr$ is $\dd + ir(\hat{h}_\sigma(\rho) + \oh\hat{g}(\rho))\dd t$ with respect to the trivialization $\bo_-$;
\item on $\{1<\rho\leq2\}$, the connection $\ahr$ is $\dd + ir(\hat{h}_\sigma(\rho) + \oh\hat{g}(\rho) - 2\sigma)\dd t$ with respect to the trivialization $\bo_+$;  note that the term $-2\sigma$ guarantees the smoothness of the connection near $\rho=2$.
\end{itemize}
On the overlap region $\{1<\rho<1+50\delta\}$, the first expression is $\dd + i[r]\dd t$, and the second expression is $\dd - i[r]\dd t$.  They obey the transition rule of $\hat{L}_{r}$, and hence $\ahr$ is a smooth connection on $\hat{L}_{r}$.  It is easy to see that the first Chern number of $\hat{L}_{r}$ is $2[r]$, either from the transition rule or the curvature computation.

\subsubsection{The Dirac operator on $\hat{L}_{r}\oplus\hat{L}_{r}\hat{K}^{-1}$}
Consider the connection and the Dirac operator
\begin{align*}\left\{\begin{aligned}
\hat{\nabla}_r &= {\nabla}_\can\otimes\ahr - \frac{ir}{2}\hat{a}  \\
\hat{D}_r &= \cl\circ({\nabla}_\can\otimes\ahr - \frac{ir}{2}\hat{a})
\end{aligned}\right.\end{align*}
on $\hat{L}_{r}\oplus\hat{L}_{r}\hat{K}^{-1}$.  Here $\hat{L}_r$ is pulled back as a bundle over $S^2\times S^1$.

To study the Dirac spectrum, consider the $S^1$ action on the $S^1$ factor of $S^2\times S^1$.  Namely,
$$    e^{i\theta'}\cdot(\rho,e^{it},e^{i\theta})\mapsto(\rho,e^{it},e^{i(\theta'+\theta)}) ~.    $$
This action preserves the stable Hamiltonian structure and the metric.  It does not change the projection map onto the $S^2$ factor.

With the help of the $S^1$ action, the Dirac equation is reduced to a Cauchy--Riemann equation on $S^2$.  Suppose that
\begin{align}
\hat{\varphi} &= \frac{1}{\sqrt{2\pi}}e^{ik\theta}(\hat{\alpha}_k,\hat{\beta}_k)
\end{align}
is an eigensection of $\hat{D}_r$ with eigenvalue $\lambda$, where $\hat{\alpha}_k$ is a section of $\hat{L}_r\to S^2$ and $\hat{\beta}_k$ is a section of $\hat{L}_r K^{-1}_{S^2}\to S^2$.  To derive the equation for $\hat{\alpha}_k$ and $\hat{\beta}_k$, trivialize the bundles $\hat{L}_r$ and $\hat{L}_r\hat{K}^{-1}$ over $\{0\leq\rho<2\}$ by the sections
\begin{align}\label{eqn_hat_trivial_K}
\bo_- \quad\text{and}\quad \bo_-\otimes\frac{1}{\sqrt{2}}e^{it} \big( \pl_\rho + \frac{i}{\hat{h}'_{\sigma}}(\pl_t-\frac{\hat{g}}{V}\pl_\theta) \big) ~,
\end{align}
respectively.  The latter expression can be identified with $\frac{1}{\sqrt{2}}e^{it}(\dd\rho + i\hat{h}'_\sigma\dd t)$, and defines a local section of $K^{-1}_{S^2}$.  With respect to this trivialization, the Dirac operator reads
\begin{align}\label{eqn_hat_Dirac_00}\begin{split}
\hat{D}_r &=\begin{bmatrix}
\frac{i}{V}\pl_\theta & -e^{it}\big(\pl_\rho + \frac{i}{\hat{h}'_\sigma}(\pl_t - \frac{\hat{g}}{V}\pl_\theta)\big) \smallskip\\
e^{-it}\big(\pl_\rho - \frac{i}{\hat{h}'_\sigma}(\pl_t - \frac{\hat{g}}{V}\pl_\theta)\big) & -\frac{i}{V}\pl_\theta
\end{bmatrix} \\
&\quad + r\begin{bmatrix}
\frac{1}{2} & \frac{e^{it}}{\hat{h}'_\sigma}(\hat{h}_\sigma+\oh\hat{g}) \smallskip\\
\frac{e^{-it}}{\hat{h}'_\sigma}(\hat{h}_\sigma+\oh\hat{g}) & -\frac{1}{2}
\end{bmatrix}    + \frac{1}{2\hat{h}'_\sigma}\begin{bmatrix}
0 & 2e^{it}(1-2\hat{h}''_\sigma)  \\ 0 & \hat{g}'
\end{bmatrix} ~.
\end{split}\end{align}
It follows that the Dirac equation $\hat{D}_r\hat{\varphi} = \lambda\hat{\varphi}$ reduces to the following equations on $S^2$:
\begin{align}\label{eqn_hat_Dirac_01}\left\{\begin{aligned}
&(\frac{r}{2}-\frac{k}{V})\hat{\alpha}_k + \sqrt{2}\dbrs_{r,k}\hat{\beta}_k = \lambda\hat{\alpha}_k ~, \\
&\sqrt{2}\dbr_{r,k}\hat{\alpha}_k - (\frac{r}{2}-\frac{k}{V} - \frac{\hat{g}'}{2\hat{h}'_\sigma})\hat{\beta}_k = \lambda\hat{\beta}_k ~.
\end{aligned}\right.\end{align}
Here, $\dbr_{r,k}:\hat{L}_r\to\hat{L}_rK^{-1}_{S^2}$ is the usual Cauchy--Riemann operator associated with
$$   \hat{\nabla}_{r,k} = \ahr - i\frac{k}{V}\hat{g}\,\dd t ~.   $$  In other words, perturb the connection $\ahr$ by the globally defined $1$-form $-i\frac{k}{V}\hat{g}\,\dd t$.  The operator $\dbrs_{r,k}$ is the $L^2$-adjoint operator of $\dbr_{r,k}$.

With this reduction, the eigenvalues of $\hat{D}_r$ can be found by the Riemann--Roch formula and the vanishing argument.

\begin{prop}\label{prop_Dirac_hat_00}
There exists a constant $c_2$ determined by the stable Hamiltonian structure $(\hat{a}, \hat{\omega})$ and the metric $\dd s^2$ such that the following holds.
\begin{enumerate}
\item For any $r\geq c_2$, the spectrum of $\hat{D}_r$ on $\hat{L}_r\oplus\hat{L}_r\hat{K}^{-1}$ that lies within $[-(\frac{1}{3} r)^\oh,(\frac{1}{3} r)^\oh]$ consists of
\begin{align*}
\big\{\frac{r}{2} - \frac{k}{V} ~\big|~ k\in\BZ ~,\text{ and }~ |\frac{r}{2} - \frac{k}{V}|\leq(\frac{1}{3} r)^\oh \big\} ~.
\end{align*}
\item For any $r\geq c_2$ and any integer $k$ satisfying $|\frac{r}{2} - \frac{k}{V}|\leq(\frac{1}{3} r)^\oh$, the corresponding eigenspace has dimension $2[r]+1$, and is isomorphic to $\ker\dbr_{r,k}$ via the identification
\begin{align*}
\hat{\alpha}_k\in\ker\dbr_{r,k} &\mapsto \frac{1}{\sqrt{2\pi}}e^{ik\theta}(\hat{\alpha}_k,0) ~.
\end{align*}
\end{enumerate}
\end{prop}
\begin{proof}
Suppose that $\hat{\psi} = \frac{1}{2\pi}e^{ik\theta}(\hat{\alpha}_k,\hat{\beta}_k)$ is an eigensection of $\hat{D}_r$ whose eigenvalue $\lambda$ lies within $[-(\ot r)^{\oh},(\ot r)^\oh]$.  Integrating the Bochner--Weitzenb\"ock formula for $\beta_k$ gives
\begin{align*}
2\int_{S^2}|\dbrs_{r,k}\hat{\beta}_k|^2 &= \int_{S^2}|\hat{\nabla}_{r,k}\hat{\beta}_k|^2 + \int_{S^2}(\frac{iF_{\hat{A}_{r,k}}}{\hat{\omega}} + \kappa)|\hat{\beta}_k|^2
\end{align*}
where $\kappa$ is the Gaussian curvature.  The left hand side is equal to $2\int_{S^2}\langle\dbr_{r,k}\dbrs_{r,k}\hat{\beta}_k,\hat{\beta}_k\rangle$.  Using (\ref{eqn_hat_Dirac_01}) to replace $\dbr_{r,k}\dbrs_{r,k}\hat{\beta}_k$, the equation becomes
\begin{align*}
0 = \int_{S^2}|\hat{\nabla}_{r,k}\hat{\beta}_k|^2 + \int_{S^2}\big( (\frac{r}{2}-\frac{k}{V})^2 + (r-\lambda^2+\frac{\hat{g}'}{\hat{h}'_\sigma}\lambda) + \kappa \big)|\hat{\beta}_k|^2 ~.
\end{align*}
Since $|\lambda|\leq(\frac{1}{3} r)^{\oh}$, there exists a constant $c_3$ such that $\hat{\beta}_k$ vanishes for any $r\geq c_3$.  It follows from $\hat{\beta}_k\equiv0$ and (\ref{eqn_hat_Dirac_01}) that
\begin{align*}\left\{\begin{aligned}
|\frac{r}{2}-\frac{k}{V}|&\leq (\frac{1}{3} r)^{\oh} ~,\\
\dbr_{r,k}\hat{\alpha}_k &=0 ~.
\end{aligned}\right.\end{align*}

According to the Riemann--Roch formula, it suffices to show that the kernel of $\dbrs_{r,k}$ is trivial to conclude that the dimension of $\ker\dbr_{r,k}$ is $2[r]+1$.  The vanishing of $\ker\dbrs_{r,k}$ follows form the same Bochner--Weitzenb\"ock formula and the condition $|\frac{r}{2}-\frac{k}{V}|\leq (\frac{1}{3} r)^{\oh}$.  This finishes the proof of the lemma.
\end{proof}

We need the following notion to describe the eigensections of $\hat{D}_r$.

\begin{defn}\label{defn_Dirac_hat_01}
There exists a constant $c_4$ such that for any $r\geq c_4$ and $|\frac{r}{2}-\frac{k}{V}|\leq (\frac{1}{3} r)^{\oh}$,
\begin{align*}
r\hat{h}'_\sigma + (\frac{r}{2} - \frac{k}{V})\hat{g}' > 0
\end{align*}
for any $\rho\in(0,2)$.  For any $r\geq c_4$ and any integer $n\in(-2[r],0)$, let $\hat{\rho}_{k,n}\in(0,2)$ be the unique solution of
$$r(\hat{h}_\sigma(\hat{\rho}_{k,n})+\oh\hat{g}(\hat{\rho}_{k,n})) - \frac{k}{V}\hat{g}(\hat{\rho}_{k,n}) + n = 0 ~.$$
For $n=0$, set $\hat{\rho}_{k,n}$ to be $0$.  For $n=-2[r]$, set $\hat{\rho}_{k,n}$ to be $2$.
\end{defn}

\begin{prop}\label{prop_Dirac_hat_01}
There exists a constant $c_5$ determined by the stable Hamiltonian structure $(\hat{a}, \hat{\omega})$ and the metric $\dd s^2$ with the following significance.  For any $r\geq c_5$ and any integer $k$ with $|\frac{r}{2}-\frac{k}{V}|\leq(\frac{1}{3} r)^\oh$, $\ker\dbr_{r,k}$ has an orthonormal basis $\{\hat{\alpha}_{k,n} = \hat{\alpha}_{k,n}^{\appr} + \hat{\alpha}_{k,n}^{\err}\}_{-2[r]\leq n\leq0}$ satisfying the following properties.
\begin{enumerate}
\item With respect to the trivialization $\bo_-$, $\pl_t \hat{\alpha}_{k,n} = in\,\hat{\alpha}_{k,n}$, so do $\hat{\alpha}_{k,n}^{\appr}$ and $\hat{\alpha}_{k,n}^{\err}$.
\item If $20\delta<\hat{\rho}_{k,n}<2-20\delta$, the support of $\hat{\alpha}^{\appr}_{k,n}$ is contained in $\{|\rho - \hat{\rho}_{k,n}|\leq2\delta\}$.  If $\hat{\rho}_{k,n}\leq20\delta$, the support of $\hat{\alpha}^{\appr}_{k,n}$ is contained in $\{\rho\leq40\delta\}$.  If $\ck{\rho}_{k,n}\geq2-20\delta$, the support of $\hat{\alpha}^{\appr}_{k,m}$ is contained in $\{\rho\geq2-40\delta\}$.
\item $\hat{\alpha}_{k,n}^{\appr}$ almost solves $\dbr_{r,k}$ in the sense that
\begin{align*}
\qquad \int_{S^2}|\dbr_{r,k}(\hat{\alpha}_{k,n}^{\appr})|^2 &\leq c_5e^{-\frac{r}{c_5}} < c_5r^{-6} ~.
\end{align*}
\item The remainder term $\hat{\alpha}_{k,n}^{\err}$ obeys
\begin{align*}
\int_{S^2}|\hat{\alpha}_{k,n}^{\err}|^2\leq c_5e^{-\frac{r}{c_5}}<c_5r^{-7} \qquad\text{and}\qquad \int_{S^2}\langle\hat{\alpha}_{k,n},\hat{\alpha}_{k,n}^{\err}\rangle = 0 ~.
\end{align*}
\item If $|n-\frac{k}{V}+[r]|<\frac{k}{V}(48\delta)$, the support of $\hat{\alpha}_{k,n}^{\appr}$ is contained in $\{|\rho-1|<50\delta\}$, and thus can be regarded as a smooth function on $S^2$ (with respect to $\bo_-$).  It is proportional to the approximate eigensection $\ck{\varphi}_{k,n+[r]}^{\appr}$ given by Proposition \ref{prop_Dirac_ck_01}.  More precisely,
\begin{align*}
\ck{\varphi}_{k,n+[r]}^{\appr} = \frac{1}{\sqrt{2\pi}}e^{i(k\theta+[r]t)}(\fc_{k,n}\hat{\alpha}_{k,n}^{\appr},0)
\end{align*}
for some scalar $\fc_{k,n}$ with $|\fc_{k,n}-1|\leq c_5e^{-\frac{r}{c_5}}$.  Here $\ck{\varphi}_{k,n+[r]}^{\appr}$ is identified with a $\BC^2$ valued function by (\ref{eqn_trivial_ck_01}).
\end{enumerate}
\end{prop}

The condition of (iv) $|n-\frac{k}{V}+[r]|<(48k\delta)/V$ transforms to $|m-\frac{k}{V}|<(48k\delta)/V$ by $m=n+[r]$.  They are equivalent to that $|\hat{\rho}_{k,n}-1|<48\delta$ and $|\ck{\rho}_{k,m}-1|<48\delta$, respectively.  The proof of the proposition is basically by solving ordinary differential equations with integral factor.  The proof appears in \S\ref{subsec_ap2}.

\section{Gluing Eigensections}\label{sec_gluing}
The main purpose of this section is to prove that the ``small eigenvalues" of $(Y,D_r)$ is almost the same as that of $(\ex{Y},\ex{D}_r)$.  The strategy is to divide $[-\oh r^\oh, \oh r^\oh]$ into sub-intervals about of unit length, and to show that the total number of eigenvalues of $D_r$ and $\ex{D}_r$ are same within each sub-interval.

\begin{lem}\label{lem_glue_01}
There exist constants $c_1$ and $c_2$ determined by the contact form $a$, the metric $\dd s^2$, and the connection $\aE$ with the following property.  For any $r\geq c_1$, there exists a sequence of numbers $\{\nu_j : -[\oh r^\oh]<j<[\oh r^\oh]\}$ such that
\begin{enumerate}
\item for any $j$, $|\nu_j - j|\leq\frac{1}{10}$, and thus $-\oh r^\oh < \cdots < \nu_j < \nu_{j+1} < \cdots < \oh r^\oh$;
\item for any $j$, there is no spectrum within $[\nu_j - c_2 r^{-1}, \nu_j + c_2 r^{-1}]$ for $(\ex{Y},\ex{D}_r)$, $(\ck{S},\ck{D}_r)$ and $(\hat{S},\hat{D}_r)$;
\item for any $j$, there is no spectrum within $[\nu_j - c_2 r^{-\frac{3}{2}},\nu_j + c_2r^{-\frac{3}{2}}]$ for $(Y,D_r)$.
\end{enumerate}
\end{lem}

\begin{proof}
For any $j\in\{-[\oh r^\oh]+1,-[\oh r^\oh]+2,\cdots,[\oh r^\oh]-1\}$, consider the interval
$$ U_j = [j-\frac{1}{15},j+\frac{1}{15}] ~. $$
\begin{itemize}
\item According to Corollary \ref{cor_Vafa_Witten} and Proposition \ref{prop_Dirac_hat_00}, there exists a constant $c_3$ such that the total number of eigenvalues of $(\ex{Y},\ex{D}_r)$ and $(\hat{S},\hat{D}_r)$ within $U_j$ is less than $c_3 r$.
\item Since $(\ck{S},\ck{D}_r)$ is constructed from a contact form with an \emph{adapted} metric (\ref{eqn_metric_ck}), \cite[Corollary 3.4]{ref_Ts2} implies that the total number of eigenvalues of $(\ck{S},\ck{D}_r)$ within $U_j$ is less than $c_4 r$ for some constant $c_4$.
\item Due to \cite[Corollary 3.3(i)]{ref_Ts2}, there exists a constant $c_5$ such that the total number of eigenvalues of $(Y,D_r)$ within $[-\oh r^\oh, \oh r^\oh]$ is less than $c_5 r^{\frac{3}{2}}$.  It follows that the total number of eigenvalues within $U_j$ is bounded by $c_5 r^{\frac{3}{2}}$.  The metric is only conformally adapted, and \cite[Corollary 3.4]{ref_Ts2} cannot apply to $(Y,D_r)$.
\end{itemize}
Let $c_6 = \max\{c_3,c_4\}$.  Divide $U_j$ into sub-intervals of length between $(60 c_6 r)^{-1}$ and $(30 c_6 r)^{-1}$.  There are at least $4c_6 r$ sub-intervals.  Let $\{U_{j,k}\}_{1\leq k\leq K}$ be the sub-intervals which do not contain any eigenvalues of $(\ex{Y},\ex{D}_r)$, $(\ck{S},\ck{D}_r)$ and $(\hat{S},\hat{D}_r)$.  It follows from the pigeonhole principle that $K\geq 3c_6 r$.

Let $\oh U_{j,k}$ be the sub-interval of $U_{j,k}$ with the same midpoint and of half length.  Further divide each $\oh U_{j,k}$ into sub-intervals of length between $(60 c_5)^{-1} r^{-\frac{3}{2}}$ and $(50 c_5)^{-1} r^{-\frac{3}{2}}$.  The total number of sub-intervals is at least
$$ (3c_6 r)\times\frac{1}{120 c_6 r}\times (50 c_5 r^{\frac{3}{2}}) > c_5r^{\frac{3}{2}} ~. $$
Hence, there exists some sub-interval which does contain any eigenvalue of $(Y,D_r)$.  Choose any one of such a sub-interval, and set $\nu_j$ to be its midpoint.  It follows from the construction that $\nu_j$ satisfies the assertion of the lemma with $c_2 = \frac{1}{500}(\max\{c_5,c_6\})^{-1}$.
\end{proof}

Item (ii) of the lemma guarantees spectral gaps of $2c_2r^{-1}$.  This allows us to invert the Dirac operator.  It plays a key role for gluing eigensections.

\begin{defn}\label{defn_glue_01}
Let $c_1$ be the same constant of Lemma \ref{lem_glue_01}.  For any $r\geq c_1$, let $\{\nu_j : -[\oh r^\oh]<j<[\oh r^\oh]\}$ be the sequence given by this same lemma.  With this sequence, introduce the following sets of eigenvalues for any $r\geq c$ and $-[\oh r^\oh]<j<[\oh r^\oh]-1$:
\begin{align*}
\CI_j &= \big\{\lambda\in\spec(D_r) ~\big|~ \nu_j<\lambda<\nu_{j+1}\big\} ~, \\
\ex{\CI}_j &= \big\{\lambda\in\spec(\ex{D}_r) ~\big|~ \nu_j<\lambda<\nu_{j+1}\big\} ~.
\end{align*}
The definition abuses\footnote{The set-theoretically correct definition is $\big\{(\lambda,k)\in\BR\times\BN ~\big|~ \lambda\in\spec(D_r),~\nu_j<\lambda<\nu_{j+1},~ k\leq\dim\ker(D_r-\lambda\BI)\big\}$.} the notation: the multiplicity of eigenvalues is counted.  Also introduce the following index sets:
\begin{align*}
\ck{\CI}_j &= \big\{(k,m)\in\BZ_{\geq0}\times\BZ ~\big|~ \spec(\ck{D}_r|_{\CS_{k,m}})\cap(\nu_j,\nu_{j+1})\neq\varnothing,\text{ and }k<mV\big\} ~, \\
\hat{\CI}_j &= \big\{(k,n)\in\BZ\times\BZ ~\big|~ \nu_j<\frac{r}{2}-\frac{k}{V}<\nu_{j+1},\text{ and }\frac{k}{V}-[r]<n\leq0 \big\} ~.
\end{align*}
The condition that $k<mV$ is equivalent to $\ck{\rho}_{k,m}<1$.  The condition that $k<(n+[r])V$ is equivalent to $\hat{\rho}_{k,n}<1$.
\end{defn}

As explained in the proof of Lemma \ref{lem_glue_01}, there exists a constant $c_7$ such that
\begin{align}\label{eqn_glue_14}
\#\ex{\CI}_j + \#\ck{\CI}_j + \#\hat{\CI}_j &\leq c_7r    &\text{and}&    &\#\CI_j &\leq c_7r^{\frac{3}{2}}
\end{align}
for any $r\geq c_7$ and $-[\oh r^\oh]<j<[\oh r^\oh]-1$.

\begin{thm}\label{thm_glue_01}
There exists a constant $c_8>c_1$ determined by the contact form $a$, the metric $\dd s^2$ and the connection $\aE$ with the following property.  For any $r\geq c_8$, let $\{\nu_j : -[\oh r^\oh]<j<[\oh r^\oh]\}$ be the sequence given by Lemma \ref{lem_glue_01}.  Then,
\begin{align*}
\#\CI_j - \#\ck{\CI}_j = \#\ex{\CI}_j - \#\hat{\CI}_j
\end{align*}
for any $r\geq c_8$ and $-[\oh r^\oh]<j<[\oh r^\oh]-1$.  As a consequence, $\#\CI_j\leq 2c_7r$.
\end{thm}

The proof of this theorem occupies the rest of this section.  It is organized as follows.  \S\ref{subsec_glue_01} is devoted to construct approximate eigensections of $(Y,D_r)$ corresponding to $\#\CI_j - \#\ck{\CI}_j$.  In \S\ref{subsec_glue_02} we prove that these approximate eigensections have small $L^2$-norm near the bindings, and hence are approximate eigensections of $(\ex{Y},\ex{D}_r)$.  \S\ref{subsec_glue_03} contains a linear algebra lemma which gives a precise estimate on the difference between genuine eigenvalues and approximate eigenvalues.  In \S\ref{subsec_glue_04} we combine the above results to prove that $(\#\CI_j - \#\ck{\CI}_j) + \#\hat{\CI}_j \leq \#\ex{\CI}_j$.  Another direction, $(\#\ex{\CI}_j - \#\hat{\CI}_j) + \#\ck{\CI}_j \leq \#\CI_j$, can be proved by the same argument.

\subsection{Approximate eigensections for $\#\CI_j - \#\ck{\CI}_j$}\label{subsec_glue_01}
The first step is to construct approximate eigensections of $D_r$ corresponding to $\#\CI_j - \#\ck{\CI}_j$.

\begin{lem}\label{lem_glue_02}
There exist constants $c_9>c_1$ and $c_{10}$ determined by the contact form $a$, the metric $\dd s^2$ and the connection $\aE$ with the following significance.  For any $r\geq c_9$, let $\{\nu_j : -[\oh r^\oh]<j<[\oh r^\oh]\}$ be the sequence given by Lemma \ref{lem_glue_01}.  Let
\begin{align*}
\CV_j &= \spn\{\psi \,|\, D_r\psi=\lambda\psi\text{ for some }\lambda\in\CI_j\}\qquad \text{and}  \\
\ck{\CV}_j &= \spn\{ \ck{\varphi}_{k,m}^{\appr} \,|\, (k,m)\in\ck{\CI}_j\}
\end{align*}
where $\ck{\varphi}^{\appr}_{k,m}$ is the  approximate eigensection given by Proposition \ref{prop_Dirac_ck_01}.  Since the elements of $\ck{\CV}_j$ only support on $\{\rho\leq1\}$, they can be regarded as smooth sections of $E\oplus EK^{-1}\to Y$ by (\ref{eqn_trivial_ck_01}) and Remark \ref{rmk_identification_01}.  Then, $\CV_j$ and $\ck{\CV}_j$ satisfy the following properties for any $-[\oh r^\oh]<j<[\oh r^\oh]-1$.
\begin{enumerate}
\item Let $\pr_j$ be the $L^2$-orthogonal projection onto $\CV_j$.  The dimension of $\pr_j(\ck{\CV}_j)$ is the same as the dimension of $\ck{\CV}_j$, i.e.\ $\pr_j:\ck{\CV}_j\to{\CV}_j$ is injective.
\item Let $\ck{\CV}_j^\prp$ be the $L^2$-orthogonal complement of $\pr_j(\ck{\CV}_j)$ in $\CV_j$.  The space $\ck{\CV}_j^\prp$ admits a $L^2$-orthonormal basis $\{\psi_{j,\ell}\}_{1\leq\ell\leq\#\CI_j - \#\ck{\CI}_j}$ such that
\begin{align*}
D_r\psi_{j,\ell} = \mu_{j,\ell}\psi_{j,\ell} + {\psi}_{j,\ell}^{\err}
\end{align*}
for some scalar $\mu_{j,\ell}\in\BR$ and $\psi^{\err}_{j,\ell}\in\pr_j(\ck{\CV}_j)$.  They obey the estimate:
\begin{align*}
\nu_j+c_2r^{-\frac{3}{2}}<\mu_{j,\ell}<\nu_{j+1}-c_2r^{-\frac{3}{2}} ~, \text{ and }~ \int_Y|\psi^{\err}_{j,\ell}|^2\leq c_{10}r^{-6}
\end{align*}    for any $1\leq\ell\leq\#\CI_j - \#\ck{\CI}_j$.
\item Moreover, let $\pr_j^{-1}:\pr_j(\ck{\CV}_j)\to\ck{\CV}_j$ be the inverse map of that in item \textup{(i)}.  Then
\begin{align*}
\int_Y|\pr_j^{-1}(\psi^{\err}_{j,\ell}) - \psi^{\err}_{j,\ell}|^2 &\leq c_{10}r^{-8} ~.
\end{align*}
\end{enumerate}
\end{lem}

\begin{proof}
(\emph{Assertion} (i): $\pr(\ck{\CV}_j)$)\;
For any $(k,m)\in\ck{\CI}_j$, let $\lambda_{k,m}$ be the corresponding eigenvalue of $\ck{D}_r$ on $\ck{\varphi}_{k,m}$.  According to Lemma \ref{lem_glue_01}(ii), $\nu_j+c_2r^{-1}<\lambda_{k,m}<\nu_{j+1}-c_2r^{-1}$.  Due to Remark \ref{rmk_identification_01} and (\ref{eqn_trivial_ck_01}), $D_r\ck{\varphi}_{k,m}^{\appr}$ on $Y$ is the same as $\ck{D}_r\ck{\varphi}_{k,m}^{\appr}$ on $\ck{S}$.  By Proposition \ref{prop_Dirac_ck_01}(ii.b) and (ii.c), there exists a constant $c_{11}$ such that
\begin{align}\label{eqn_glue_01}
\int_Y|D_r\ck{\varphi}^{\appr}_{k,m} - \lambda_{k,m}\ck{\varphi}^{\appr}_{k,m}|^2 &\leq c_{11}r^{-6}  ~,
&\text{and}& &\big|1 - \int_Y|\ck{\varphi}^{\appr}_{k,m}|^2\big| &\leq c_{11}r^{-7}  ~.
\end{align}

In terms of the spectral decomposition induced by $D_r$, write $\ck{\varphi}^{\appr}_{k,m}$ as
\begin{align}\label{eqn_glue_02}
\ck{\varphi}^{\appr}_{k,m} &= \pr_j(\ck{\varphi}^{\appr}_{k,m}) +  \ck{\varphi}^{+}_{k,m} + \ck{\varphi}^{-}_{k,m}
\end{align}
where $\pr_j(\ck{\varphi}^{\appr}_{k,m})$ is the $L^2$-orthogonal projection of $\ck{\varphi}^{\appr}_{k,m}$ onto $\CV_j$, $\ck{\varphi}^{+}_{k,m}$ is the $L^2$-orthogonal projection of $\ck{\varphi}^{\appr}_{k,m}$ onto the space spanned by eigensections whose eigenvalue is greater than $\nu_{j+1}$, and $\ck{\varphi}^{-}_{k,m}$ is the $L^2$-orthogonal projection of $\ck{\varphi}^{\appr}_{k,m}$ onto the space spanned by eigensections whose eigenvalue is less than $\nu_{j}$.  It follows from Lemma \ref{lem_glue_01}(iii) that
\begin{align*}
c_2r^{-\frac{3}{2}}\int_Y|\ck{\varphi}^{+}_{k,m}|^2 &\leq \int_Y\langle (D_r-\lambda_{k,m})\ck{\varphi}^{\appr}_{k,m}, \ck{\varphi}^{+}_{k,m}\rangle  \\
&\leq \frac{1}{2c_2}r^{\frac{3}{2}}\int_Y |(D_r-\lambda_{k,m})\ck{\varphi}^{\appr}_{k,m}|^2 + \oh c_2r^{-\frac{3}{2}}\int_Y|\ck{\varphi}^{+}_{k,m}|^2 ~.
\end{align*}
Then appeal to (\ref{eqn_glue_01}) to conclude that
\begin{align}\label{eqn_glue_03}
\int_Y|\ck{\varphi}^{+}_{k,m}|^2 &\leq \frac{c_{11}}{(c_2)^2} r^{-3} ~.
\end{align}
Similarly, $\int_Y|\ck{\varphi}^{-}_{k,m}|^2$ has the same upper bound.

Proposition \ref{prop_Dirac_ck_01}(ii) implies that $\{\ck{\varphi}^{\appr}_{k,m}\,|\,(k,m)\in\ck{\CI}_j\}$ are mutually orthogonal to each other with respect to the $L^2$-inner product.  It together with (\ref{eqn_glue_02}) and (\ref{eqn_glue_03}) finds a constant $c_{12}$ so that
\begin{align}\label{eqn_glue_04}\begin{split}
\big| 1- \int_Y|\pr(\ck{\varphi}^{\appr}_{k,m})|^2 \big| &\leq c_{12}r^{-3} ~, \\
\big|\int_Y\langle\pr(\ck{\varphi}^{\appr}_{k,m}),\pr(\ck{\varphi}^{\appr}_{k',m'})\rangle\big|
&= \sum_{+,-}\big|\int_Y\langle\ck{\varphi}^{\pm}_{k,m},\ck{\varphi}^{\pm}_{k',m'}\rangle\big|
\leq c_{12}r^{-3}
\end{split}\end{align}
for any $(k,m),(k',m')\in\ck{\CI}_j$ and $(k,m)\neq(k',m')$.  On the other hand, the dimension of $\ck{\CV}_j$ is no greater than $c_7 r$ (\ref{eqn_glue_14}).  Based on these facts, a linear algebra argument shows that $\{\pr_j(\ck{\varphi}^{\appr}_{k,m})\,|\,(k,m)\in\ck{\CI}_j\}$ still forms a linearly independent set.  This proves Assertion (i) of the lemma.

\smallskip
(\emph{Assertion} (ii): $\ck{\CV}_j^\prp$)\;
It follows from the construction that $D_r(\CV_j)\subset\CV_j$.  Let $D_{r,j}^\prp:\ck{\CV}_j^\prp\to\ck{\CV}_j^\prp$ be the restriction of $D_r$ on $\ck{\CV}_j^\prp$ composing with the $L^2$-orthogonal projection onto $\ck{\CV}_j^\prp$.  The $L^2$ self-adjointness of $D_r$ implies the $L^2$ self-adjointness of $D_{r,j}^\prp$.  The construction of $D_{r,j}^\prp$ is done within $\CV_j$, and it is simply a finite-dimensional linear algebra.  Hence, there exists a $L^2$-orthonormal eigenbasis $\{\psi_{j,\ell}\}_\ell$ of $D_{r,j}^\prp$ on $\ck{\CV}_j^\prp$.  Denote the corresponding eigenvalue by $\mu_{j,\ell}$.  By (\ref{eqn_glue_02}), $\psi_{j,\ell}\in\ck{\CV}_j^\prp\subset\CV_j$ is not only $L^2$-orthogonal to $\pr(\ck{\varphi}^{\appr}_{k,m})$, but also $L^2$-orthogonal to $\ck{\varphi}^{\appr}_{k,m}$.

Let $\psi_{j,\ell}^{\err} = D_r\psi_{j,\ell} - \mu_{j,\ell}\psi_{j,\ell}$.  It follows from the construction that any $\psi_{j,\ell}^{\err}$ belongs to the $L^2$-orthogonal complement of $\ck{\CV}_j^\prp$ in $\CV_j$, which is $\pr_j(\ck{\CV}_j)$.  To show that $\nu_j+c_2r^{-\frac{3}{2}}<\mu_{j,\ell}<\nu_{j+1}-c_2r^{-\frac{3}{2}}$, express $\mu_{j,\ell}$ as $\int_Y\langle D_r\psi_{j,\ell},\psi_{j,\ell}\rangle$.  An elementary linear algebra argument shows that
\begin{align*}
\inf \CI_j \leq \sup_{\psi\in\CV_j\backslash\{0\}}\frac{\int_Y\langle D_r\psi,\psi\rangle}{\int_Y|\psi|^2} \leq \sup \CI_j ~,
\end{align*}
and the desired bound on $\mu_{j,\ell}$ follows from Lemma \ref{lem_glue_01}(ii).

\smallskip
It remains to estimate the $L^2$-norm of $\psi_{j,\ell}^{\err}$.  Suppose that $\sum_{\ck{\CI}_j}\fc_{k,m}\pr_j(\ck{\varphi}^{\appr}_{k,m})$ is a smooth section in $\pr_j(\ck{\CV}_j)$ with unit $L^2$-norm.  Due to (\ref{eqn_glue_01}), (\ref{eqn_glue_02}), (\ref{eqn_glue_03}), (\ref{eqn_glue_14}) and (\ref{eqn_glue_04}),
\begin{align*}
1 &= \sum_{\ck{\CI}_j}|\fc_{k,m}|^2\int_Y|\pr_j(\ck{\varphi}^{\appr}_{k,m})|^2 - \sum_{\scriptscriptstyle(k,m)\in\ck{\CI}_j}\sum_{(k',m')\atop\neq(k,m)}\fc_{k,m}\bar{\fc}_{k',m'}\int_Y\langle\pr_j(\ck{\varphi}^{\appr}_{k,m}),\pr_j(\ck{\varphi}^{\appr}_{k',m'})\rangle \\
&\geq \sum_{\ck{\CI}_j}|\fc_{k,m}|^2\int_Y|\pr_j(\ck{\varphi}^{\appr}_{k,m})|^2 - \oh\sum_{\scriptscriptstyle(k,m)\in\ck{\CI}_j}\sum_{(k',m')\atop\neq(k,m)}(|\fc_{k,m}|^2+|{\fc}_{k',m'}|^2)\Big|\int_Y\langle\pr_j(\ck{\varphi}^{\appr}_{k,m}),\pr_j(\ck{\varphi}^{\appr}_{k',m'})\rangle\Big| \\
&\geq \sum_{\ck{\CI}_j}|\fc_{k,m}|^2\int_Y|\pr_j(\ck{\varphi}^{\appr}_{k,m})|^2 - \sum_{\scriptscriptstyle(k,m)\in\ck{\CI}_j}\Big(|\fc_{k,m}|^2\sum_{(k',m')\atop\neq(k,m)}\big|\int_Y\langle\pr_j(\ck{\varphi}^{\appr}_{k,m}),\pr_j(\ck{\varphi}^{\appr}_{k',m'})\rangle\big|\Big) \\
&\geq (1 - c_{13}r^{-2})\sum_{\ck{\CI}_j}|\fc_{k,m}|^2
\end{align*}
for some constant $c_{13}>0$.  It follows that
\begin{align}\label{eqn_glue_05}
\sum_{\ck{\CI}_j}|\fc_{k,m}|^2 &\leq 1 + 2c_{13}r^{-2} ~.
\end{align}
Consider the following $L^2$-inner product:
\begin{align*}
\big|\int_Y\langle\psi^{\err}_{j,\ell},\sum_{\ck{\CI}_j}\fc_{k,m}\pr_j(\varphi^{\appr}_{k,m})\rangle\big| &\leq \sum_{\ck{\CI}_j}|\fc_{k,m}|\big|\int_Y\langle D_r\psi_{j,\ell} - \mu_{j,\ell}\psi_{j,\ell},\pr_j(\varphi^{\appr}_{k,m})\rangle\big| \\
&= \sum_{\ck{\CI}_j}|\fc_{k,m}|\big|\int_Y\langle D_r\psi_{j,\ell},\varphi^{\appr}_{k,m}\rangle\big| = \sum_{\ck{\CI}_j}|\fc_{k,m}|\big|\int_Y\langle \psi_{j,\ell},D_r\varphi^{\appr}_{k,m}\rangle\big| \\
&= \sum_{\ck{\CI}_j}|\fc_{k,m}|\big|\int_Y\langle \psi_{j,\ell},(D_r-\lambda_{k,m})\varphi^{\appr}_{k,m}\rangle\big| ~.
\end{align*}
By (\ref{eqn_glue_01}) and (\ref{eqn_glue_05}),
\begin{align*}
\big|\int_Y\langle\psi^{\err}_{j,\ell},\sum_{\ck{\CI}_j}\fc_{k,m}\pr_j(\varphi^{\appr}_{k,m})\rangle\big|^2 &\leq
\big(\sum_{\ck{\CI}_j}|\fc_{k,m}|^2\big) c_{11}r^{-6} \big(\sum_{\ck{\CI}_j}\int_{\rho\leq 1}|\pr_{k,m}(\psi_{j,\ell})|^2\big) \\
&\leq c_{11} r^{-6} (1+2c_{13}r^{-2}) \int_Y |\psi_{j,\ell}|^2 ~.
\end{align*}
Here, $\psi_{j,\ell}|_{\rho\leq1}$ is regarded as a local section of $\underline{\BC}\oplus\ck{K}^{-1}\to\ck{S}$ by Remark \ref{rmk_identification_01}, and $\pr_{k,m}$ is the projection onto $\CS_{k,m}$ defined in \S\ref{sec_ck_Dirac_component}.  Since the estimate holds for \emph{any} unit-normed section in $\pr_j(\ck{\CV}_j)$, we conclude that $\int_Y|\psi^{\err}_{j,\ell}|^2\leq 2c_{11}r^{-6}$.

\smallskip
(\emph{Assertion} (iii): $\pr^{-1}(\psi^{\err}_{j,\ell})$)\;
Proposition \ref{prop_Dirac_ck_01}(ii.b) and (\ref{eqn_glue_05}) imply that
$$   \int_Y|\pr_j^{-1}(\zeta)|^2\leq (1+2c_{13}r^{-2})\int_Y|\zeta|^2   $$
for any $\zeta\in\pr_j(\ck{\CV}_j)$.  Since $\int_Y |\zeta|^2 = \int_Y\langle\pr^{-1}(\zeta),\zeta\rangle$ for any $\zeta\in\pr_j(\ck{\CV}_j)$,
\begin{align*}
\int_Y|\pr_j^{-1}(\psi^{\err}_{j,\ell})-\psi^{\err}_{j,\ell}|^2 & = \int_Y|\pr_j^{-1}(\psi^{\err}_{j,\ell})|^2 - |\psi^{\err}_{j,\ell}|^2 \leq 2c_{13}r^{-2}\int_Y|\psi^{\err}_{j,\ell}|^2 \leq 4c_{11}c_{13}r^{-8} ~.
\end{align*}
This completes the proof of the lemma.
\end{proof}

\subsection{Almost vanishing near the bindings}\label{subsec_glue_02}
The main purpose of this subsection is to prove that the approximate eigensections constructed by Lemma \ref{lem_glue_02} have small $L^2$-integral near the bindings.

\begin{lem}\label{lem_glue_03}
There exists a constant $c_{15}$ determined by the contact form $a$, the metric $\dd s^2$ and the connection $\aE$ such that
\begin{align*}
\int_{\rho\leq\rho_0-5\delta}|\psi_{j,\ell}|^2 &\leq c_{15}(r^{-6} + r^{-1}\int_{\rho\leq\rho_0}|\psi_{j,\ell}|^2)
\end{align*}
for any $\psi_{j,\ell}$ produced by Lemma \ref{lem_glue_02}(ii) and $\rho_0\in[\oh,1]$.  ($\rho$ is the coordinate near the bindings as in \S\ref{subsec_open_book}.)
\end{lem}

\begin{proof}
Let $\chi(x)$ be a standard cut-off function with $\chi(x) = 1$ when $x\leq-1$, and $\chi(x) = 0$ when $x\geq0$.  Let $\chi_{\rho_0} = \chi((\rho-\rho_0)/\delta)$.  It can be regarded as a smooth function on $Y$ which is only non-zero near the bindings.

According to Remark \ref{rmk_identification_01} and (\ref{eqn_trivial_ck_01}), $\chi_{\rho_0}\psi_{j,\ell}$ may be regarded as a smooth section of $\underline{\BC}\oplus\ck{K}^{-1}\to\ck{S}$.  It follows from the construction of $(\ck{S},\ck{D}_r)$ that $D_r(\chi_{\rho_0}\psi_{j,\ell}) = \ck{D}_r(\chi_{\rho_0}\psi_{j,\ell})$ under the identification.

For brevity, drop the subscript $\ell$.  Express $\chi_{\rho_0}\psi_j$ in terms of the spectral decomposition induced by $\ck{D}_r$:
\begin{align}\label{eqn_glue_06}
\chi_{\rho_0}\psi_j &= \sum_{|\lambda_{k,m}|^2\leq\frac{r}{3}}\fc_{k,m}\ck{\varphi}_{k,m} + \ck{\zeta}_j^{\err} ~.
\end{align}
The first term is the $L^2$-orthogonal projection of $\chi_{\rho_0}\psi$ onto the space spanned by the eigensections whose eigenvalue $|\lambda_{k,m}|^2\leq\frac{1}{3}r$, and $\ck{\varphi}_{k,m}$'s are the eigensections given by Proposition \ref{prop_Dirac_ck_01}.  This same proposition guarantees that each $(k,m)$ appears at most once in the summation.  The remainder term $\ck{\zeta}_j^{\err}$ belongs to the $L^2$-orthogonal complement, which is spanned by eigensections whose eigenvalue $|\lambda|^2>\frac{1}{3}r$.

\smallskip
(\emph{Estimate $\ck{\zeta}^{\err}_j$})\;
Let $\mu_j$ be the approximate eigenvalue of $\psi_j$ given by Lemma \ref{lem_glue_02}(ii).  The operator $\ck{D}_r-\mu_j$ preserves the $L^2$-orthogonality between $\sum\fc_{k,m}\ck{\varphi}_{k,m}$ and $\ck{\zeta}^{\err}$.  Since $|\mu_j|<\frac{1}{2}r^\oh$, $|\mu_j \pm (\frac{1}{3}r)^\oh|\geq\frac{1}{16}r^\oh$.  It follows that
\begin{align*}
\frac{1}{256}r\int_{\ck{S}}|\ck{\zeta}^{\err}_j|^2 &\leq \int_{\ck{S}}|(\ck{D}_r-\mu_j)(\chi_{\rho_0}\psi_j)|^2
\leq 2\int_Y|(D_r-\mu_j)\psi_j|^2 + 2\int_Y |\cl(\dd\chi_{\rho_0})\psi_j|^2 ~,
\end{align*}
By Lemma \ref{lem_glue_02}(ii), there exists a constant $c_{16}$ such that
\begin{align}\label{eqn_glue_07}
\int_{\ck{S}}|\ck{\zeta}^{\err}_j|^2 &\leq c_{16}(r^{-7} + r^{-1}\int_{\rho\leq\rho_0}|\psi_j|^2) ~.
\end{align}

\smallskip
(\emph{Estimate $\fc_{k,m}$ for $(k,m)\in\ck{\CI}_j$ and $\ck{\rho}_{k,m}\leq\rho_0-3\delta$})\;
When $\ck{\rho}_{k,m}\leq\rho_0-3\delta$, Proposition \ref{prop_Dirac_ck_01}(ii.a) implies that $\chi_{\rho_0} = 1$ on the support of $\ck{\varphi}_{k,m}^{\appr}$, and thus $\chi_{\rho_0}\ck{\varphi}^{\appr}_{k,m} = \ck{\varphi}^{\appr}_{k,m}$.  The coefficient $\fc_{k,m}$ is equal to
\begin{align*}
\fc_{k,m} &= \int_{\ck{S}}\langle\chi_{\rho_0}\psi_j,\ck{\varphi}_{k,m}\rangle = \int_{Y}\langle\psi_j,\ck{\varphi}_{k,m}^{\appr}\rangle + \int_{\ck{S}}\langle\chi_{\rho_0}\psi_j,\ck{\varphi}_{k,m}^{\err}\rangle ~.
\end{align*}
Since $(k,m)\in\ck{\CI}_j$, it follows from Lemma \ref{lem_glue_02}(ii) that $\int_Y\langle\psi_j,\ck{\varphi}_{k,m}^{\appr}\rangle = 0$.  According to Proposition \ref{prop_Dirac_ck_01}(ii.b) and the Cauchy--Schwarz inequality, there exists a constant $c_{17}$ such that
\begin{align}\label{eqn_glue_08}
|\fc_{k,m}|^2 &\leq c_{17}r^{-7}\int_{\ck{S}}|\pr_{k,m}(\chi_{\rho_0}\psi_j)|^2
\end{align}
for any $(k,m)\in\ck{\CI}_j$ with $\ck{\rho}_{k,m}\leq\rho_0-3\delta$.  The map $\pr_{k,m}$ is the projection onto $\CS_{k,m}$ defined in \S\ref{sec_ck_Dirac_component}.

\smallskip
(\emph{Estimate $\fc_{k,m}$ for $(k,m)\notin\ck{\CI}_j$ and $\ck{\rho}_{k,m}\leq\rho_0-3\delta$})\;
Similar to the previous case, the coefficient $\fc_{k,m}$ is bounded by
\begin{align*}
|\fc_{k,m}|^2 &\leq 2\big|\int_Y\langle\psi_j,\ck{\varphi}^{\appr}_{k,m}\rangle\big|^2 + 2c_{17}r^{-7}\int_{\ck{S}}|\pr_{k,m}(\chi_{\rho_0}\psi_j)|^2 ~.
\end{align*}
Since $(k,m)\notin\ck{\CI}_j$, it follows from Lemma \ref{lem_glue_02}(ii) and Lemma \ref{lem_glue_01}(ii) that $|\mu_j-\lambda_{k,m}|\geq c_2r^{-1}$.  And then
\begin{align*}
c_2r^{-1}\big|\int_Y\langle\psi_j,\ck{\varphi}^{\appr}_{k,m}\rangle\big| &\leq \big|\int_Y\langle(D_r-\mu_j)\psi_j,\ck{\varphi}^{\appr}_{k,m}\rangle\big| + \big|\int_Y\langle\psi_j,(D_r-\lambda_{k,m})\ck{\varphi}^{\appr}_{k,m}\rangle\big|\\
&=\big|\int_Y\langle\psi^{\err}_j,\chi_{\rho_0}\ck{\varphi}^{\appr}_{k,m}\rangle\big| + \big|\int_Y\langle\psi_j,\chi_{\rho_0}(D_r-\lambda_{k,m})\ck{\varphi}^{\appr}_{k,m}\big| ~.
\end{align*}
Due to Proposition \ref{prop_Dirac_ck_01}(ii.c), there exists a constant $c_{18}$ such that
\begin{align}\label{eqn_glue_09}
|\fc_{k,m}|^2 &\leq c_{18}\big(r^2\int_{\ck{S}}|\pr_{k,m}(\chi_{\rho_0}\psi_j^{\err})|^2 + r^{-4}\int_{\ck{S}}|\pr_{k,m}(\chi_{\rho_0}\psi_j)|^2 \big)
\end{align}
for any $(k,m)\notin\ck{\CI}_j$ with $\ck{\rho}_{k,m}\leq\rho_0-3\delta$.

\smallskip
(\emph{Estimate $\fc_{k,m}$ for $\ck{\rho}_{k,m}>\rho_0-3\delta$})\;
It follows from the Cauchy--Schwarz inequality that
\begin{align}\label{eqn_glue_10}
|\fc_{k,m}|^2 &= \big|\int_{\ck{S}}\langle\chi_{\rho_0}\psi_j,\ck{\varphi}_{k,m}\rangle\big| \leq \int_{\ck{S}}|\pr_{k,m}(\chi_{\rho_0}\psi_j)|^2
\end{align}
for any $(k,m)$ with $\ck{\rho}_{k,m}>\rho_0-3\delta$.

\smallskip
(\emph{The integral of $\psi_j$ over $\rho\leq\rho_0-5\delta$})\;
Consider the $L^2$-integral of (\ref{eqn_glue_06}) on $\{\rho\leq\rho_0-5\delta\}$:
\begin{align}\label{eqn_glue_19}
\int_{\rho\leq\rho_0-5\delta}|\psi_j|^2 &\leq 2\int_{\ck{S}}|\ck{\zeta}^{\err}_j|^2 + 2\sum_{|\lambda_{k,m}|^2\leq\frac{r}{3}}|\fc_{k,m}|^2\big(\int_{\rho\leq\rho_0-5\delta}|\ck{\varphi}_{k,m}|^2\big)
\end{align}
where the inequality uses the fact that $\int_{\rho\leq\rho_0-5\delta}\langle\ck{\varphi}_{k,m},\ck{\varphi}_{k',m'}\rangle = 0$ for any $(k,m)\neq(k',m')$.  When $\ck{\rho}_{k,m}>\rho_0-3\delta$, Proposition \ref{prop_Dirac_ck_01}(ii.a) and (ii.b) imply that
\begin{align*}
\int_{\rho\leq\rho_0-5\delta}|\ck{\varphi}_{k,m}|^2 = \int_{\rho\leq\rho_0-5\delta}|\ck{\varphi}_{k,m}^{\err}|^2 \leq \int_{\ck{S}}|\ck{\varphi}^{\err}_{k,m}|^2 \leq c_{19}r^{-7} ~.
\end{align*}
By (\ref{eqn_glue_08}), (\ref{eqn_glue_09}) and (\ref{eqn_glue_10}), there exists a constant $c_{20}$ such that
\begin{align}
&\sum_{|\lambda_{k,m}|^2\leq\frac{r}{3}}|\fc_{k,m}|^2\big(\int_{\rho\leq\rho_0-5\delta}|\ck{\varphi}_{k,m}|^2\big) \notag\\
\leq&\, \sum_{\ck{\rho}_{k,m}\leq\rho_0-3\delta}|\fc_{k,m}|^2 + c_{19}r^{-6}\sum_{\ck{\rho}_{k,m}>\rho_0-3\delta}|\fc_{k,m}|^2 \notag\\
\leq&\, c_{20}\big( r^{-4}\int\sum_{(k,m)}|\pr_{k,m}(\chi_{\rho_0}\psi_j)|^2 + r^2\int\sum_{(k,m)\notin\ck{\CI}_j}|\pr_{k,m}(\chi_{\rho_0}\psi_j^{\err})|^2 \big) \notag\\
\leq&\, c_{20}\big( r^{-4}\int_{\rho\leq\rho_0}|\psi_j|^2 + r^2\int_{Y}\big|\chi_{\rho_0}\psi^{\err}_j - \chi_{\rho_0}\pr_j^{-1}(\psi_j^{\err})\big|^2 \big) ~. \label{eqn_glue_23}
\end{align}
where $\pr_j^{-1}$ is the map defined by Lemma \ref{lem_glue_02}(iii).  The last inequality follows from the fact that $\pr_j^{-1}(\psi_j)$ is contained in $\oplus_{(k,m)\in\ck{\CI}_j}\CS_{k,m}$, and so is $\chi_{\rho_0}\pr_j^{-1}(\psi_j)$.  According to  (\ref{eqn_glue_19}), (\ref{eqn_glue_07}), (\ref{eqn_glue_23}) and Lemma \ref{lem_glue_02}(iii),
\begin{align}
\int_{\rho\leq\rho_0-5\delta}|\psi_j|^2 &\leq c_{21}\big( r^{-6} + r^{-1}\int_{\rho\leq\rho_0}|\psi_j|^2 \big) ~.
\end{align}
This completes the proof of the lemma.
\end{proof}

By the iteration argument \cite{ref_Moser}, the $L^2$-norm of $\psi_{j,\ell}$ is rather small near the bindings.

\begin{cor}\label{cor_glue_01}
Suppose that $\chi_B(\rho)$ is a cut-off function with $\chi_B(\rho) = 1$ when $\rho\leq1-45\delta$ and $\chi_B(\rho)=0$ when $\rho\geq1-40\delta$.  There exists a constant $c_{23}$ determined by $a$, $\dd s^2$, $\aE$ and $\chi_B$ such that the following holds.  For any $r\geq c_{23}$,
\begin{enumerate}
\item the $L^2$-norm of $\psi_{j,\ell}$ near the bindings satisfies
\begin{align*}
\int_{\rho\leq1-30\delta}|\psi_{j,\ell}|^2 &\leq c_{23} r^{-6}
\end{align*}
for any $\psi_{j,\ell}$ produced by Lemma \ref{lem_glue_02}(ii);
\item for any $(k,m)\in\ck{\CI}_j$,
\begin{align*}
\big|\int_{\ck{S}}\langle\chi_B\psi_{j,\ell},\ck{\varphi}_{k,m}^{\appr}\rangle\big| &\leq c_{23}r^{-4}
\end{align*}
where $\ck{\varphi}_{k,m}^{\appr}$ is the approximate eigensection given by Proposition \ref{prop_Dirac_ck_01}(ii).  The section $\chi_B\psi_{j,\ell}$ is regarded as a section of $\underline{\BC}\oplus\ck{K}^{-1}\to\ck{S}$ by Remark \ref{rmk_identification_01} and (\ref{eqn_trivial_ck_01}).
\end{enumerate}
\end{cor}

\begin{proof}
Assertion (i) is a direct consequence of Lemma \ref{lem_glue_03}.

({Assertion} (ii))\;
When $\ck{\rho}_{k,m}>1-38\delta$, Proposition \ref{prop_Dirac_ck_01}(ii.a) implies that $\chi_B\ck{\varphi}_{k,m}=0$.  Thus, $\int_{\ck{S}}\langle\chi_B\psi_{j},\ck{\varphi}_{k,m}^{\appr}\rangle = 0$.

It remain to estimate the $L^2$-inner product when $\ck{\rho}_{k,m}\leq1-38\delta$.  Let $\rho_0=1-35\delta$, and let $\chi_{\rho_0}$ be the cut-off function introduced in the proof of Lemma \ref{lem_glue_03}.  Since $\chi_B = 1$ on the support of $\chi_{\rho_0}$, $\int_{\ck{S}}\langle\chi_B\psi_{j},\ck{\varphi}_{k,m}^{\appr}\rangle = \int_{\ck{S}}\langle\chi_{\rho_0}\psi_{j},\chi_B\ck{\varphi}_{k,m}^{\appr}\rangle$.

Express $\chi_{\rho_0}\psi_j$ in terms of the spectral decomposition induced by $\ck{D}_r$ as (\ref{eqn_glue_06}):
\begin{align*}
\chi_{\rho_0}\psi_j &= \sum_{|\lambda_{k',m'}|^2\leq\frac{r}{3}}\fc_{k',m'}\ck{\varphi}_{k',m'} + \ck{\zeta}_j ~.
\end{align*}
For any $(k',m')\in\ck{\CI}_j$ with $\ck{\rho}_{k',m'}\leq1-38\delta$, it follows from Assertion (i) and (\ref{eqn_glue_08}) that
$$    |\fc_{k',m'}|^2\leq c_{24}r^{-13} ~.    $$
By Assertion (i) and (\ref{eqn_glue_07}),
$$    \int_{\ck{S}}|\ck{\zeta}_j|^2\leq c_{25}r^{-7} ~.    $$
Similarly, express $\chi_B\ck{\varphi}^{\appr}_{k,m}$ in terms of the spectral decomposition induced by $\ck{D}_r$:
\begin{align*}
\chi_B\ck{\varphi}^{\appr}_{k,m} &= \fs\ck{\varphi}_{k,m} + \ck{\xi}_{k,m}
\end{align*}
where $\ck{\xi}_{k,m}\in\CS_{k,m}$ and is $L^2$-orthogonal to $\ck{\varphi}_{k,m}$.  According to Proposition \ref{prop_Dirac_ck_01}(ii.b) and (ii.c),
\begin{align*}
|\fs|\leq2 \qquad\text{ and }\qquad |\ck{\xi}_{k,m}|^2 \leq cr^{-1}\int|(\ck{D}_r - \lambda_{k,m})(\chi_B\ck{\varphi}_{k,m}^{\appr})|^2 \leq c_{26}r^{-1} ~.    \end{align*}

It follows that
\begin{align*}
\Big|\int_{\ck{S}}\langle\chi_{\rho_0}\psi_{j},\chi_B\ck{\varphi}_{k,m}^{\appr}\rangle\Big| &= \Big| \fc_{k,m}\fs + \int_{\ck{S}}\langle\ck{\zeta}_j,\ck{\xi}_{k,m}\rangle \Big| \leq c_{27}r^{-4} ~.
\end{align*}
This completes the proof of the corollary.
\end{proof}

\subsection{A linear algebra lemma}\label{subsec_glue_03}
The following lemma produces genuine eigenvalues of a Dirac operator from approximate eigenvalues.  It only involves linear algebra, and is a minor modification of \cite[Lemma 6.4]{ref_Ts1}.

\begin{lem}\label{lem_glue_04}
Let $\fD$ be a Dirac operator on a spinor bundle $\BS$.  If there are constants $\epsilon_1$ and $\epsilon_2$, a finite number of smooth sections $\{\xi_\ell\}_{\ell=1}^L$ of $\BS$, and real numbers $\{\mu_\ell\}_{\ell=1}^L$ satisfying the following properties:
\begin{enumerate}
\item $0<\epsilon_2<\frac{1}{4}$ and $\epsilon_2\big(\max\{\mu_\ell\}-\min\{\mu_\ell\}\big)^2\leq\epsilon_1$;
\item for any $1\leq\ell\leq L$, $\big|1-\int|\xi_\ell|^2\big|\leq\epsilon_2$, and $\sum_{\ell'=1,\ell'\neq\ell}^L\big|\int\langle\xi_{\ell'},\xi_{\ell}\rangle\big|\leq\epsilon_2$;
\item $\int|(\fD-\mu_\ell)\xi_\ell|^2\leq\epsilon_1$ for any $1\leq\ell\leq L$;
\item $\sum_{\ell'=1,\ell'\neq\ell}^L\big|\int\langle(\fD-\mu)\xi_{\ell'},(\fD-\mu)\xi_\ell\rangle\big|\leq\epsilon_1$ for any $1\leq\ell\leq L$ and $\min\{\mu_\ell\}\leq\mu\leq\max\{\mu_\ell\}$.
\end{enumerate}
Then, there exist $L$ eigenvalues (counting multiplicity) $\{\lambda_\ell\}_{\ell=1}^L$ of $\fD$ such that $|\lambda_\ell - \mu_\ell|\leq4\sqrt{\epsilon_1}$ for any $1\leq\ell\leq L$.
\end{lem}

\begin{proof}
The lemma is clearly true for $L=1$.  The plan is to do induction on the total number of approximate eigenvalues.  Suppose that the lemma is true for $L-1$ approximate eigenvalues.  Without loss of generality, assume that $\{\mu_\ell\}_{\ell=1}^L$ is non-decreasing in $\ell$.

For any $\ell\in\{1,2,\cdots,L\}$, remove $\xi_\ell$ and $\mu_\ell$, and apply the lemma.  If this procedure produces $L$ eigenvalues (counting multiplicity), it is done.  Now, suppose that there are only $(L-1)$ eigenvalues, $\{\lambda_\ell\}_{\ell=1}^{L-1}$.  They must satisfy
\begin{align*}
|\lambda_\ell - \mu_\ell| &\leq 4\sqrt{\epsilon_1} &&\text{and} &|\lambda_\ell - \mu_{\ell+1}| &\leq 4\sqrt{\epsilon_1}
\end{align*}
for any $\ell\in\{1,2,\cdots,L-1\}$.  It follows from the triangle inequalities that
\begin{align}\label{eqn_glue_11}
|\mu_\ell - \mu_{\ell+1}| \leq 8\sqrt{\epsilon_1}
\end{align}
for any $\ell\in\{1,2,\cdots,L-1\}$.

Property (ii) guarantees that $\{\xi_\ell\}_{\ell=1}^L$ forms a linear independence set.  Thus, there exist complex numbers $\{\fs_\ell\}_{\ell=1}^{L}$ such that $\sum_{\ell=1}^{L}\fs_\ell\xi_\ell$ is $L^2$-orthogonal to the corresponding eigensections of $\lambda_\ell$ for all $\ell\in\{1,2,\cdots,L-1\}$.  Normalize its $L^2$-norm to be $1$.  It follows that
\begin{align*}
1 &= \sum_{\ell=1}^L|\fs_\ell|^2\int|\xi_{\ell}|^2 + \sum_{\ell=1}^L\sum_{\ell'\neq\ell}\fs_\ell\bar{\fs}_{\ell'}\int\langle\xi_\ell,\xi_{\ell'}\rangle \\
&\geq (1-\epsilon_2)\sum_{\ell=1}^L|\fs_\ell|^2 - \oh\sum_{\ell=1}^L\sum_{\ell'\neq\ell}(|\fs_\ell|^2+|\fs_{\ell'}|^2)\big|\int\langle\xi_{\ell},\xi_{\ell'}\rangle\big| \\
&\geq (1-2\epsilon_2) \sum_{\ell=1}^L|\fs_\ell|^2 ~,
\end{align*}
and then $\sum_{\ell=1}^L|\fs_\ell|^2\leq 1+4\epsilon_2$.

For any real number $\mu$ with $\min\{\mu_\ell\}\leq\mu\leq\max\{\mu_\ell\}$,
\begin{align*}
\int\big|(\fD-\mu)(\xi_\ell)\big|^2 &= \int \big|(\fD-\mu_\ell)(\xi_\ell) - (\mu-\mu_\ell)\xi_\ell\big|^2 \\
&\leq \epsilon_1 + 2|\mu-\mu_\ell|\sqrt{\epsilon_1(1+\epsilon_2)} + |\mu-\mu_\ell|^2(1+\epsilon_2) \\
&\leq 2\epsilon_1 + 2\sqrt{2}|\mu-\mu_\ell|\sqrt{\epsilon_1} + |\mu-\mu_\ell|^2 = \big( \sqrt{2\epsilon_1} + |\mu-\mu_\ell| \big)^2 ~.
\end{align*}
It follows that
\begin{align*}
\int \big| (\fD-\mu)(\sum_{\ell=1}^{L}\fs_\ell\xi_\ell) \big|^2 &\leq \sum_{\ell=1}^L |\fs_\ell|^2 \int \big| (\fD-\mu)\xi_\ell \big|^2 + \sum_{\ell=1}^L|\fs_\ell|^2\sum_{\ell'\neq\ell}\big| \int\langle(\fD-\mu)\xi_\ell,(\fD-\mu)\xi_{\ell'}\rangle \big| \\
&\leq \sum_{\ell=1}^L |\fs_\ell|^2 \big( (\sqrt{2\epsilon_1} + |\mu-\mu_\ell|)^2 + \epsilon_1\big) \\
&\leq (1 + 4\epsilon_2)(\sqrt{3 \epsilon_1} + \max_{1\leq\ell\leq L} |\mu-\mu_\ell|)^2 \\
&\leq (4\sqrt{\epsilon_1} + \max_{1\leq\ell\leq L} |\mu-\mu_\ell|)^2~.
\end{align*}
By taking $\mu = \frac{\mu_1+\mu_L}{2}$, the inequality finds another eigenvalue $\lambda_L$ with
\begin{align*}
\mu_1 - 4\sqrt{\epsilon_1} \leq \lambda_L \leq \mu_L + 4\sqrt{\epsilon_1} ~,
\end{align*}
and its eigensection is $L^2$-orthogonal to the eigensections of $\{\lambda_\ell\}_{\ell=1}^{L-1}$.  According to (\ref{eqn_glue_11}), there exist some $\ell\in\{1,2,\cdots,L\}$ such that $|\lambda_L - \mu_\ell| \leq 4\sqrt{\epsilon_1}$.  After re-numbering the indices of $\{\lambda_\ell\}_{\ell=1}^L$ in the non-decreasing order, these $L$ eigenvalues satisfy the assertion of the lemma.
\end{proof}

\subsection{Gluing eigensections to $\ex{Y}$}\label{subsec_glue_04}
This step constructs eigenvalues of $(\ex{Y},\ex{D}_r)$ corresponding to $(\#\CI_j-\#\ck{\CI}_j)+\#\hat{\CI}_j$.

\begin{prop}\label{prop_glue_05}
There exists a constant $c_{30}>c_1$ determined by the contact form $a$, the metric $\dd s^2$ and the connection $\aE$ with the following significance.  For any $r\geq c_{30}$, let $\{\nu_j : -[\oh r^\oh]<j<[\oh r^\oh]\}$ be the sequence produced by Lemma \ref{lem_glue_01}.  Then,
\begin{align*}   (\#\CI_j-\#\ck{\CI}_j)+\#\hat{\CI}_j &\leq \#\ex{\CI}_j   \end{align*}
for any $-[\oh r^\oh]<j<[\oh r^\oh]-1$.
\end{prop}

\begin{proof}
The strategy is to construct approximate eigensections of $(\ex{Y},\ex{D}_r)$ corresponding to $\#\CI_j-\#\ck{\CI}_j$ and $\#\hat{\CI}_j$, and apply Lemma \ref{lem_glue_04} to estimate the eigenvalues of $(\ex{Y},\ex{D}_r)$.

\smallskip
(\emph{Step 1: approximate eigensections for $\#\CI_j-\#\ck{\CI}_j$})\;
For any $r\geq c_9$ and $-[\oh r^\oh]<j<[\oh r^\oh]-1$, consider the approximate eigensections $\{\psi_{j,\ell}\}_{1\leq\ell\leq\#\CI_j - \#\ck{\CI}_j}$ and the approximate eigenvalues $\{\mu_{j,\ell}\}_{1\leq\ell\leq\#\CI_j - \#\ck{\CI}_j}$ produced by Lemma \ref{lem_glue_02}(ii).  Let $\ex{\chi}(\rho)$ be a cut-off function with $\ex{\chi}(\rho) = 1$ when $\rho\geq1-40\delta$ and $\ex{\chi}(\rho) = 0$ when $\rho\leq1-45\delta$.  It can be regarded as a smooth function on both $Y$ and $\ex{Y}$, and it is equal to $1$ on $\Sigma\times_\tau S^1$.

The following recipe produces sections of $(\ex{E}\oplus\ex{E}_r\ex{K}^{-1})\otimes\ex{L}_r\to\ex{Y}$ from $\{\psi_{j,\ell}\}_{1\leq\ell\leq\#\CI_j - \#\ck{\CI}_j}$.
\begin{itemize}
\item Multiply $\{\psi_{j,\ell}\}_\ell$ by $\ex{\chi}$.  $\{\ex{\chi}\psi_{j,\ell}\}_\ell$ are still smooth sections of $E\oplus EK^{-1}\to Y$.
\item $\{\ex{\chi}\psi_{j,\ell}\}_\ell$ vanish on the tubular neighborhood $\{0\leq\rho\leq1-45\delta\}$ of the binding.  According to the constructions in \S\ref{subsec_sh_str} and \S\ref{subsec_deg_r}, $(\ex{E}\oplus\ex{E}_r\ex{K}^{-1})\otimes\ex{L}_r\to\supp(\ex{\chi})\subset\ex{Y}$ is isomorphic to $E\oplus EK^{-1}\to\supp(\ex{\chi})\subset Y$.  Thus, $\{\ex{\chi}\psi_{j,\ell}\}_\ell$ can be thought as as smooth sections of $(\ex{E}\oplus\ex{E}_r\ex{K}^{-1})\otimes\ex{L}_r\to\ex{Y}$.
\item According to the construction of $\ex{D}_r$ in \S\ref{subsec_ext_Lr}, $D_r(\ex{\chi}\psi_{j,\ell})$ is the same as $\ex{D}_r(\ex{\chi}\psi_{j,\ell})$ via this identification.
\end{itemize}
Take $\{\ex{\chi}\psi_{j,\ell}\}_{1\leq\ell\leq\#\CI_j-\#\ck{\CI}_j}$ to be the approximate eigensections corresponding to $\#\CI_j-\#\ck{\CI}_j$.  The approximate eigenvalues are $\{\mu_{j,\ell}\}_{1\leq\ell\leq\#\CI_j-\#\ck{\CI}_j}$.

\smallskip
(\emph{Step 2: approximate eigensections for $\#\hat{\CI}_j$})\;
For any $(k,n)\in\hat{\CI}_j$, consider the section
\begin{align*}  \hat{\varphi}_{k,n}^{\appr} &= \frac{1}{\sqrt{2\pi}}e^{ik\theta}(\hat{\alpha}_{k,n}^{\appr},0) \quad\text{ of }\hat{L}_r\oplus\hat{L}_r\hat{K}^{-1}\to\hat{S} \end{align*}
where $\hat{\alpha}_{k,n}^{\appr}$ is the approximate eigensection of $\hat{D}_r$ given by Proposition \ref{prop_Dirac_hat_01}, and the expression is with respect to the trivialization (\ref{eqn_hat_trivial_K}) on $\{0\leq\rho<2\}$.  Since $\hat{\rho}_{k,n}<1$ for any $(k,n)\in\hat{\CI}_j$, Proposition \ref{prop_Dirac_hat_01}(ii) implies that the support of $\hat{\varphi}_{k,n}^{\appr}$ is contained in $\{0\leq\rho\leq1+2\delta\}$.

With the help of the trivialization (\ref{eqn_hat_trivial_K}) and Remark \ref{rmk_identification_01}, $\hat{\varphi}_{k,n}^{\appr}$ can be regarded as a section of $(\ex{E}\oplus\ex{E}_r\ex{K}^{-1})\otimes\ex{L}_r\to\ex{Y}$.  It is not hard to see that $\hat{D}_r$ coincides with $\ex{D}_r$.  With this understood, take $\{\hat{\varphi}_{k,n}^{\appr}\}_{(k,n)\in\hat{\CI}_j}$ to be the approximate eigensections corresponding to $\#\hat{\CI}_j$.  The approximate eigenvalues are $\{\frac{r}{2}-\frac{k}{V}\}_{(k,n)\in\hat{\CI}_j}$.

\smallskip
(\emph{Step 3: condition} (ii) \emph{of Lemma \ref{lem_glue_04}})\;
We are going to apply Lemma \ref{lem_glue_04} on the approximate eigensections
\begin{align*} \big\{ \ex{\chi}\psi_{j,\ell} \big\}_{1\leq\ell\leq\#\CI_j-\#\ck{\CI}_j}\cup\big\{ \hat{\varphi}_{k,n}^{\appr}\big\}_{(k,n)\in\hat{\CI}_j}
\end{align*}
and approximate eigenvalues $\{\mu_{j,\ell}\}\cup\{\frac{r}{2}-\frac{k}{V}\}$.  The precise values of $\epsilon_1$ and $\epsilon_2$ will be chosen in step 6.

According to Lemma \ref{lem_glue_02}(ii), $\psi_{j,\ell}$ has unit $L^2$-norm.  By Corollary \ref{cor_glue_01}(i),
\begin{align}\label{eqn_glue_12}
\big| 1 - \int_{\ex{Y}}|\ex{\chi}\psi_{j,\ell}|^2 \big| = \int_Y(1-\ex{\chi}^2)|\psi_{j,\ell}|^2 \leq \int_{\rho\leq1-40\delta}|\psi_{j,\ell}|^2 \leq c_{31}r^{-6} ~.
\end{align}
Due to Proposition \ref{prop_Dirac_hat_01}(iv),
\begin{align}\label{eqn_glue_13}
\big| 1 - \int_{\ex{Y}}|\hat{\varphi}_{k,n}^{\appr}|^2 \big| = \int_{S^2}|\hat{\alpha}_{k,n}^{\err}|^2 \leq c_{32}r^{-7} ~.
\end{align}

There are four cases of the $L^2$-inner products between approximate eigensections.
\begin{itemize}
\item Lemma \ref{lem_glue_02}(ii) says that $\int_Y\langle\psi_{j,\ell},\psi_{j,\ell'}\rangle=0$ for any $\ell\neq\ell'$.  It follows that $\int_{\ex{Y}}\langle\ex{\chi}\psi_{j,\ell},\ex{\chi}\psi_{j,\ell'}\rangle = \int_{\ex{Y}}(\ex{\chi}^2-1)\langle\psi_{j,\ell},\psi_{j,\ell'}\rangle$.  By Corollary \ref{cor_glue_01}(i),
$$    \big|\int_{\ex{Y}}\langle\ex{\chi}\psi_{j,\ell},\ex{\chi}\psi_{j,\ell'}\rangle\big| \leq c_{31}r^{-6} ~.    $$
\item If $(1+47\delta)k<(n+[r])V$, the $L^2$-inner product
$$    \int_{\ex{Y}}\langle\hat{\varphi}_{k,n}^{\appr},\ex{\chi}\psi_{j,\ell}\rangle = 0 ~.    $$
The reason goes as follows.  The condition $(1+47\delta)k<(n+[r])V$ is equivalent to that $\hat{\rho}_{k,n}<1-47\delta$.  It follows from Proposition \ref{prop_Dirac_hat_01}(ii) that the support of $\hat{\varphi}_{k,n}^{\appr}$ is contained in where $\{\rho<1-45\delta\}$.  Hence, the supports of $\hat{\varphi}_{k,n}^{\appr}$ and $\ex{\chi}\psi_{j,\ell}$ are disjoint.
\item If $(1+47\delta)k\geq(n+[r])V>k$, Proposition \ref{prop_Dirac_hat_01}(v) identifies $\hat{\varphi}_{k,n}^{\appr}$ with $\fc_{k,n}\ck{\varphi}^{\appr}_{k,n+[r]}$ for some constant $\fc_{k,n}$ with $|1-\fc_{k,n}|\leq ce^{-\frac{c}{r}}$.  Since the approximate eigenvalue of $\hat{\psi}_{k,n}^{\appr}$ lies within $(\nu_j,\nu_{j+1})$, $(k,n+[r])$ must belong to $\ck{\CI}_j$.  According to Lemma \ref{lem_glue_02}(ii), $\int_{Y}\langle\ck{\varphi}_{k,n+[r]}^{\appr},\psi_{j,\ell}\rangle = 0$.  It follows that $\int_{\ex{Y}}\langle\hat{\varphi}_{k,n}^{\appr},\ex{\chi}\psi_{j,\ell}\rangle = \int_{\ex{Y}}\langle\hat{\varphi}_{k,n}^{\appr},(1-\ex{\chi})\psi_{j,\ell}\rangle$.  According to Corollary \ref{cor_glue_01}(ii) (with $\chi_B = 1-\ex{\chi}$),
$$   |\int_{\ex{Y}}\langle\hat{\varphi}_{k,n}^{\appr},\ex{\chi}\psi_{j,\ell}\rangle|\leq c_{33}r^{-4} ~.    $$
\item If $(k,n)\neq(k',n')$, $\hat{\varphi}_{k,n}^{\appr}$ and $\hat{\varphi}_{k',n'}^{\appr}$ have different Fourier frequencies, and must be $L^2$-orthogonal to each other. 	
\end{itemize}
The above estimates and (\ref{eqn_glue_14}) imply that for any $\ell\in\{1,\cdots,\#\CI_j-\#\ck{\CI}_j\}$,
\begin{align}\label{eqn_glue_15}
\sum_{\scriptstyle\ell'=1\atop\scriptstyle\ell'\neq\ell}^{\#\CI_j-\#\ck{\CI}_j}\big|\int_{\ex{Y}}\langle\ex{\chi}\psi_{j,\ell'},\ex{\chi}\psi_{j,\ell}\rangle\big| + \sum_{(k,n)\in\hat{\CI}_j}\big|\int_{\ex{Y}}\langle\hat{\varphi}_{k,n}^{\appr},\ex{\chi}\psi_{j,\ell}\rangle\big| &\leq c_{34}r^{-3} ~,
\end{align}
and for any $(k,n)\in\hat{\CI}_j$,
\begin{align}\label{eqn_glue_16}
\sum_{\ell=1}^{\#\CI_j-\#\ck{\CI}_j}\big|\int_{\ex{Y}}\langle\ex{\chi}\psi_{j,\ell},\hat{\varphi}_{k,n}^{\appr}\rangle\big| + \sum_{\scriptstyle(k',n')\in\hat{\CI}_j\atop\scriptstyle(k',n')\neq(k,n)}\big|\int_{\ex{Y}}\langle\hat{\varphi}_{k',n'}^{\appr},\hat{\varphi}_{k,n}^{\appr}\rangle\big| &\leq c_{34}r^{-\frac{5}{2}} ~.
\end{align}

\smallskip
(\emph{Step 4: condition} (iii) \emph{of Lemma \ref{lem_glue_04}})\;
For any $1\leq\ell\leq\#\CI_j-\#\ck{\CI}_j$, it follows from Lemma \ref{lem_glue_02}(ii) and Corollary \ref{cor_glue_01}(i) that
\begin{align}\label{eqn_glue_17}
\int_{\ex{Y}}|(\ex{D}_r - \mu_{j,\ell})(\ex{\chi}\psi_{j,\ell})|^2 &\leq 2\int_{Y}|(D_r - \mu_{j,\ell})\psi_{j,\ell}|^2 + 2\int_Y|\dd\ex{\chi}|^2|\psi_{j,\ell}|^2 \notag \\
&\leq c_{35}r^{-6} ~.
\end{align}
For any $(k,n)\in\hat{\CI}_j$, it follows from (\ref{eqn_hat_Dirac_01}) and Proposition \ref{prop_Dirac_hat_01}(iii) that
\begin{align}\label{eqn_glue_18}
\int_{\ex{Y}}|(\ex{D}_r - \frac{r}{2} + \frac{k}{V})(\hat{\varphi}_{k,n}^{\appr})|^2 &\leq c_{36}r^{-6} ~.
\end{align}

\smallskip
(\emph{Step 5: condition} (iv) \emph{of Lemma \ref{lem_glue_04}})\;
Lemma \ref{lem_glue_01}(i) implies that $\oh\leq\nu_{j+1}-\nu_j\leq2$ for any $-[\oh r^\oh]<j<[\oh r^\oh]-1$.  For any $\mu\in(\nu_j,\nu_{j+1})$, we write
\begin{align*}
(\ex{D}_r - \mu)(\ex{\chi}\psi_{j,\ell}) &= (\ex{D}_r - \mu_{j,\ell})(\ex{\chi}\psi_{j,\ell}) - (\mu - \mu_{j,\ell})(\ex{\chi}\psi_{j,\ell}) \\
&= \ex{\chi}\psi_{j,\ell}^{\err} + \cl(\dd\ex{\chi})\psi_{j,\ell} - (\mu - \mu_{j,\ell})(\ex{\chi}\psi_{j,\ell})  ~, \\
(\ex{D}_r - \mu)(\hat{\varphi}_{k,n}^{\appr}) &= (\ex{D}_r - \frac{r}{2} + \frac{k}{V})(\hat{\varphi}_{k,n}^{\appr}) - (\mu - \frac{r}{2} + \frac{k}{V})\hat{\varphi}_{k,n}^{\appr}
\end{align*}
where $\psi_{j,\ell}^{\err}$ is the error term given by Lemma \ref{lem_glue_02}(ii).

There are four cases of $L^2$-inner products between them.
\begin{itemize}
\item For any $\ell\neq\ell'$, Lemma \ref{lem_glue_02}(ii) says that $\int_{\ck{S}}\langle\psi_{j,\ell},\psi_{j,\ell'}^{\err}\rangle = 0$.  It follows that
$$    \int_{\ex{Y}}\langle\ex{\chi}\psi_{j,\ell},\ex{\chi}\psi_{j,\ell'}^{\err}\rangle = \int_{\ck{S}}(1-\ex{\chi}^2)\langle\psi_{j,\ell},\psi_{j,\ell'}^{\err}\rangle ~.    $$
By this trick, all the terms involving $\psi_{j,\ell}$ or $\psi_{j,\ell'}$ are integrated over $\{\rho\leq1-40\delta\}$.  According to Lemma \ref{lem_glue_02}(ii) and Corollary \ref{cor_glue_01},
\begin{align*}
\Big|\int_{\ex{Y}}\langle(\ex{D}_r - \mu)(\ex{\chi}\psi_{j,\ell}),(\ex{D}_r - \mu)(\ex{\chi}\psi_{j,\ell'})\rangle\Big| &\leq c_{37} r^{-6} ~.
\end{align*}
\item If $(1+47\delta)k < (n+[r])V$, the supports of $(\ex{D}_r - \mu)(\ex{\chi}\psi_{j,\ell})$ and $(\ex{D}_r - \mu)(\hat{\varphi}_{k,n}^{\appr})$ are disjoint to each other.  As in step 3, the $L^2$-inner product vanishes.
\item If $(1+47\delta)k \geq (n+[r])V > k$, Proposition \ref{prop_Dirac_hat_01}(iii) says that $\int_{\ex{Y}}|(\ex{D}_r - \frac{r}{2} + \frac{k}{V})(\hat{\varphi}_{k,n}^{\appr})|^2\leq ce^{-\frac{r}{c}}$.  After integration by parts, the $L^2$-inner product between $(\ex{D}_r - \mu)(\ex{\chi}\psi_{j,\ell})$ and $(\ex{D}_r - \mu)(\hat{\varphi}_{k,n}^{\appr})$ is equal to
\begin{align*}
&\int_{\ex{Y}}\langle(\ex{D}_r - \mu_{j,\ell})(\ex{\chi}\psi_{j,\ell}), (\ex{D}_r - \frac{r}{2} + \frac{k}{V})\hat{\varphi}_{k,n}^{\appr}\rangle + (\mu-\frac{r}{2}+\frac{k}{V})^2\int_{\ex{Y}}\langle\ex{\chi}\psi_{j,\ell},\hat{\varphi}_{k,n}^{\appr}\rangle\\
&\quad - (2\mu-\mu_{j,\ell}-\frac{r}{2}+\frac{k}{V})\int_{\ex{Y}}\langle\ex{\chi}\psi_{j,\ell},(\ex{D}_r - \frac{r}{2} + \frac{k}{V})\hat{\varphi}_{k,n}^{\appr}\rangle  ~.
\end{align*}
According to Lemma \ref{lem_glue_02}(ii) and the estimate in step 3,
\begin{align*}
\Big|\int_{\ex{Y}}\langle(\ex{D}_r - \mu)(\ex{\chi}\psi_{j,\ell}),(\ex{D}_r - \mu)(\hat{\varphi}_{k,n}^{\appr})\rangle\Big| &\leq c_{38} r^{-4} ~.
\end{align*}
\item If $(k,n)\neq(k',n')$, $(\ex{D}_r - \mu)(\hat{\varphi}_{k,n}^{\appr})$ and $(\ex{D}_r - \mu)(\hat{\varphi}_{k',n'}^{\appr})$ have different Fourier frequencies.  It follows that the $L^2$-inner product vanishes.
\end{itemize}
The above estimate and (\ref{eqn_glue_14}) imply that for any $\ell\in\{1,\cdots,\#\CI_j-\#\ck{\CI}_j\}$,
\begin{align}\label{eqn_glue_20}\begin{split}
&\sum_{\scriptstyle\ell'=1\atop\scriptstyle\ell'\neq\ell}^{\#\CI_j-\#\ck{\CI}_j}\big|\int_{\ex{Y}}\langle(\ex{D}_r - \mu)(\ex{\chi}\psi_{j,\ell'}),(\ex{D}_r - \mu)(\ex{\chi}\psi_{j,\ell})\rangle\big| \\
&\qquad + \sum_{(k,n)\in\hat{\CI}_j}\big|\int_{\ex{Y}}\langle(\ex{D}_r - \mu)(\hat{\varphi}_{k,n}^{\appr}),(\ex{D}_r - \mu)(\ex{\chi}\psi_{j,\ell})\rangle\big| \leq c_{39}r^{-3} ~,
\end{split}\end{align}
and for any $(k,n)\in\hat{\CI}_j$,
\begin{align}\label{eqn_glue_21}\begin{split}
&\sum_{\ell=1}^{\#\CI_j-\#\ck{\CI}_j}\big|\int_{\ex{Y}}\langle(\ex{D}_r - \mu)(\ex{\chi}\psi_{j,\ell}),(\ex{D}_r - \mu)(\hat{\varphi}_{k,n}^{\appr})\rangle\big| \\
&\qquad + \sum_{\scriptstyle(k',n')\in\hat{\CI}_j\atop\scriptstyle(k',n')\neq(k,n)}\big|\int_{\ex{Y}}\langle(\ex{D}_r - \mu)(\hat{\varphi}_{k',n'}^{\appr}),(\ex{D}_r - \mu)(\hat{\varphi}_{k,n}^{\appr})\rangle\big| \leq c_{39}r^{-\frac{5}{2}} ~.
\end{split}\end{align}

\smallskip
(\emph{Step 6: apply Lemma \ref{lem_glue_04}})\;
Let $c_{40} = \max\{c_{34},c_{39}\}$.  Set $\epsilon_1 = 4c_{40}r^{-\frac{5}{2}}$ and $\epsilon_2 = c_{40}r^{-\frac{5}{2}}$.  It follows from (\ref{eqn_glue_12}), (\ref{eqn_glue_13}), (\ref{eqn_glue_15}), (\ref{eqn_glue_16}), (\ref{eqn_glue_17}), (\ref{eqn_glue_18}), (\ref{eqn_glue_20}) and (\ref{eqn_glue_21}) that the approximate eigensections and eigenvalues satisfy the conditions of Lemma \ref{lem_glue_04}.  Therefore, $(\ex{Y},\ex{D}_r)$ admits $(\#\CI_j-\#\ck{\CI}_j)+\#\hat{\CI}_j$ eigenvalues within
$$    (\nu_j - 8\sqrt{c_{40}}r^{-\frac{5}{4}}, \nu_{j+1} + 8\sqrt{c_{40}}r^{-\frac{5}{4}}) ~.    $$
With the help of Lemma \ref{lem_glue_01}(ii), these eigenvalues actually belongs to $(\nu_j + c_2r^{-1}, \nu_{j+1} - c_2r^{-1})$.  It follows that $(\#\CI_j-\#\ck{\CI}_j)+\#\hat{\CI}_j \leq \#\ex{\CI}_j$.  This completes the proof of Proposition \ref{prop_glue_05}.
\end{proof}

\subsection{Proof of Theorem \ref{thm_glue_01}}
With Proposition \ref{prop_glue_05}, it remains to show that
\begin{align}\label{eqn_glue_22}    (\#\ex{\CI}_j-\#\hat{\CI}_j)+\#\ck{\CI}_j \leq \#{\CI}_j ~.    \end{align}
Note that Proposition \ref{prop_glue_05} also implies that the spectral gap of $(Y,D_r)$ in Lemma \ref{lem_glue_01}(iii) is in fact of order $r^{-1}$.  With this understood, (\ref{eqn_glue_22}) can be proved by the same gluing construction.  Since the argument is completely parallel, we will not duplicate it.

\section{Proof of Theorem \ref{thm_main_01}}\label{sec_main_thm}
The purpose of this section is to prove Theorem \ref{thm_main_01}.  With the help of Theorem \ref{thm_glue_01}, the spectral asymmetry of $(Y,D_r)$ can be related to the spectral asymmetries of $(\ex{Y},\ex{D}_r)$, $(\ck{S},\ck{D}_r)$ and $(\hat{S},\hat{D}_r)$.  Since we understand the spectrum of these three models pretty well, their spectral asymmetries are more tractable.

\begin{defn}
For any $r$ greater than the constants of Proposition \ref{prop_Dirac_ck_01} and Proposition \ref{prop_Dirac_hat_01}, introduce the following set of orthonormal eigensections.
\begin{enumerate}
\item Let $\ex{\CV}_r$ be the set of orthonormal eigensections of $\ex{D}_r$ whose eigenvalue belongs to $(-\frac{1}{3}r^\oh,\frac{1}{3}r^\oh)$.  $\ex{\CV}_r^\pm$ is the subset of $\ex{\CV}_r$ in which the corresponding eigenvalue is positive (negative, respectively).
\item Let $\ck{\CV}_r$ be the set of $\ck{\varphi}_{k,m}$ (given by Proposition \ref{prop_Dirac_ck_01}) with
$$ k<mV $$
and whose eigenvalue of $\ck{D}_r$ belongs to $(-\frac{1}{3}r^\oh,\frac{1}{3}r^\oh)$.  $\ck{\CV}_r^\pm$ is the subset of $\ck{\CV}_r$ in which the corresponding eigenvalue is positive (negative, respectively).
\item Let $\hat{\CV}_r$ be the set of $\hat{\varphi}_{k,n}$ (given by Proposition \ref{prop_Dirac_hat_00} and Proposition \ref{prop_Dirac_hat_01} and) with
\begin{align*}
|\frac{r}{2}-\frac{k}{V}|<\frac{1}{3}r^\oh  \quad~,\qquad  \frac{k}{V}-[r]<n\leq0 ~.
\end{align*}
$\hat{\CV}_r^\pm$ is the subset of $\hat{\CV}_r$ in which $\frac{r}{2}-\frac{k}{V}$ is positive (negative, respectively).
\end{enumerate}
\end{defn}

\begin{defn}
With these sets of eigensections, let $\dot{\eta}(\ex{\CV}_r)$ be (\ref{eqn_eta_00}) with the index set replaced by $\ex{\CV}_r^+$ and $\ex{\CV}_r^-$, and $\lambda_\psi$ replaced by the corresponding eigenvalue of $\ex{D}_r$.  The functions $\ddot{\eta}(\ex{\CV}_r)$, $\dot{\eta}(\ck{\CV}_r)$, $\ddot{\eta}(\ck{\CV}_r)$, $\dot{\eta}(\hat{\CV}_r)$ and $\ddot{\eta}(\hat{\CV}_r)$ are defined in the same way.
\end{defn}

\subsection{Compare the eta functions}
The first step is to estimate the difference between $\dot{\eta}(r)$, $\ddot{\eta}(r)$ and the corresponding eta-functions of the models.

\begin{lem}\label{lem_final_01}
There exists a constant $c_1$ determined by $a$, $\dd s^2$ and $A_E$ such that
\begin{align*}
\big|\, \dot{\eta}(r) - \dot{\eta}(\ex{\CV}_r) - \dot{\eta}(\ck{\CV}_r) + \dot{\eta}(\hat{\CV}_r) \,\big| &\leq c_1 r \qquad\text{and} \\
\big|\, \ddot{\eta}(r) - \ddot{\eta}(\ex{\CV}_r) - \ddot{\eta}(\ck{\CV}_r) + \ddot{\eta}(\hat{\CV}_r) \,\big| &\leq c_1 \log r
\end{align*}
for any $r\geq c_1$
\end{lem}

\begin{proof}
Let $c_2$ be the constant given by Theorem \ref{thm_glue_01}.  For any $r\geq c_2$, let $\{\nu_j~:~-[\oh r^\oh]<j<[\oh r^\oh]\}$ be the sequence given by Lemma \ref{lem_glue_01}.  The corresponding index sets defined by Definition \ref{defn_glue_01} satisfy
\begin{align}\label{eqn_final_01}
\#\CI_j + \#\hat{\CI}_j = \#\ex{\CI}_j + \#\ck{\CI}_j \leq c_3r
\end{align}
for any $r\geq c_2$ and $-[\oh r^\oh]<j<[\oh r^\oh]-1$. Let
\begin{align*}
J^- &= \min\{j~|~\nu_j>-\frac{1}{3}r^\oh\}  ~,  &J^+ &= \max\{j~|~\nu_j<\frac{1}{3}r^\oh\} ~, \\
J^-_o &= \max\{j~|~\nu_j<0\} \qquad\text{and}  &J^+_o &= \min\{j~|~\nu_j>0\} ~.
\end{align*}

\smallskip
(\emph{Estimate $\ddot{\eta}(r)$})\;
The function $\ddot{\eta}(r)$ can be written as
\begin{align*}  r^{-\frac{3}{2}}\log r\Big(\sum_{j=J_-}^{J^-_o}\sum_{\lambda_\psi\in(\nu_{j-1},\nu_j)}(\lambda_\psi e^{-20(r^{-1}\log r)\lambda_\psi^2}) + \sum_{j=J_o^+}^{J^+}\sum_{\lambda_\psi\in(\nu_j,\nu_{j+1})}(\lambda_\psi e^{-20(r^{-1}\log r)\lambda_\psi^2}) \Big)  \end{align*}
up to an error term $c_4 r^{-\oh}\log r$, which governs the contribution near $\lambda=0$ and near $\lambda=\pm\frac{1}{3}r^\oh$.  By the mean value theorem,
\begin{align}\label{eqn_final_03}
\big| xe^{-20(r^{-1}\log r)x^2} - ye^{-20(r^{-1}\log r)y^2} \big| &\leq |x-y|
\end{align}
for any non-negative $x$ and $y$.  With (\ref{eqn_final_01}) and (\ref{eqn_final_03}),
\begin{align*}
\big|\ddot{\eta}(r) + \ddot{\eta}(\hat{\CV}_r) - \ddot{\eta}(\ex{\CV}_r) - \ddot{\eta}(\ck{\CV}_r)\big|
&\leq  r^{-\frac{3}{2}}\log r \, \big( c_4 r + 2(c_3r)(\frac{1}{3}r^\oh) \big) \leq c_5 \log r ~.
\end{align*}

\smallskip
(\emph{Estimate $\dot{\eta}(r)$})\;
By the same cancellation trick,
\begin{align*}
&\,\big|\dot{\eta}(r) + \dot{\eta}(\hat{\CV}_r) - \dot{\eta}(\ex{\CV}_r) - \dot{\eta}(\ck{\CV}_r)\big| \\
\leq &\, c_6 r^{-\oh}(\log r)^\oh \, \big( r^\frac{3}{2}(\log r)^{-\oh} + \sum_{j=J^-}^{J^-_o} e^{-20(r^{-1}\log r)\nu_j^2} + \sum_{j=J_o^+}^{J^+} e^{-20(r^{-1}\log r)\nu_j^2} \big) \\
\leq &\, c_7 r^{-\oh}(\log r)^\oh \, \big( r^\frac{3}{2}(\log r)^{-\oh} + \int_0^\infty e^{-20(r^{-1}\log r)s^2}\dd s \big)
\leq c_8 r ~.
\end{align*}
This completes the proof of the lemma.
\end{proof}

\subsection{Eta functions of the mapping torus}
The distribution of small eigenvalues of $(\ex{Y},\ex{D}_r)$ is described by Theorem \ref{thm_Vafa_Witten}.  It leads to the following lemma.

\begin{lem}\label{lem_final_02}
There exists a constant $c_{10}$ determined by $\ex{a}$, $\ex{\omega}$, $\dd s^2$ and $\aeE$ such that
\begin{align*}
|\,\dot{\eta}(\ex{\CV}_r)\,| &\leq c_{10} r(\log r)^\oh   &\text{and}&   &|\,\ddot{\eta}(\ex{\CV}_r)\,| &\leq c_{10} \log r
\end{align*}
for any $r\geq c_{10}$.
\end{lem}

\begin{proof}
Let $\{\lambda_j\}_{j\in\BZ}$ be the spectrum of $\ex{D}_r$, which is arranged in ascending order and is normalized so that $\lambda_1$ is the smallest non-negative eigenvalue.  Theorem \ref{thm_Vafa_Witten} finds constants $c_{11}>0$, $c_{12}$ and $c_{13}$ such that
\begin{align}\label{eqn_final_02}
\big|\lambda_{j+c_{11}r+c_{12}} - \lambda_j - \frac{1}{V}\big| &\leq c_{13} r^{-\oh}
\end{align}
for any $r\geq c_{13}$ and $|\lambda_j|\leq\oh r^\oh$.  Denote $c_{11} r+ c_{12}$ by $\fj$.

Note that $\lambda_1\leq \frac{2}{V}$.  This can be seen by gluing a solution from Proposition \ref{prop_Dirac_hat_00}.  It follows from (\ref{eqn_final_02}) that $\lambda_0\geq-\frac{1}{V}$.  Let
\begin{align*}
\fk = \max\{k\in\BN ~|~ \lambda_{1+k\fj}<\frac{1}{3}r^\oh \text{ and }\lambda_{-k\fj}>-\frac{1}{3}r^\oh \} ~.
\end{align*}
The estimate (\ref{eqn_final_02}) also implies that
\begin{align}
\big| \fk - \frac{V}{3}r^\oh \big| &\leq c_{14} ~,
&\lambda_{1+\fk\fj} &\geq \frac{1}{3}r^{\oh} - c_{14} &\text{and}&
&\lambda_{-\fk\fj} &\leq -\frac{1}{3}r^\oh + c_{14}
\end{align}
for some constant $c_{14}$.

\smallskip
(\emph{Estimate $\ddot{\eta}(\ex{\CV}_r)$})\;
With this understood, $\ddot{\eta}(\ex{\CV}_r)$ is equal to
\begin{align*}
&r^{-\frac{3}{2}}\log r \, \sum_{k=1}^{\fk} \sum_{j=1}^{\fj}\Big( \lambda_{j+\fj(k-1)}\exp\big({-20(r^{-1}\log r)\lambda_{j+\fj(k-1)}^2}\big) \\ 
&\qquad\qquad\qquad\quad + \lambda_{1-j-\fj(k-1)}\exp\big({-20(r^{-1}\log r)\lambda_{1-j-\fj(k-1)}^2}\big) \Big)
\end{align*}
up to an error term of $c_{15}r^{-2}\log r$.  According to (\ref{eqn_final_02}),
\begin{align*}
(\frac{1}{V} - c_{13}r^{-\oh})(k-1)\leq \lambda_{1+\fj(k-1)}\leq &\lambda_{j+\fj(k-1)} \leq \lambda_{1+\fj k} \leq \frac{2}{V} + (\frac{1}{V} + c_{13}r^{-\oh})k \\
-\frac{1}{V} - (\frac{1}{V} + c_{13}r^{-\oh})k \leq \lambda_{-\fj k}\leq &\lambda_{1-j-\fj(k-1)} \leq \lambda_{-\fj(k-1)} \leq -(\frac{1}{V} - c_{13}r^{-\oh})(k-1)
\end{align*}
for any $k\in\{1,\cdots,\fk\}$ and $j\in\{1,\cdots,\fj\}$.  It follows from (\ref{eqn_final_03}) that
\begin{align*}
|\,\ddot{\eta}(\ex{\CV}_r)\,| &\leq c_{15}\big( r^{-2}\log r + r^{-\frac{3}{2}}\log r \, \sum_{k=1}^{\fk} \sum_{j=1}^{\fj}(1+r^{-\oh}k) \big) \\
&\leq c_{16}\log r  ~.
\end{align*}
The latter inequality uses the facts that $\fj = c_{11} r+ c_{12}$ and $\fk\leq\frac{V}{3}r^\oh + c_{14}$.

\smallskip
(\emph{Estimate $\dot{\eta}(\ex{\CV}_r)$})\;
By the same token,
\begin{align*}
|\dot{\eta}(\ex{\CV}_r)| &\leq c_{17}\big( r + r^{-\oh}(\log r)^\oh \,  \sum_{k=1}^{\fk} \sum_{j=1}^{\fj}(1+r^{-\oh}k) \big) \\
&\leq c_{18} r(\log r)^\oh  ~.
\end{align*}
This completes the proof of the lemma.
\end{proof}

\subsection{Eta functions of the local model for the open book}
Proposition \ref{prop_Dirac_ck_01} gives approximations for small eigenvalues of $(\ck{S},\ck{D}_r)$.  It leads to the following lemma.

\begin{lem}\label{lem_final_03}
There exists a constant $c_{20}$ determined by determined by $\ck{a}$ and $\dd s^2$ such that
\begin{align*}
|\,\dot{\eta}(\ck{\CV}_r)\,| &\leq c_{20} r(\log r)^\oh   &\text{and}&   &|\,\ddot{\eta}(\ex{\CV}_r)\,| &\leq c_{20}\log r
\end{align*}
for any $r\geq c_{20}$.
\end{lem}

\begin{proof}
(\emph{Estimate $\ddot{\eta}(\ck{\CV}_r)$})\;
Due to Proposition \ref{prop_Dirac_ck_01} and (\ref{eqn_final_03}),
\begin{align}\label{eqn_final_04}
|\ddot{\eta}(\ck{\CV}_r)| &\leq \oh (r^{-\frac{3}{2}}\log r) \Big|\sum_{\ck{\varphi}_{k,m}\in\ck{\CV}_r}(r-\ck{\gamma}_{k,m})\exp\big(-{5}r^{-1}\log r(r-\ck{\gamma}_{k,m})^2\big)\Big| + c_{21}\log r ~.
\end{align}
To proceed, note that the $\ck{\gamma}_{k,m}$ has a naturally extension as a smooth function on the right half plane of $(k,m)$.  Consider the re-parametrized polar coordinate on the right half plane:
\begin{align*}
k(s,\rho) &= s\big(\ck{f}^2(\rho) + \ck{g}^2(\rho)\big)^{-\oh}\ck{f}(\rho) ~,
&m(s,\rho) &= s\big(\ck{f}^2(\rho) + \ck{g}^2(\rho)\big)^{-\oh}\ck{g}(\rho)
\end{align*}
where $s\geq0$ and $0\leq\rho\leq 2$.  Note that
$$\dd k\wedge\dd m = \frac{\ck{f}'\ck{g}-\ck{f}\ck{g}'}{\ck{f}^2 + \ck{g}^2}s\dd\rho\wedge\dd s ~.$$
Let
\begin{align*}
\ck{\gamma}(s,\rho) = 2s(\ck{f}^2(\rho) + \ck{g}^2(\rho))^{-\oh} ~.
\end{align*}
A straightforward computation shows that $\ck{\gamma}$ has the following properties:
\begin{itemize}
\item if $(k(s,\rho),m(s,\rho))$ belongs to the lattice $\BZ_{\geq0}\times\BZ$, $\ck{\gamma}(s,\rho)$ is equal to $\ck{\gamma}_{k,m}$ defined by Definition \ref{defn_Dirac_ck_01};
\item for any $k + |m|\geq 1$, $\ck{\gamma}(s,\rho)$ is a smooth function in $(k,m)$;  there exists a constant $c_{19}$ such that $$ |\ck{\gamma}(k,m)-\ck{\gamma}(k',m') |\leq c_{22} $$ provided $|k'-k|+|m'-m|\leq 1$.
\end{itemize}

Let $\rho_V\in(0,2)$ be the unique solution to $V\ck{g}(\rho)-\ck{f}(\rho) = 0$.  It follows from the above discussion and Proposition \ref{prop_Dirac_hat_00} that
\begin{align*}
&\,\Big| \sum_{\ck{\varphi}_{k,m}\in\ck{\CV}_r}(r-\ck{\gamma}_{k,m})\exp(-{5}r^{-1}\log r(r-\ck{\gamma}_{k,m})^2) \\
&\,\quad - \big(\int_0^{\rho_V}\int_{s_-}^{s_+} (r-\ck{\gamma}(s,\rho))\exp(-{5}r^{-1}\log r(r-\ck{\gamma}(s,\rho))^2) \frac{\ck{f}'\ck{g}-\ck{f}\ck{g}'}{\ck{f}^2 + \ck{g}^2}s\dd s\dd\rho \big) \Big| \leq c_{23}r^{\frac{3}{2}}
\end{align*}
where $s_- = \frac{\sqrt{\ck{f}^2+\ck{g}^2}}{2}(r-\frac{2}{3}r^\oh)$ and $s_+ = \frac{\sqrt{\ck{f}^2+\ck{g}^2}}{2}(r+\frac{2}{3}r^\oh)$.  In other words, the summation is equal to the integration up to an error term of order $c_{23}r^{\frac{3}{2}}$.  After integration by parts,
\begin{align*}
&\, \int_{s_-}^{s_+} (r-\ck{\gamma}(s,\rho))\exp(-{5}r^{-1}\log r(r-\ck{\gamma}(s,\rho))^2) s\dd s \\
=&\, \frac{1}{4}(\ck{f}^2+\ck{g}^2)^{-\oh} r (\log r)^{-1} \Big( s\exp\big(-{5}r^{-1}\log r(r-\ck{\gamma}(s,\rho))^2\big)\big|_{s_-}^{s_+} \Big) \\
&\,\quad + \frac{1}{4}(\ck{f}^2+\ck{g}^2)^{-\oh} r (\log r)^{-1}\int_{s_-}^{s_+}\exp\big(-{5}r^{-1}\log r(r-\ck{\gamma}(s,\rho))^2\big) \dd s \\
\leq&\, c_{24}r^{\frac{3}{2}}(\log r)^{-1} ~.
\end{align*}
It follows that
\begin{align*}
\big| \sum_{\ck{\varphi}_{k,m}\in\ck{\CV}_r}(r-\ck{\gamma}_{k,m})\exp(-{5}r^{-1}\log r(r-\ck{\gamma}_{k,m})^2) \big| &< c_{25}r^{\frac{3}{2}} ~.
\end{align*}
By combining this inequality with (\ref{eqn_final_04}), it implies that
\begin{align*}
|\ddot{\eta}(\ck{\CV}_r)| &\leq c_{26}\log r ~.
\end{align*}

\smallskip
(\emph{Estimate $\dot{\eta}(\ck{\CV}_r)$})\;
By the same token, $\dot{\eta}(r)$ is equal to
\begin{align*}
&r^{-\oh}\log r\Big( \int_0^{\rho_V}\int_{s_-}^{s_0}\int_{\oh(r-\ck{\gamma}(s,\rho))}^{\frac{1}{3}r^\oh}e^{-20(r^{-1}\log r)u^2}\frac{\ck{f}'\ck{g}-\ck{f}\ck{g}'}{\ck{f}^2 + \ck{g}^2}s\dd u \dd s\dd\rho \\
&\qquad\qquad - \int_0^{\rho_V}\int_{s_0}^{s_+}\int^{\oh(r-\ck{\gamma}(s,\rho))}_{-\frac{1}{3}r^\oh}e^{-20(r^{-1}\log r)u^2}\frac{\ck{f}'\ck{g}-\ck{f}\ck{g}'}{\ck{f}^2 + \ck{g}^2}s\dd u \dd s\dd\rho \Big)
\end{align*}
up to an error term $c_{27} r(\log r)^\oh$.  Here, $s_0 = \frac{\sqrt{\ck{f}^2+\ck{g}^2}}{2}r$, $s_- = \frac{\sqrt{\ck{f}^2+\ck{g}^2}}{2}(r-\frac{2}{3}r^\oh)$ and $s_+ = \frac{\sqrt{\ck{f}^2+\ck{g}^2}}{2}(r+\frac{2}{3}r^\oh)$.  By changing the order of integration,
\begin{align*}
&\, \int_{s_-}^{s_0}\int_{\oh(r-\ck{\gamma}(s,\rho))}^{\frac{1}{3}r^\oh}e^{-20(r^{-1}\log r)u^2}s\dd u \dd s - \int_{s_0}^{s_+}\int^{\oh(r-\ck{\gamma}(s,\rho))}_{-\frac{1}{3}r^\oh}e^{-20(r^{-1}\log r)u^2}s\dd u \dd s \\
=&\, \int_{0}^{\frac{1}{3}r^\oh}\int_{s_0 - \sqrt{\ck{f}^2+\ck{g}^2}u}^{s_0}e^{-20(r^{-1}\log r)u^2}s\dd s \dd u - \int_{0}^{\frac{1}{3}r^\oh}\int^{s_0 + \sqrt{\ck{f}^2+\ck{g}^2}u}_{s_0}e^{-20(r^{-1}\log r)u^2}s\dd s \dd u \\
=&\, -\frac{\ck{f}^2+\ck{g}^2}{2}\int_0^{\frac{1}{3}r^\oh} u^2 e^{-20(r^{-1}\log r)u^2} \dd u
\geq -c_{28}r^{\frac{3}{2}}(\log r)^{-\frac{3}{2}} ~.
\end{align*}
It follows that
\begin{align*}
|\dot{\eta}(\ck{\CV}_r)| &\leq c_{29}r(\log r)^\oh ~.
\end{align*}
This finishes the proof of the lemma.
\end{proof}

\subsection{Eta functions of the local model for the mapping torus}
The small eigenvalues of $(\hat{S},\hat{D}_r)$ are characterized by Proposition \ref{prop_Dirac_hat_00}.  It leads to the following lemma.

\begin{lem}\label{lem_final_04}
There exists a constant $c_{31}$ determined by $\hat{a}$, $\hat{\omega}$ and $\dd s^2$ such that
\begin{align*}
|\,\dot{\eta}(\hat{\CV}_r)\,| &\leq c_{31}r   &\text{and}&   &|\,\ddot{\eta}(\hat{\CV}_r)\,| &\leq c_{31}\log r
\end{align*}
for any $r\geq c_{31}$.
\end{lem}

\begin{proof}
According to Proposition \ref{prop_Dirac_hat_00} and Proposition \ref{prop_Dirac_hat_01}, the eigenvalues of $\hat{D}_r$ on $\hat{\CV}_r$ are
$$ \big\{ \frac{r}{2} - \frac{k}{V} ~\big|~ k\in\BZ\text{ and } |\frac{r}{2} - \frac{k}{V}|<\frac{1}{3}r^\oh \big\} ~, $$
and the multiplicity of $\frac{r}{2} - \frac{k}{V}$ is $[r]-[\frac{k}{V}]$.

In particular, the smallest non-negative eigenvalue happens at $k=[\frac{rV}{2}]$.  Rewrite positive eigenvalues as $\lambda_j^+ = \frac{r}{2} - \frac{1}{V}([\frac{rV}{2}] - j)$ for $j\geq0$, and its multiplicity is $n_j^+ = [r] - [\frac{1}{V}[\frac{rV}{2}] - \frac{j}{V}]$.  Similarly, the negative eigenvalues are $\lambda_j^- = \frac{r}{2} - \frac{1}{V}([\frac{rV}{2}] + j + 1)$ for $j\geq0$, and the multiplicity is $n_j^- = [r] - [\frac{1}{V}[\frac{rV}{2}] + \frac{j + 1}{V}]$.

With the help of (\ref{eqn_final_03}), a cancellation argument shows that
\begin{align*}
\big|\ddot{\eta}(\hat{\CV}_r)\big| &\leq c_{32}(r^{-\frac{3}{2}}\log r) \sum_{j=0}^{[\frac{V}{3}r^\oh]}\big( n_j^- + (n_j^+ - n_j^-)\lambda_j^+ \big)  \\
&\leq c_{33}(r^{-\frac{3}{2}}\log r)  \sum_{j=0}^{[\frac{V}{3}r^\oh]} \Big( \big([r] - [\frac{1}{V}[\frac{rV}{2}] + \frac{j + 1}{V}]\big) + (j+1)\big(\frac{r}{2} - \frac{1}{V}([\frac{rV}{2}] - j)\big) \Big)  \\
&\leq c_{34}\log r ~.
\end{align*}
By the same cancellation argument,
\begin{align*}
\big|\dot{\eta}(\hat{\CV}_r)\big| &\leq c_{35}r^{-\oh}(\log r)^{\oh}\sum_{j=0}^{[\frac{V}{3}r^\oh]} \big( n_j^- + (n_j^+ - n_j^-)\int_0^\infty e^{-20(r^{-1}\log r)u^2}\dd u \big) \\
&\leq c_{36}r^{-\oh}(\log r)^{\oh}\sum_{j=0}^{[\frac{V}{3}r^\oh]} \Big( \big([r] - [\frac{1}{V}[\frac{rV}{2}] + \frac{j + 1}{V}]\big) + r^{\oh}(\log r)^{-\oh}(j+1) \Big) \\
&\leq c_{37}r ~.
\end{align*}
This completes the proof of the lemma.
\end{proof}

\subsection{Proof of Theorem \ref{thm_main_01}}
It follows from the triangle inequality on Lemma \ref{lem_final_01}, Lemma \ref{lem_final_02}, Lemma \ref{lem_final_03} and Lemma \ref{lem_final_04} that
\begin{align*}
|\,\dot{\eta}(r)\,| &\leq c_{40}r(\log r)^\oh   &\text{and}&   &|\,\ddot{\eta}(r)\,| &\leq c_{40}\log r ~.
\end{align*}
This completes the proof of Theorem \ref{thm_main_01}.

\appendix
\section{Dirac Equation on the Local Models}

\subsection{The local model for the open book}\label{subsec_ap1}
Proposition \ref{prop_Dirac_ck_01} was proved in \cite[\S5]{ref_Ts1}.  The purpose of this subsection is to give precise reference for each assertion.

\smallskip
(\emph{Assertion} (i))\;
With respect to the trivialization (\ref{eqn_trivial_ck_01}), the Dirac operator is written down explicitly in \cite[(4.1)]{ref_Ts1}.  The expression is invariant under the $S^1\times S^1$-action in $\theta$ and $t$, and Assertion (i) follows.

\smallskip
(\emph{Assertion} (ii) \emph{on the eigenvalues})\;
The existence, uniqueness and estimate of the eigenvalue $\lambda$ on $\CS_{k,m}$ was proved in \cite[Lemma 5.5]{ref_Ts1}.  Note that it implies that $\ck{\gamma}_{k,m}$ is of the same order as $r$, i.e.\ $\frac{1}{c_1}r\leq\ck{\gamma}_{k,m}\leq c_1 r$ for some constant $c_1>0$.

\smallskip
(\emph{Assertion} (ii) \emph{on the eigensections})\;
When $\ck{\rho}_{k,m}\leq20\delta$ or $\ck{\rho}_{k,m}\geq2-20\delta$, Assertion (ii.a), (ii.b) and (ii.c) were proved in \cite[Proposition 5.10]{ref_Ts1}.

When $20\delta<\ck{\rho}_{k,m}<2-20\delta$, Assertion (ii.a), (ii.b) and (ii.c) were basically proved in \cite[Proposition 5.6]{ref_Ts1}, with a larger error term.  We now explain how to construct the approximate eigensection with the error term claimed by Proposition \ref{prop_Dirac_ck_01}.  The strategy was mentioned in \cite[the paragraph after (5.24)]{ref_Ts1}.

Let $\psi = (\ck{\alpha},\ck{\beta})\in\CS_{k,m}$.  Use separation of variables to write $\ck{\alpha}$ and $\ck{\beta}$ as
\begin{align*}
\ck{\alpha} &= \ck{\alpha}_{k,m}(\rho)e^{i(k\theta+mt)}(2\pi\ck{\Delta})^{-\oh}  &\text{and}&
&\ck{\beta} &= \ck{\beta}_{k,m}(\rho)e^{i((k+1)\theta+t)}(2\pi\ck{\Delta})^{-\oh}
\end{align*}
where $\ck{\Delta}$ is $\oh(\ck{f}'\ck{g}-\ck{f}\ck{g}')$.  The eigensections equation \cite[(5.23)]{ref_Ts1} reads
\begin{align}\label{eqn_a101}\left\{\begin{aligned}
\big(\frac{r}{2}+\frac{k\ck{g}'-m\ck{f}'}{2\ck{\Delta}}\big)\ck{\alpha}_{k,m} - \ck{\beta}'_{k,m} - \big( \frac{k\ck{g}-m\ck{f}}{\ck{\Delta}} + \frac{\ck{\Delta}'}{2\ck{\Delta}}\big)\ck{\beta}_{k,m} &= \lambda\ck{\alpha}_{k,m} ~, \\
\ck{\alpha}'_{k,m} - \big( \frac{k\ck{g}-m\ck{f}}{\ck{\Delta}} + \frac{\ck{\Delta}'}{2\ck{\Delta}}\big)\ck{\alpha}_{k,m} - \big(\frac{r}{2}+\frac{k\ck{g}'-m\ck{f}'}{2\ck{\Delta}} + 1 + \frac{\ck{f}''\ck{g}'-\ck{f}'\ck{g}''}{8\ck{\Delta}}\big)\ck{\beta}_{k,m} &= \lambda\ck{\beta}_{k,m}  ~.
\end{aligned}\right.\end{align}

Consider the Taylor series expansion of the coefficients at $\ck{\rho}_{k,m}$:
\begin{align*}
-\frac{k\ck{g}-m\ck{f}}{\ck{\Delta}} &= \ck{\gamma}_{k,m}x - \sum_{j=2}^7\fr_j x^j - \fR_0(x) x^8 ~, \\
\frac{k\ck{g}'-m\ck{f}'}{2\ck{\Delta}} &= -\frac{\ck{\gamma}_{k,m}}{2} - \sum_{j=2}^7\fr'_j x^j - \fR_1(x) x^8 ~, \\
\frac{\ck{\Delta}'}{2\ck{\Delta}} &= \sum_{j=0}^5 \fe_j x^j + \fE_0(x) x^6 ~, \\
1+\frac{\ck{f}''\ck{g}'-\ck{f}'\ck{g}''}{8\ck{\Delta}} &= \sum_{j=0}^5 \fe'_j x^j + \fE_1(x) x^6
\end{align*}
where $x = \rho - \ck{\rho}_{k,m}$ and the above equations\footnote{These constants and function depend on $k$ and $m$.  We drop the subscripts $k$ and $m$ for simplicity.} holds for $|x|\leq4\delta$.  A direct computation shows that the linear term of $(k\ck{g}'-m\ck{f}')/{2\ck{\Delta}}$ vanishes.  Since $20\delta<\ck{\rho}_{k,m}<2-20\delta$, there exists a constant $c_2>0$ such that
\begin{align}\label{eqn_a105}
|\fe_j|+|\fe'_j| &\leq c_2 ~, &\text{and}&  &|\fE_0(x)| + |\fE_1(x)| &\leq c_2
\end{align} for any $|x|\leq4\delta$.
According to (\ref{eqn_Dirac_ck_01}) and $20\delta<\ck{\rho}_{k,m}<2-20\delta$, there exists a constant $c_3>0$ so that $k + |m|\leq c_3 r$.  It follows that there exists a constant $c_4>0$ such that
\begin{align}\label{eqn_a106}
|\fr_j|+|\fr'_j| &\leq c_4 r ~, &\text{and}&  &|\fR_0(x)| + |\fR_1(x)| &\leq c_4 r
\end{align} for any $|x|\leq4\delta$.

Let $\xi(x)$ be $(\frac{\ck{\gamma}_{k,m}}{\pi})^{\frac{1}{4}}\exp(-\frac{\ck{\gamma}_{k,m}}{2}x^2)$.  Consider the approximate solution for (\ref{eqn_a101})
\begin{align}\label{eqn_a102}\left\{\begin{aligned}
a(x) &= \big(1 + a_1(x) + a_2(x) + \cdots + a_6(x)\big)\xi(x) ~, \\
b(x) &= \big( b_1(x) + b_2(x) + \cdots + b_6(x) \big)\xi(x) ~, \\
\mu &= \frac{r}{2} - \frac{\ck{\gamma}_{k,m}}{2} + \mu_1 + \mu_2 + \cdots + \mu_6
\end{aligned}\right.\end{align}
defined by the following \emph{recursive formulae}:
\begin{align}
-\frac{\dd b_j(x)}{\dd x} + (2\ck{\gamma}_{k,m}x)b_j(x) &= \sum_{i=1}^{j-1} (\fe_{i-1}x^{i-1} + \fr_{i+1}x^{i+1})b_{j-i}(x) \notag \\
&\qquad + \sum_{i=1}^j(\mu_{i} + \fr'_{i+1}x^{i+1})a_{j-i}(x) \label{eqn_a103} ~, \\
\frac{\dd a_j(x)}{\dd x} &= \sum_{i=1}^{j} (\fe_{i-1}x^{i-1} + \fr_{i+1}x^{i+1})a_{j-i}(x) \notag \\
&\qquad + \sum_{i=1}^{j-1}(\mu_{i} + \fe'_{i-1}x^{i-1} - \fr'_{i+1}x^{i+1})b_{j-i}(x) \label{eqn_a104} ~.
\end{align}
To say more,
\begin{itemize}
\item $a_j(x)$, $b_j(x)$ and $\mu_j$ are determined by $\{a_i(x),b_i(x),\mu_i\}_{i<j}$;  the initial term $a_0(x)$ is set to be $1$, $b_0(x)$ is set to be $0$, and $\mu_0$ is set to be $(r-\ck{\gamma}_{k,m})/2$;
\item for any $j>0$, solve (\ref{eqn_a104}) for $a_j(x)$ with the initial condition $a_j(0)=0$;
\item the image of $(-\frac{\dd}{\dd x} + 2\ck{\gamma}_{k,m}x)$ is $L^2$-orthogonal to $\exp(-\ck{\gamma}_{k,m}x^2)$;  $\mu_j$ is determined by the condition that the right hand side of (\ref{eqn_a103}) is $L^2$-orthogonal to $\exp(-\ck{\gamma}_{k,m}x^2)$;  then solve (\ref{eqn_a103}) for $b_j(x)$.
\end{itemize}
With (\ref{eqn_a105}) and (\ref{eqn_a106}), an induction argument shows that
\begin{itemize}
\item $\{a_j(x), b_j(x)\}_{j=1}^6$ are polynomials in $x$ with $\deg(a_j(x))\leq 3j$ and $\deg(b_j(x))\leq 3j-2$; their coefficients in front of $x^i$ obeys
\begin{align}
\Big|\frac{\dd^i a_j(x)}{\dd x^i}|_{x=0}\Big| + \Big|\frac{\dd^i b_j(x)}{\dd x^i}|_{x=0}\Big| \leq (i!)c_6 r^{\frac{i-j}{2}} ~; \label{eqn_a107}
\end{align}
\item $\mu_j$ obeys $|\mu_j|\leq c_6r^{\frac{1-j}{2}}$
\end{itemize}
for some constant $c_6$.

Let $\chi(x)$ be the cut-off function with $\chi(x) = 1$ when $|x|\leq\delta$ and $\chi(x) = 0$ when $|x|\geq2\delta$.  Consider the section $\psi^{\appr} = \chi(\rho-\rho_{k,m})\big(a(\rho-\rho_{k,m}),b(\rho-\rho_{k,m})e^{ik\theta}\big)e^{i(k\theta+mt)}(2\pi\ck{\Delta})^{-\oh}\in\CS_{k,m}$.  It has the following significance.
\begin{itemize}
\item Note that $\int_{\BR}|x^i\xi(x)|^2\dd x\leq ((i+3)!)(\ck{\gamma}_{k,m})^{-i}$ for any non-negative integer $i$.  It follows from (\ref{eqn_a107}) that
\begin{align}\label{eqn_a108}
\big| 1 - \int_{\ck{S}}|\psi^{\appr}|^2 \big| = \big| 1 - \int_{\BR}(\chi(x))^2(|a(x)|^2 + |b(x)|^2)\dd x \big| \leq c_7r^{-1} ~.
\end{align}
\item By plugging $a(x)$ and $b(x)$ into (\ref{eqn_a101}), the recursive relations (\ref{eqn_a103}) and (\ref{eqn_a104}) together with (\ref{eqn_a105}) and (\ref{eqn_a106}) imply that
\begin{align}\label{eqn_a109}
\int_{\ck{S}}\big|\ck{D}_r\psi^{\appr} - \mu\psi^{\appr}\big|^2 &\leq c_7r^{-6} ~.
\end{align}
\end{itemize}

Let $\ck{\varphi}_{k,m}$ be the $L^2$-orthogonal projection of $\psi^{\appr}$ onto the eigenspace of $\lambda$, and $\ck{\varphi}_{k,m}^{\err}$ be $\ck{\varphi}_{k,m}-\psi^{\appr}$.  Due to (\ref{eqn_a109}) and (\ref{eqn_a108}), $|\lambda - \mu|\leq c_8 r^{-3}$.  According to \cite[Lemma 5.5]{ref_Ts1}, any eigenvalue of $D_r$ on $\CS_{k,m}$ other than $\lambda $ has magnitude greater than $(\frac{r}{2})^{\oh}$.  It follows that
\begin{align*}
\int_{\ck{S}}|\ck{\varphi}_{k,m}^{\err}|^2 &\leq c_8r^{-1}\int_{\ck{S}}|\ck{D}_r\psi^{\appr}-\mu\psi^{\appr}|^2 \leq c_7c_8 r^{-7} ~,
\end{align*}
and then
\begin{align*}
\big| 1 - \int_{\ck{S}}|\ck{\varphi}_{k,m}|^2 \big| &\leq c_9r^{-1} ~.
\end{align*}
After normalizing the $L^2$-norm of $\ck{\varphi}_{k,m} = \psi^{\appr} + \ck{\varphi}_{k,m}^{\err}$, they satisfy Assertion (ii.a), (ii.b) and (ii.c) of the proposition.

\smallskip
(\emph{Assertion} (iii))\;
When $\ck{\rho}_{k,m}\leq20\delta$ or $\ck{\rho}_{k,m}\geq2-20\delta$, Assertion (iii) were proved in \cite[Lemma 5.12]{ref_Ts1}.  When $20\delta<\ck{\rho}_{k,m}<2-20\delta$, (\ref{eqn_Dirac_ck_01}) implies that $k + |m|\leq c_{3} r$.  In particular, the estimates on the coefficients (\ref{eqn_a106}) still hold, and so do (\ref{eqn_a108}) and (\ref{eqn_a109}).  Hence, Assertion (iii) follows.  This finishes the proof of Proposition \ref{prop_Dirac_ck_01}.

\smallskip
(\emph{Remark on the case when $|1-\ck{\rho}_{k,m}|\leq48\delta$})\;
On the region where $|\rho-1|\leq50\delta$, $\ck{f} = V$ and $\ck{g} =  2-\rho$.  By Definition \ref{defn_Dirac_ck_01}, $\ck{\rho}_{k,m} = 2-\frac{mV}{k}$ and then $|1-\frac{mV}{k}|\leq48\delta$.  According to (\ref{eqn_Dirac_ck_01}), $\ck{\gamma}_{k,m} = \frac{2k}{V}$.  The eigensection equation (\ref{eqn_a101}) on the region where $|\rho-1|\leq50\delta$ reads
\begin{align*}\left\{\begin{aligned}
(\frac{r}{2} - \frac{\ck{\gamma}_{k,m}}{2})\ck{\alpha}_{k,m} - \ck{\beta}_{k,m}' + (\ck{\gamma}_{k,m}x)\ck{\beta}_{k,m} &= \lambda\ck{\alpha}_{k,m} ~, \\
\ck{\alpha}_{k,m}' + (\ck{\gamma}_{k,m}x)\ck{\alpha}_{k,m} - (\frac{r}{2} - \frac{\ck{\gamma}_{k,m}}{2} + 1)\ck{\beta}_{k,m} &= \lambda\ck{\beta}_{k,m} ~.
\end{aligned}\right.\end{align*}
It follows that $\ck{\alpha}_{k,m} = \xi(x) = (\frac{\ck{\gamma}_{k,m}}{\pi})^{\frac{1}{4}}\exp(-\frac{\ck{\gamma}_{k,m}}{2}x^2)$, $\ck{\beta}_{k,m} = 0$ and $\lambda = (r-\ck{\gamma}_{k,m})/2$ is a true solution on the region where $|\rho-1|\leq50\delta$.  It is easy to see that the upper bound in (\ref{eqn_a108}) and (\ref{eqn_a109}) is actually $c_{10}e^{-\frac{r}{c_{10}}}$ in this case.  Thus,
\begin{align}\label{eqn_a110}
\ck{\varphi}_{k,m}^{\appr} &= \fc_{k,m}\big( \chi(\rho-\ck{\rho}_{k,m})\xi(\rho-\rho_{k,m})e^{i(k\theta+mt)}\pi^{-\oh}, 0 \big)
\end{align}
for some constant $\fc_{k,m}$ with $|1-\fc_{k,m}|\leq c_{11}e^{-\frac{r}{c_{11}}}$.

\subsection{The local model for the mapping torus}\label{subsec_ap2}
The purpose of this section is to prove Proposition \ref{prop_Dirac_hat_01}.

By (\ref{eqn_hat_Dirac_00}) and (\ref{eqn_hat_Dirac_01}), the Cauchy--Riemann operator $\dbr_{r,k}$ is
\begin{align*}
 \pl_\rho - \frac{i}{\hat{h}_\sigma'}\pl_t + \frac{1}{\hat{h}'_\sigma}\big( r(\hat{h}_\sigma + \oh\hat{g}) - \frac{k}{V}{\hat{g}} \big)
\end{align*}
with respect to the trivialization $\bo_-$ (\ref{eqn_hat_trivial_K}).  The equation is invariant under the $S^1$-action in $e^{it}$.  Consider the section $\hat{\alpha}_{k,n}(\rho)e^{ikt}$.  The Cauchy--Riemann equation reads
\begin{align}\label{eqn_a201}
\pl_\rho\hat{\alpha}_{k,n} + \frac{1}{\hat{h}'_\sigma}\big( n + r(\hat{h}_\sigma + \oh\hat{g}) - \frac{k}{V}{\hat{g}} \big)\hat{\alpha}_{k,n} &=0 ~.
\end{align}

Suppose that $20\delta<\hat{\rho}_{k,n}<2-20\delta$.  The equation (\ref{eqn_a201}) can be solved by the integral factor:
\begin{align}\label{eqn_a202}
\hat{\alpha}_{k,n}(\rho) &= \hat{\fc}_{k,n}\exp\big(-\int_{\hat{\rho}_{k,n}}^{\rho}\frac{1}{\hat{h}'_\sigma(s)}( n + r(\hat{h}_\sigma(s) + \oh\hat{g}(s)) - \frac{k}{V}{\hat{g}(s)})\dd s \big) ~.
\end{align}
The constant $\hat{\fc}_{k,n}$ is chosen so that $\int_0^2|\hat{\alpha}_{k,n}(\rho)|^2\,\hat{h}'_\sigma\dd\rho = 1$.  Based on the properties of $\hat{h}_\sigma$ and $\hat{g}$ explained in Definition \ref{defn_Dirac_hat_01} and with the condition $|\frac{r}{2}-\frac{k}{V}|\leq(\frac{1}{3}r)^\oh$, there exists a constant $c_{12}$ such that $|\hat{\alpha}_{k,n}(\rho)|\leq c_{12}e^{-\frac{r}{c_{12}}}$ for any $|\rho-\hat{\rho}_{k,n}|>\delta$.  Consider $\psi^{\appr}_{k,n} = \chi(\rho-\hat{\rho}_{k,n})\hat{\alpha}_{k,n}(\rho)$ where $\chi$ is cut-off function as in \S\ref{subsec_ap1}.  It is not hard to see that
\begin{align*}
\hat{\alpha}_{k,n}^{\appr} &= \big(\int_0^2\langle\psi_{k,n}^{\appr},\hat{\alpha}_{k,n}\rangle\,\hat{h}'_\sigma\dd\rho\big)^{-1} \psi^{\appr}_{k,n} &\text{ and }&    &\hat{\alpha}_{k,n}^{\err} &= \hat{\alpha}_{k,n} - \hat{\alpha}_{k,n}^{\appr}
\end{align*}
satisfy Assertions (i), (ii), (iii) and (iv) of the proposition.  When $\hat{\rho}_{k,n}\leq20\delta$ or $\hat{\rho}_{k,n}\geq2-20\delta$, the construction of the approximate solution is essentially the same as that for \cite[Proposition 5.10]{ref_Ts1}, and we omit it here.

For Assertion (v), note that $\hat{\rho}_{k,n} = 2 - \frac{(n+[r])V}{k}$ when $|n-\frac{k}{V}+[r]|<\frac{k}{V}(48\delta)$.  It follows that the integrand of (\ref{eqn_a202}) is equal to $\frac{2k}{V}(\rho-\hat{\rho}_{k,n})$ when $|\rho-\hat{\rho}_{k,n}|\leq2\delta$.  Assertion (v) follows from a straightforward computation and (\ref{eqn_a110}).  This finishes the proof of Proposition \ref{prop_Dirac_hat_01}.


\end{document}